\documentclass[a4paper,11pt]{article}
\usepackage[letterpaper,margin=0.80in,bottom=0.90in,top=0.85in]{geometry}

\usepackage[english]{babel}

\usepackage[utf8]{inputenc}
\usepackage{amsmath,mathrsfs}
\usepackage{indentfirst}
\usepackage{fancyhdr}
\usepackage{amssymb}
\usepackage{amsthm}
\usepackage[dvips]{graphicx}
\usepackage{color}
\usepackage{cancel}
\usepackage{subcaption}
\usepackage[dvipsnames]{xcolor}
\usepackage{eucal}
\usepackage{latexsym}
\usepackage{chngcntr}
\usepackage{faktor}
\usepackage{microtype}
\usepackage{newlfont}
\usepackage{enumitem}
\usepackage{hyperref,geometry,mathtools} 
\usepackage{todonotes}

\usepackage{amsmath, amsfonts, amsthm,amssymb,  bbm}

\usepackage{tikz-cd}
\usepackage{multicol}
\usepackage{cancel}
\usepackage{mathtools}
\usepackage{graphicx}

\newcommand{\white}[1]{{\textcolor{white}{#1}}}

\newcommand{\cV}{{\mathcal V}}

\newcommand{\tf}{{\mathtt f}}

\newcommand{\uphi}{\underline{\varphi}}

\newcommand{\umu}{\underline{\mu}}

\newcommand{\Res}[2]{\mathtt{R}\begingroup 
\setlength\arraycolsep{0pt}{\scriptsize \begin{matrix} #2 \\[0.7mm] #1 \end{matrix} }\endgroup}

\newcommand{\ent}[6]{\begingroup 
\setlength\arraycolsep{-2pt}\begin{matrix}{\tB_{#1}^{[#2]}} &\white{{|}^{|}}_{#3,#4}^{#5,#6}\end{matrix}\endgroup}

\newcommand{\piuinf}{\scalebox{0.7}{$(+\infty)$}\;}

\newcommand{\betone}[3]{[b_{#1}^{(\tp)}]_{#2}^{#3}}
\newcommand{\betonep}[3]{\big[b_{#1}^{(\tp)}\big]_{#2}^{#3}}
\newcommand{\ac}[1]{\mathtt{n}_{#1}^{(\tp)}}

\newcommand{\bc}[1]{\mathtt{b}_{#1}^{(\tp)}}
\newcommand{\hbc}[1]{\widehat{\;\mathtt{b}}_{#1}^{(\tp)}\!}

\newcommand{\e}{\epsilon}
\newcommand{\ttf}{\mathtt{f}}
\newcommand{\im}{\mathrm{i}\,}

\newcommand\restr[2]{{
  \left.\kern-\nulldelimiterspace 
  #1 
  \vphantom{\big|} 
  \right|_{#2} 
  }}

\allowdisplaybreaks
\newcommand{\ch}{{\mathtt c}_{\mathtt h}}

\theoremstyle{plain}
\newtheorem{lem}{Lemma}
\newtheorem{teo}[lem]{Theorem}
\newtheorem{prop}[lem]{Proposition}
\newtheorem{ass}[lem]{Assumption}

\newtheorem{clem}[lem]{Crucial Lemma}

\theoremstyle{definition}

\newtheorem{sia}[lem]{Definition}

\newtheorem{rmk}[lem]{Remark}
\newtheorem{ntt}[lem]{Notation}

\newcommand{\nequiv}{{\; \scriptsize \begin{matrix} \cancel{\equiv}\end{matrix}\; }}
\renewcommand{\bar}{\overline}

\newcommand{\sgn}{\mathrm{sgn}}

\newcommand{\vet}[2]{\begin{bmatrix}#1 \\ #2 \end{bmatrix}}
\newcommand{\uno}{\mathrm{Id}}

\newcommand{\bR}{\mathbb{R}}
\newcommand{\bT}{\mathbb{T}}
\newcommand{\bZ}{\mathbb{Z}}
\newcommand{\bN}{\mathbb{N}}
\newcommand{\bQ}{\mathbb{Q}}
\newcommand{\bC}{\mathbb{C}}

\newcommand{\cL}{{\cal L}}
\newcommand{\cO}{\mathcal{O}}

\newcommand{\cR}{\mathcal{R}}
\newcommand{\cJ}{\mathcal{J}}
\newcommand{\cB}{{\cal B}}

\newcommand{\tB}{\mathtt{B}}

\newcommand{\tJ}{\mathtt{J}}
\newcommand{\tL}{\mathtt{L}}
\newcommand{\tp}{\mathtt{p}}
\newcommand{\de}{\mathrm{d}}
\newcommand{\pa}{\partial}

\newcommand{\tth}{\mathtt{h}}

\newcommand{\cH}{\mathcal{H}}

\newcommand{\bro}{\bar\rho}

\newcommand{\tg}{\mathtt{g}}

\numberwithin{equation}{section}
\counterwithin{lem}{section}

\title{\bf 
Infinitely many  isolas of \\
 modulational instability for Stokes waves
}

\begin{document}

 \author{Massimiliano Berti, Livia Corsi, Alberto Maspero, Paolo Ventura}

\date{}

\maketitle

\noindent 
{\bf Abstract.}  
This paper proves long-standing conjectures regarding  the 
existence of {\it infinitely}  many   
high-frequency  modulational  instability 
``isolas" 
for a Stokes wave in 
arbitrary depth $ \tth > 0 $,   
under  
longitudinal perturbations. 
We provide a complete characterization of the unstable spectral bands in the $L^2(\mathbb{R})$-spectrum of the water wave equations linearized around a Stokes wave of sufficiently small amplitude $\e$.
The unstable spectrum
is the union of 
isolated
``\emph{isolas}" of elliptical shape,   
indexed by integers $ \tp\geq 2 $, 
each with  semiaxis of size 
$  |\beta_1^{(\tp)} (\tth)| \e^\tp+ \cO(\e^{\tp+2} )$. 
As first key achievement, we obtain 
an explicit formula for the coefficient  $ \beta_1^{(\tp)} (\tth) $
for {\it any} $ \tp \geq 2 $, that 
remarkably depends
solely on 
the maximal 
Taylor-Fourier coefficients of the Stokes wave.  
We provide simple 
expressions of the asymptotic expansion of such 
coefficients  in the shallow-water limit 
$ \tth \to 0^+ $, for any $ \tp \geq 2 $.
This allows to establish that the analytic function $\beta_1^{(\tp)}(\tth)$ is
not zero for any $\tp \geq  2$, 
by verifying that a combinatorial sum is not zero;
{this relies  on a crucial 
combinatorial identity due to Koutschan, van Hoeij,
and Zeilberger. } 

\tableofcontents

\section{Introduction}

Stokes waves are among the most renowned solutions of the gravity water waves equations. Discovered in the groundbreaking work 
\cite{stokes} of Stokes in 1847, 
these waves are spatially periodic and travel 
steadly in the $ x $-direction, 
persisting for all times 
despite dispersive effects.  
  The first mathematically rigorous proof of their existence was given by \cite{Struik,LC,Nek} almost one century ago.

 A fundamental physical question concerns their stability or instability. 
 The mathematical problem can be formulated as follows.  Consider 
   a $2\pi$-periodic  Stokes wave 
    with amplitude $0< \e \ll 1$
    in an ocean with depth $ \tth > 0 $. 
The linearized  water waves equations at the Stokes waves
are, in the 
reference frame moving with the speed of the wave, 
 a linear autonomous system 
\begin{equation}
\label{linoriginale}
\pa_t h (t,x) = \mathcal{L}_{\e}
( \tth) \, [h(t,x)]   
\end{equation}
where 
$  \mathcal{L}_{\e}
(\tth)  $
is a  real Hamiltonian pseudo-differential operator  with $ 2 \pi $-periodic coefficients, cfr.  \eqref{cLepsilon}.
The dynamics
of \eqref{linoriginale} for    
perturbations $ h (t,\cdot) $  in  $ L^2 (\bR) $ is fully determined by answering the following
\begin{itemize}\item
{\sc Question: } {\it What is the $ L^2 (\bR)$ spectrum of
$ \mathcal{L}_{\e}
(\tth) $?} 
\end{itemize}

So far, it has only been known that such spectrum contains a figure ``8" close to the origin \cite{BrM,NS,BMV1,BMV3,BMV_ed} and an isola right above it 
\cite{HY,CNS, CNS2, JRSY, BMV4}
(and its  symmetric one below).  
Based on numerical simulations,
Deconink-Oliveras \cite{DO}  conjectured 
 the existence of an infinite sequence  
 of isolated unstable spectral bands within the spectrum, 
 located along the imaginary axis, 
exhibiting a near-elliptical shape --named ``isolas"--
shrinking exponentially fast away from the origin.

The objective of this paper is to rigorously prove their existence. 

\begin{teo}\label{thm:main}
For any integer  $ {\mathtt p} \geq 2$,  there exist $\e_1^{({\mathtt p} )} >0$ and a closed set of isolated depths ${\cal S}^{({\mathtt p} )} \subset (0, + \infty)$ such that for any 
$\tth \notin 
{\cal S}^{({\mathtt p} )} $ and $0< \e \leq \e_1^{({\mathtt p} )}$, the spectrum $\sigma_{L^2(\bR)} \big( \cL_\e(\tth) \big) $ 
contains  
$\tp -1$ disjoint ``isolas" in the complex upper half-plane,
see Figure \ref{fig:isole}, plus the symmetric ones in 
the lower half-plane. 
For any 
$\ell=2,\dots,{\mathtt p} $, the $\ell$-th  isola 
is approximated by an ellipse
with semiaxes of size
 $ \sim | \beta_1^{(\ell)}(\tth)| \e^\ell $, centered 
 on the imaginary axis at a point
  $ \im \omega_*^{(\ell)}(\tth) + \cO(\e^2) $, 
where 
\begin{equation}\label{iomegap}
0<\omega_*^{(2)}(\tth )  
< \dots < \omega_*^{(\ell)}(\tth ) < \dots
\qquad \text{satisfy}\qquad
\lim_{\ell \to + \infty}
\omega_*^{(\ell)}(\tth) =  + \infty \,  .
\end{equation} 
Each  $ \beta_1^{(\ell)}(\tth)$ is a non-zero analytic function of the depth
$ \tth > 0 $.  
\end{teo}
\begin{figure}[h!!!]\centering \subcaptionbox*{}[.45\textwidth]{\includegraphics[width=6cm]{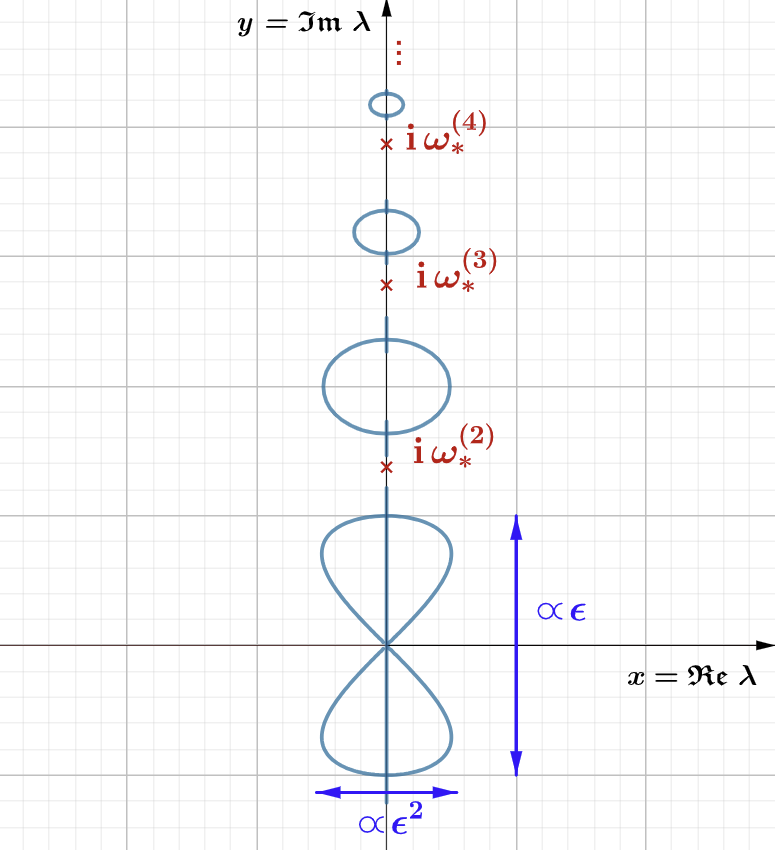}}
\subcaptionbox*{}[.45\textwidth]{
\includegraphics[width=7cm]{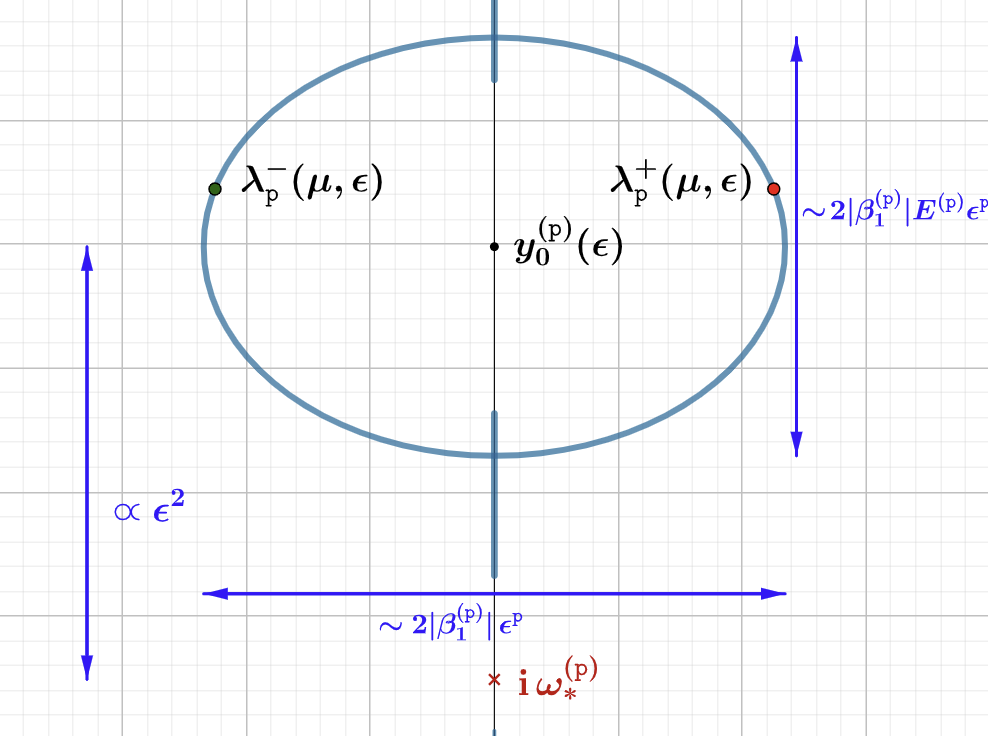} }
 \vspace{-0.8cm}
\caption{ \label{fig:isole} 
Spectral bands with non zero real part of the 
$ L^2 (\bR) $-spectrum of
$ \cL_{\e} $. 
On the right,  zoom of the $\tp$-th unstable isola. 
Its center $ y_0^{({\mathtt p} )} (\e) $ is $ \cO (\e^2) $ distant from $ \im \omega_*^{({\mathtt p} )} $ and its size is $ \propto \e^{{\mathtt p} } $. As shown in Theorem \ref{thm:main2}, two unstable 
eigenvalues $\lambda^{\pm}_\tp(\mu,\e)$ of the Floquet operator $\cL_{\mu,\e}$  span the $\tp$-th isola for Floquet exponents $\mu \in (\mu_\wedge^{({\mathtt p} )}(\e),\mu_\vee^{({\mathtt p} )}(\e) ) $ and recollide on the imaginary axis at the ends of the interval.  
}
\end{figure}

A more  complete statement is given in Theorem \ref{thm:main2}.
We point out that
 the number of isolas 
 in the spectrum of ${\cal L}_\e (\tth) $
 given by Theorem \ref{thm:main}
  is {\em arbitrarily large},
 provided $\e$ is sufficiently small.

\smallskip

 Let us summarize the state of the art
 on this problem.

 In the sixties 
 Benjamin and Feir \cite{BF,Benjamin}, Whitham \cite{Whitham},  Lighthill \cite{Li} and Zakharov \cite{Z0,ZK} discovered, 
 through experiments and formal arguments, that small amplitude  Stokes waves in sufficiently deep water are modulationally unstable
when subject to long-wave perturbations.  
This phenomenon, called ``Benjamin-Feir'' --or modulational-- instability, 
is nowadays supported by an enormous  amount of  
physical observations and numerical simulations, see e.g. \cite{DO,KDZ},  
and detected in  several dispersive and fluid PDE models, such as KdV, gKdV, NLS, Whitham equation, ..., see \cite{Wh,SHCH,GH,HK,BJ,J,HJ,BHJ,HP, MR}.

In mathematical terms, this phenomenon corresponds to the emergence of spectral bands 
of $ \cL_\e $  close to zero, spanned by 
eigenvalues with non zero real part. 
The first rigorous proof 
of local branches of Benjamin-Feir spectrum was obtained   by Bridges-Mielke \cite{BrM}  in finite depth,  see also \cite{HY}, and by Nguyen-Strauss \cite{NS}  
in deep water. We also mention the nonlinear result in Chen-Su \cite{ChenSu}.  Recently 
Berti-Maspero-Ventura \cite{BMV1,BMV2,BMV_ed} proved that these local branches extend to form a complete figure ``8'',
as observed numerically in \cite{DO}. 
This settles the question about the shape of the spectrum close to the origin. 

\smallskip

Preliminary numerical studies concerning the existence of unstable spectrum situated away from the origin were carried out by McLean 
\cite{Mc1,Mc2} in the 1980s, analyzing longitudinal and transversal wave perturbations. Later Deconinck and Oliveras \cite{DO}  
numerically 
computed the first  isolas of high-frequency instability 
in the longitudinal case and 
conjectured 
 the existence of infinitely many ones parameterized by integers $\tp \geq 2 $.
 In the last years a lot of 
 work has been devoted to this problem yielding the
 formulation of the following conjectures,  see e.g. \cite{CrT,CD}: 
 for any $ \tp \geq 2 $
\begin{align}\label{weakconj}
&\bullet \textup{ \underline{Weak conjecture}:  for any depth $ \tth > 0 $, the size of the $\tp$-th unstable isola, if present, is $\cO(\e^\tp)$.} \tag{\bf W}\\ \label{strongconj}
&\bullet  \tag{\bf S} \textup{ \underline{Strong conjecture}:
for ``many" depths, the $\tp$-th isola 
exists and with size asymptotically $\sim \e^\tp$. }
\end{align}
Clearly \eqref{strongconj} implies \eqref{weakconj}. 
Rigorous results have been proved  for $ \tp = 2 $  
in finite depth
in  \cite{HY}  and for  deep water
in \cite{BMV4}. 
For $\tp= 3 $   a formal result is given in \cite{CDT}.  

Theorems \ref{thm:main} and \ref{thm:main2} solve both 
\eqref{weakconj} and  \eqref{strongconj} for any $ \tp \geq 2 $,  as well as  the 
\begin{align} \notag 
\bullet  &\text{ \underline{Improved weak conjecture}: 
for the  
``degenerate" depths
excluded in \eqref{strongconj}
the $\tp$-th unstable isola,}\\ \tag{\bf W'}
&\text{if present,  has size $\cO(\e^{\tp+2})$,  for any 
$ \tp \geq 2 $.} \label{impweakconj}
\end{align}

\smallskip 


We also mention the results \cite{CNS, CNS2, JRSY,HTW} 
regarding transversal wave perturbations.

A fundamental difficulty in proving \eqref{strongconj} lies 
in establishing the non-degeneracy of a 
specific coefficient which characterizes
the real part of the eigenvalues. 
When  $ \tp = 2 $ 
this coefficient  depends 
solely on 
 the second-order Taylor expansion of the Stokes wave
 (on fourth-order  in the degenerate case \cite{BMV4}
and third-order for transversal perturbations in \cite{CNS, CNS2}) 
thus enabling  explicit computations 
with a manageable  effort
(frequently assisted by computer algebra systems and numerical methods). 

On the other hand, the rigorous proof of 
the strong conjecture \eqref{strongconj} 
faces the following 
challenges: 
\begin{itemize}\label{challenge}
\item[(1)] Obtaining the expression of the relevant Taylor coefficient
 --the function $ \beta_1^{(\tp)} (\tth) $ in Theorem \ref{thm:main}-- 
for {\it any} integer $ \tp \geq 2 $. 
We achieve the result in Theorem \ref{lem:expansionL}, that provides  
an explicit formula of  $\beta_1^{(\tp)} (\tth)$  in terms of the ``maximal"  
Taylor-Fourier coefficients of the Stokes wave.
\item[(2)] 
 Demonstrating that this 
coefficient $ \beta_1^{(\tp)} (\tth) \not \equiv 0 $ does not vanish identically in $\tth  $ for any $ \tp \geq 2 $. 
This is proved in Section \ref{combinatoricsincoming} 
relying on the 
formulas for 
the maximal Taylor-Fourier 
coefficients of the Stokes waves in the 
shallow limit 
$ \tth \to 0^+ $, given in 
Proposition \ref{lem:asympetapsi}.
\end{itemize}
Both steps are considerably demanding. 
Although Kato's perturbation theory of eigenvalues is algorithmic in principle, its direct practical implementation proves computationally infeasible.   
Inputting all the 
extensive power series coefficients of the Stokes waves
at any order in $ \e $,  
along with 
the Taylor expansion of the Dirichlet-to-Neumann operator,  and calculating the numerous Kato projectors 
required to 
compute the expression of the function
$ \beta_1^{(\tp)} (\tth) $ for any $ \tp $, 
is not practicable, even with current computing power.  Furthermore, although recursive, the  power series expansion of the Stokes wave coefficients,
are not explicitly known at any order, 
which is 
required to check the non vanishing of 
$ \beta_1^{(\tp)} (\tth) $ for any $ \tp $. 

In Section \ref{sec:res1} 
we  describe in detail the results
and, in Section \ref{sec:res2}, the main ideas to achieve them. 

\subsection{Main result} 
\label{sec:res1}

Since the operator $\cL_\e := \cL_\e(\tth)$ has $2\pi$-periodic coefficients, the  starting point to obtain Theorem \ref{thm:main} is the Bloch-Floquet decomposition of the spectrum 
 \begin{equation}
     \label{BFtheory}
     \sigma_{L^2(\bR)}\big( 
     \cL_\e  \big)
 = \bigcup_{\mu \in \big[-\tfrac12, \tfrac12 \big)}
 \sigma_{L^2(\bT)}\big( \cL_{\mu,\e} 
 \big)\, \quad \textup{where}\quad 
 \cL_{\mu,\e} :=
\cL_{\mu,\e} (\tth) := e^{-\im \mu x}\, \cL_\e\,  e^{\im \mu x}  
 \end{equation}
is a complex pseudo-differential  
operator,
with $2 \pi $ coefficients, 
depending on $    \mu $, see \eqref{WW}. The spectrum  of
$ \cL_{\mu,\e} $ 
on 
$ L^2 (\bT ) $ 
is discrete and its eigenvalues  span, as $ \mu $ varies,   the continuous spectral  bands of $\sigma_{L^2(\bR)}\big( 
     \cL_\e  \big)$. 
The parameter 
$ \mu $ is usually referred to as the Floquet exponent.
A direct consequence of the Bloch-Floquet decomposition is that any 
solution 
of \eqref{linoriginale} can be decomposed as a linear superposition of Bloch-waves
\begin{equation}
\label{hesplode}
h(t,x) =
e^{ \lambda t}
e^{\im \mu x} v(x)
\end{equation}
where  $\lambda
$ is an eigenvalue of $\cL_{\mu,\e} $ 
with  associated  eigenvector 
$ v(x) \in L^2(\bT) $. 
 If $\lambda $ is  unstable, i.e.\ has positive real part,  then the 
solution 
\eqref{hesplode}
 grows exponentially fast in time.  
 

We remark that the spectrum
$ \sigma_{L^2(\bT)}(\cL_{\mu,\e}) $
is a set which is
$1 $-periodic in $ \mu $, and thus
it is sufficient to consider 
$ \mu $ in  the first zone of Brillouin $ [-\tfrac12, \tfrac12) $. Furthermore by the reality 
of $ \cL_{\e} $ the spectrum $ \sigma_{L^2(\bT)}(\cL_{-\mu,\e}) $
is the complex conjugated of 
$ \sigma_{L^2(\bT)}(\cL_{\mu,\e})$
and we can restrict  to $\mu \in [0,\tfrac12) $.

\smallskip
The Hamiltonian nature of the operator $\cL_{\mu,\e}  $ constrains the eigenvalues with nonzero real part to arise, for
small $ \e > 0  $, as perturbation of {\it multiple}  eigenvalues of  $\cL_{\mu,0} (\tth) $.  The spectrum of $\cL_{\mu,0}(\tth) $ is  purely imaginary. 
As stated in Lemma \ref{thm:unpert.coll}, for any 
$\tth > 0 $ and $ \mu \in \bR$, the operator $\cL_{\mu,0} (\tth) $ possesses, away from $0$, only simple or double eigenvalues: 
there exists a diverging sequence of 
Floquet exponents  
 $$
 \umu=\uphi(\tp,\tth) > 0 \, , \qquad  
 \tp \in \bN \, ,  \ \tp 
 \geq 2 \, ,  
 $$
(and its integer  shifts $ \umu +k$, $ k \in \bZ $), such that  
$\cL_{\umu,0}(\tth) $ possesses   double  eigenvalues
$ \{\pm \im \omega_*^{(\tp)}(\tth ) \}_{\tp=2,3,\dots} $  
forming  a diverging sequence  
as in \eqref{iomegap}. 
 In other words  the unperturbed spectral bands of $ \cL_{\mu,0} 
 (\tth) $ intersect at $ \umu = \underline{\varphi} (\tp,\tth) $,
 mod $ 1$.

To state the next result, we introduce the following notation.

\begin{itemize} \item  {\bf Notation}.
We denote by  $ \cO (y^{m_1}\e^{n_1},\dots,y^{m_q}\e^{n_q}) $ ($y =\mu$, $\delta$ or $\nu$ depending on the context) 
an analytic function 
with values in a Banach space $ X $ that  satisfies,  for some $ C > 0 $ and $(y, \e)$ small, the bound
 $ \| \cO  (y^{m_1}\e^{n_1},\dots,y^{m_q}\e^{n_q}) \|_X   \leq C \sum_{j = 1}^q|y|^{m_j}|\e|^{n_j}\, 
 . $
  We denote by $r(y^{m_1}\e^{n_1},\dots,y^{m_q}\e^{n_q}) $ any
scalar  function  $\cO(y^{m_1}\e^{n_1},\dots,y^{m_q}\e^{n_q})$ which is  also  real analytic. For any  $\delta > 0 $, we denote by $B_\delta(x)$ the real interval $(x-\delta,x+\delta)$ centered at $x$.
\end{itemize}

The  next result  sharply describes 
the  spectrum of $ \cL_{\mu,\e} (\tth) $
near $\im \omega_*^{(\tp)}(\tth)$ 
for any $(\mu, \e)$  close to $(\umu,0)$.

\begin{teo}\label{thm:main2} 
{\bf (Spectral Bloch-Floquet 
bands)} For any integer $\tp \in \bN $, $\tp  \geq 2$, for any $ \tth >0 $ let $\umu
= \uphi(\tp,\tth) > 0 $ (given by Lemma \ref{collemma}) such that the operator $ \cL_{\umu,0}(\tth)$ in \eqref{cLmu} has a 
double eigenvalue at $\im \omega^{(\tp)}_* (\tth) $.   
Then there exist 
\begin{itemize}
\item[i)] 
a non zero real analytic function 
$\beta_1^{(\tp)}(\tth) $ defined for any
$ \tth > 0 $,  satisfying 
\begin{equation}\label{beta1plimit}
\lim_{\tth\to 0^+}
\beta_1^{(\tp)}
(\tth) = - \infty \, , 
\qquad
 \quad  
 \lim_{\tth\to+ \infty}
\beta_1^{(\tp)}
(\tth) = 0 \, ;
\end{equation}
\item[ii)]
constants  
$\e_1^{(\tp)}, \delta_0^{(\tp)}
>0$ and  real  analytic functions $\mu_0^{(\tp)},\mu_\vee^{(\tp)},\mu_\wedge^{(\tp)} \colon [0, \e_1^{(\tp)}) \to \bR $ 
satisfying 
$$
\mu_0^{(\tp)} (0) =  \mu_{\wedge}^{(\tp)}(0) = \mu_\vee^{(\tp)}(0) =\umu
$$ 
and, for any 
$ 0 < \e < \e_1^{(\tp)}   $, 
 \begin{equation}\label{mupm}
  \mu_\wedge^{(\tp)}(\e) < \mu_0^{(\tp)}(\e) < \mu_\vee^{(\tp)}(\e)\, ,\quad  |\mu_{\wedge,\vee}^{(\tp)}(\e)- \mu_0^{(\tp)}(\e) |   =   \frac{2 |\beta_1^{(\tp)}(\tth) |}{T_1^{(\tp)} (\tth) } \e^\tp+ r(\e^{\tp+{2}})\, ,
 \end{equation}
(see Figure \ref{fig:instareg}) where   $T_1^{(\tp)}(\tth) > 0 $ 
is an 
analytic function for any
$ \tth > 0 $; 
\end{itemize}
such that
\begin{itemize}
\item {\bf (Unstable eigenvalues)}
for any $(\mu, \e) \in B_{\delta_0^{(\tp)}}(\umu) \times B_{\e_1^{(\tp)}}(0)$ 
the operator $ \cL_{\mu,\e} (\tth) $ possesses a pair of eigenvalues  
\begin{align}\label{final.eig}
&\lambda^\pm_\tp \big(\mu, \e \big) 
=\begin{cases} \im \omega_*^{(\tp)}+ \im s^{(\tp)}(\mu,\e)  \pm \dfrac12 \sqrt{ D^{(\tp)}\big(\mu,\e\big)}  &\mbox{if }  \mu \in \big( \mu_\wedge^{(\tp)}(\e) , \mu_\vee^{(\tp)} (\e)\big) \, , \\[3mm]
\im \omega_*^{(\tp)}+ \im  s^{(\tp)}(\mu,\e) 
 \pm \im\sqrt{ \big|D^{(\tp)}\big(\mu,\e\big)\big| } &\mbox{if } \mu \notin \big( \mu_\wedge^{(\tp)} (\e) , \mu_\vee^{(\tp)} (\e)\big)\, ,
\end{cases} 
\end{align} 
 where $s^{(\tp)}(\mu,\e)$, $D^{(\tp)}(\mu,\e)$ are  real-analytic functions of the form 
\begin{equation}\label{degenerateexpD0}
\begin{aligned}
& s^{(\tp)}(\mu,\e) = r(\e^2,\nu) \qquad \text{where} \qquad  
\nu := \mu - \mu_0^{(\tp)} (\e)  \, , 
\\
& D^{(\tp)}(\mu,\e) =  4(\beta_1^{(\tp)}(\tth))^2 \e^{2\tp} -  (T_1^{(\tp)}(\tth))^2\nu^2+r(\e^{2\tp+{2}},\nu\e^{2\tp},\nu^2\e^{{2}},\nu^3)\, .
\end{aligned}
\end{equation}
 For any $ \tth > 0 $ such that 
 $ \beta_1^{(\tp)} (\tth) \neq 0 $  the function 
$$ 
D^{(\tp)}(\mu,\e) 
\  \text{is} \ 
\begin{cases}
> 0 \qquad \ \forall \mu \in 
\big(\mu_\wedge^{(\tp)} (\e), \mu_\vee^{(\tp)} (\e)\big) \, \\
= 0 \qquad \ \forall \mu \in \{ \mu_{\wedge}^{(\tp)} (\e) \, , 
\mu_{\vee}^{(\tp)} (\e)\}   \\
< 0 \qquad   \  \forall \mu \notin 
\big(\mu_\wedge^{(\tp)} (\e), \mu_\vee^{(\tp)} (\e)\big) \, . 
\end{cases}
$$
\item 
{\bf (Unstable isola)} If $\beta_1^{(\tp)}(\tth) \neq 0$ then, for any  fixed $\e \in (0,\e_1^{(\tp)})$, as $\mu$ varies in $\big(\mu_\wedge^{(\tp)}(\e),\mu_\vee^{(\tp)}(\e)\big)$
 the pair of unstable eigenvalues $\lambda^\pm_\tp(\mu,\e)$ in \eqref{final.eig} describes 
 a closed analytic  curve 
 in the complex  plane 
 $ \lambda = x + \im y $
 that
  intersects orthogonally the imaginary axis,  
  encircles a convex region, and it is symmetric 
  with respect to $y $-axis.
  Such isola is
 approximated by an  ellipse
 \begin{equation}\label{isolaintro}
x^2  + 
(E^{(\tp)}(\tth))^2(1+r(\e^2)) \Big( y-y_0^{(\tp)}(\e) \Big)^2 =   (\beta_1^{(\tp)}(\tth))^2 \e^{2\tp} (1+r(\e^{{2}})) 
\end{equation}
where $ E^{(\tp)}(\tth) 
\in (0,1) $ is
a real analytic function  for any  
$ \tth > 0 $
and $ y_0^{(\tp)} (\e) $
is $ \cO(\e^2) $-close to $ \omega_*^{(\tp)} (\tth) $. 
\item 
{\bf (Degenerate depths)} If $\beta_1^{(\tp)}(\tth) = 0 $ the real part of the unstable eigenvalues $\lambda^{\pm}_\tp(\mu,\e)$, if present,  is of size $\cO(\e^{\tp+2}) $. 
\end{itemize}
 \end{teo}

Theorem 
\ref{thm:main2} follows by 
Theorem 
\ref{thm:finalmat},  
the abstract Theorem  \ref{primeisola}
and Lemma \ref{lem:appell}.

Let us make some comments.  

 \begin{figure}[h!!]
\centering
\includegraphics[width=7cm]{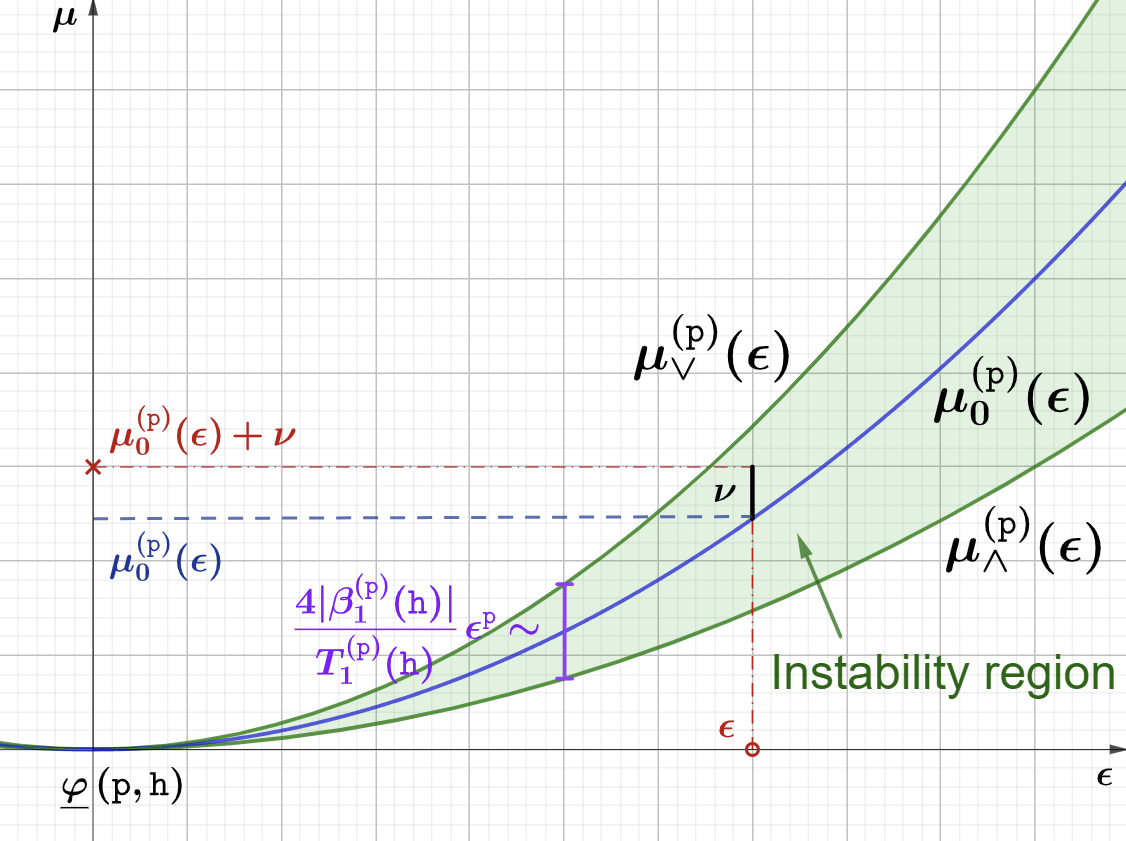}
\caption{ \label{fig:instareg} The instability region around the curve $\mu_0^{(\tp)}(\e)$  delimited by  the curves 
$\mu_\wedge^{(\tp)}(\e)$ and $\mu_\vee^{(\tp)}(\e)$.}
\end{figure}

\begin{enumerate}
\item {\bf Weak conjecture, upper bounds:}  \label{upperbounds}
Theorem 
\ref{thm:main2} proves the weak conjecture \eqref{weakconj} for all depths.
Indeed, in view of \eqref{degenerateexpD0} and \eqref{mupm}, for any $\tp \geq 2$ and  any depth $\tth>0$,
the real part of the eigenvalues
\eqref{final.eig}
 is at most of size $ \cO(\e^\tp) $
implying  that the unstable isolas,
 if ever exist,   shrink exponentially fast as  $ \tp  \to + \infty $.    
\item 
{\bf Strong conjecture,  
lower bounds:} 
Theorem 
\ref{thm:main2} proves  the strong conjecture \eqref{strongconj} for all depths such that $\beta_1^{(\tp)}(\tth) \neq 0$.
In this case
the real part of the eigenvalues
\eqref{final.eig}
 is  of size $  |\beta_1^{(\tp)} (\tth)| \e^\tp+ \cO(\e^{\tp+2} )$
 and by \eqref{isolaintro}
 the $ \tp $-th  unstable isola  has   
elliptical shape.  
 The most unstable eigenvalue is reached at $\nu \propto\!\e^{2\tp}$, see  \eqref{nuRemaxnondeg}. 
Note also 
that by \eqref{isolaintro}  each isola drifts
of $\cO(\e^2)$ from its known zeroth-order center 
$ \im  \omega_*^{(\tp)} (\tth)  $, thus rather quickly relative to its  $  \e^\tp $-size. 
 \item
 {\bf Unstable Floquet exponents:} 
 By \eqref{mupm},  the portion of the 
 unstable spectral $\tp$-th isola 
 is parametrized by Floquet exponents $ \mu $ 
in  the very narrow interval
 $ \big( \mu_{\wedge}^{(\tp)}(\e),\mu_{\vee}^{(\tp)}(\e)\big) $ which has exponentially small
 width 
 $ \sim  \frac{4|\beta_1^{(\tp)}(\tth) |}{T_1^{(\tp)}(\tth)} 
 \e^{\tp} $. 
\item
{\bf 
Non degeneracy of the function $\beta_1^{(\tp)}(\tth) $:}
we 
analytically {\it prove} 
that, for {\it any} $\tp\geq 2 $, the map 
$\tth \mapsto \beta^{(\tp)}_1(\tth)$  is   real analytic and  
satisfies \eqref{beta1plimit}. 
In the next section we explain how we achieve this challenging result.  
For $\tp= 2,3,4 $ the graphs of 
$ \beta_1^{(\tp)}
(\tth) $ are  numerically plotted   in Figures 
\ref{plotb1p2}.
\begin{figure}[h!]
 \centering
\includegraphics[width=16cm]{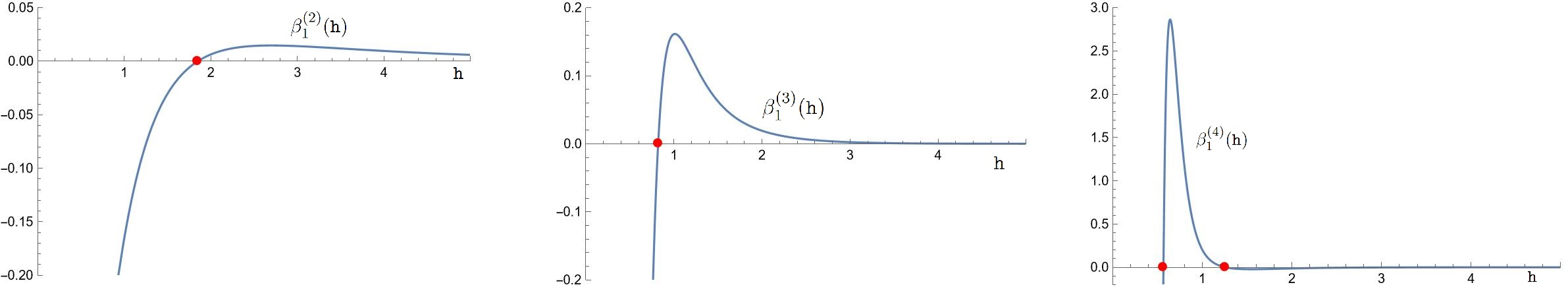}
\caption{ \label{plotb1p2} Cases $\tp= 2,3,4 $. 
The function $\beta_1^{(2)}(\tth)$ vanishes 
only at the red point 
$\tth_*^{(2)}\approx 1.84940 $ coherently with  \cite{CDT,HY}. The function $\beta_1^{(3)}(\tth)$ vanishes only at the red point  $\tth_*^{(3)}\approx 0.82064 $ coherently with  \cite{CDT}. the function $\beta_1^{(4)}(\tth)$ vanishes only at the red points $\tth_*^{(4)}\approx 0.566633$ and $\tth_{**}^{(4)}\approx 1.255969$.}
\end{figure}
\item
{\bf Improved weak conjecture:
} 
 Each function $\beta_1^{(\tp)}(\tth) $  vanishes only at a closed set of isolated depths. 
The set ${\cal S}^{(\tp)}$ 
in Theorem \ref{thm:main} is the union 
$$
{\cal S}^{(\tp)}:= \bigcup_{\ell =2}^\tp \big( \beta_1^{(\ell)} \big)^{-1}(0) \, . 
$$
When $ \beta_1^{(\tp)} (\tth) = 0 $ 
it is necessary to further  expand the discriminant $D^{(\tp)}(\mu,\e)$
in \eqref{degenerateexpD0}
to determine its sign. The latter degeneracy occurs
at  
$ \tth = + \infty $ for {\em every} $\tp \geq 2 $ since  $ \lim_{\tth \to + \infty}
\beta_1^{(\tp)} (\tth) = 0 $,
as we prove in Proposition \ref{deeplimit}.
 Such degeneration ultimately descends from a 
structural 
property of the Stokes waves in deep water at any order proved in  Proposition 
\ref{degenstokesinfinity}. 

The finding looks as a further interesting  aspect of the integrability properties of the deep water wave equations,  discussed in \cite{ZakD,BFP,Wu,DIP}. 

The last item of Theorem \ref{thm:main2} 
actually 
proves the 
improved weak conjecture: 
if 
$ \beta_1^{(\tp)} (\tth) = 0 $
(degenerate depths), 
the $\tp$-th unstable isola, if present,  has size $\cO(\e^{\tp+2})$, 
for any 
$ \tp \geq 2 $.

For  $\tp = 2 $  the 
computations 
required to detect the unstable isola of size $ \sim \e^4 $ at 
the degenerate depth 
$ \tth = + \infty $ 
have been  performed 
in \cite{BMV4}. 
\item {\bf Conjectures: }In  light of Theorem \ref{thm:main2} we conjecture that 
\begin{enumerate}
\item The analytic function $\tth \mapsto\beta_1^{(\tp)}(\tth)$ 
has only a finite number of roots for any $\tp \geq 2$.
\item For any $\tp \geq 2$ the $\tp$-th isola exists also for degenerate depths and
its size is 
$\sim \e^{\tp+2}$.
\end{enumerate}
\end{enumerate} 

The following section comments on the main difficulties and key 
novelties of the proof.

\subsection{ Ideas of proof}\label{sec:res2}

The overall goal is to compute the eigenvalues of the Hamiltonian and reversible 
operator
\begin{equation}\label{WWinit}
\cL_{\mu,\e}    = \begin{bmatrix} (\pa_x+\im\mu)\circ (\ch+p_\e(x)) & 
 |D+\mu| \tanh\big((\tth + \ttf_\e) |D+\mu| \big) \\ -(1+a_\e(x)) & (\ch+p_\e(x))(\pa_x+\im \mu) \end{bmatrix}
\end{equation}
that branch off from the $\tp$-th double eigenvalue $\im \omega_*^{(\tp)} (\tth)  $ of $\cL_{\umu,0} (\tth) $ for $(\mu,\e)$ close to $(\umu,0)$. The functions $a_\e(x)$, $p_\e(x)$ are real-analytic $2\pi$-periodic even functions that depend 
on the Stokes wave $(\eta_\e(x), \psi_\e(x))$ along which we linearize the water waves equations \eqref{WWeq}. In \eqref{WWinit} the constant $\ch = \sqrt{\tanh(\tth)}$ is the linear speed of the Stokes wave.

In Section \ref{Katoapp}, 
following the general approach in  \cite{BMV1,BMV3,BMV_ed,BMV4}, we  reduce the spectral problem to study the eigenvalues of a 
$2\times 2$ symplectic and reversible matrix, i.e.\  of the form (cfr.  \eqref{tocomputematrix})
\begin{equation}\label{L.intro}
\tL^{(\tp)}(\mu,\e) = 
\begin{pmatrix} -\im\alpha^{(\tp)} (\mu,\e) & \beta^{(\tp)}(\mu,\e) \\ \beta^{(\tp)} (\mu,\e) & \im \gamma^{(\tp)} (\mu,\e) \end{pmatrix} \, \end{equation}
where $\alpha^{(\tp)}(\mu,\e)$, $\gamma^{(\tp)}(\mu,\e)$ 
and $\beta^{(\tp)}(\mu,\e)$
are  real  analytic functions depending on 
the Stokes wave. 
In particular the function
$\beta^{(\tp)}(\mu,\e)$ is given by  the scalar product 
\begin{equation}
\label{betaph}
\beta^{(\tp)}(\mu,\e) = \frac{1}{\im}
 (\mathfrak{B}(\mu,\e) f_0^-, f_\tp^+)  
\end{equation} 
where 
$\mathfrak{B}(\mu,\e)$  
is the operator \eqref{Bgotico} 
obtained by the Kato similarity reduction 
process and 
\begin{equation}\label{autovettorikernel0}
f_{0}^- = 
\frac{1}{\sqrt{2\Omega(\uphi(\tp,\tth),\tth)}} \vet{- \im \,\Omega(\uphi(\tp,\tth),\tth)}{1}   \, , 
\  f_{\tp}^+  =
\frac{1}{\sqrt{2\Omega(\uphi(\tp,\tth)+\tp,\tth)}} \vet{- \Omega(\uphi(\tp,\tth)+\tp,\tth)}{\im}
e^{\im \tp x}   \, , 
\end{equation}
are the unperturbed eigenvectors 
of ${\cal L}_{\umu,0}$ in \eqref{autovettorikernel} associated to the double eigenvalue $\im \omega^{(\tp)}_* (\tth)  $. 
Note that   $ \tp $   is the ``spectral gap" between the 
harmonics of the eigenvectors 
$ f_{0}^- $ and $ f_{\tp}^+  $ in \eqref{autovettorikernel0}.

 The crucial step in proving Theorem \ref{thm:main2} is to show  that
\begin{equation}\label{betaexpansion}
\beta^{(\tp)} (\umu+\delta,\e) = 
\beta_1^{(\tp)}
(\tth)\e^\tp + \text{h.o.t}
\qquad \text{at} \qquad 
\umu = \uphi(\tp,\tth)\, ,
\end{equation}
and prove that 
$\beta_1^{(\tp)} (\tth) $ 
is not zero.  
We now list 
various obstacles encountered, as discussed in items (1)
and (2) on page \pageref{challenge},  and the corresponding 
ideas for addressing them. 
\\[1mm]
{\bf (1) The exact expression of the function $\beta_1^{(\tp)} (\tth)$}. 
In view of \eqref{betaph},
the first step to  prove \eqref{betaexpansion} is to compute the Taylor expansion of the  operator $\mathfrak{B}(\mu,\e)$  at $(\umu, 0)$ up to order $\tp$.
This is achieved in Proposition \ref{lem:expBgot}, which provides
an explicit formula for its  Taylor coefficients  
involving those of the operator $\cB(\mu,\e) $ defined in \eqref{WW} 
and  of the spectral projector
$P(\mu, \e) = - \frac{1}{2\pi \im} \oint_\Gamma (\cL_{\mu,\e} - \lambda)^{-1} \de \lambda$.  
This expression depends on all Taylor coefficients of the Stokes wave expansion up to order $ \tp $. Computing 
all such coefficients at an arbitrary order $ \tp $ is prohibitively complex and lengthy, even with computer-assistance.
However, we prove this is unnecessary: 
$ \beta_1^{(\tp)}(\tth) $
depends solely on a select subset of 
Taylor-Fourier coefficients that we identify,  
drastically reducing computational complexity.

To prove this, our new strategy  fully leverages a structural property of the operator $\cL_{\mu,\e}$ that  propagates throughout Kato's reduction scheme. 
 This property, put simply, is that the matrix representing the action of each ``jet" 
$$
\cL_{i,n}:= \frac{ \delta^i \e^j}{i! n!} (\pa_\mu^i \pa_\e^n \cL_{\umu,0})
$$
of the operator $\cL_{\mu,\e}$ exhibits 
only ``finitely many bands":  
precisely, for any $i \in \bN_0 $,  for any $v_1, v_2 \in \bC^2$
\begin{equation}\label{LijF.intro}
( \cL_{i,n} \,  v_1 e^{\im j_1 x} , \ 
 v_2 e^{\im j_2 x} ) = 0 \ \ \mbox{ if } |j_1 - j_2| > n \mbox{ or } j_1 - j_2 \not\equiv n \mbox{ mod } 2 \ . 
\end{equation}
We say that an operator of this form belongs to the class $\mathfrak{F}_n $, and a sum of $\ell$-homogeneous terms as 
$\sum_{i+n = \ell} \cL_{i,n}$ to the class $\mathtt{F}_\ell$, see Definition \ref{defFell}.
For $ A \in \mathfrak{F}_n $ and $\kappa \in \bZ$, 
its $\kappa$-th``band" operator $ A^{[\kappa]} $ 
is defined as the operator 
with  matrix coefficients $ A^{j_1}_{j_2} $
supported on the band 
$ j_2 -j_1 = \kappa $ only, see definitions \eqref{def.matrix.elements} and \eqref{band.def}.
A crucial fact is that the graded vector space $\bigoplus_{\ell \in \bN_0}\mathtt{F}_\ell$ is closed under composition, 
if  $A \in \mathtt{F}_\ell$  and $B \in \mathtt{F}_{\ell'}$ then  $A\circ B \in  \mathtt{F}_{\ell +\ell'}$.
This allows us to propagate this structure along Kato's reduction scheme, and prove that the jets 
$\mathfrak{B}_{i,n}$ of the operator $\mathfrak{B}(\mu, \e)$ belong to $\mathfrak{F}_n$.

Let us explain how this structure originates. 
The crucial step is to observe that the Stokes wave has itself a very particular structure. 
Precisely, expanding the Stokes waves in power series of $\e$ as
$$
 \eta_\e(x) = \e \cos(x) + \sum_{\ell \geq 2} \e^\ell \eta_\ell (x)\, , \quad \psi_\e(x) =  \e\ch^{-1}
 \sin(x) + \sum_{\ell \geq 2} \e^\ell \psi_\ell(x)\,
 $$
we prove  in Theorem \ref{LeviCivita} 
(see also  Section \ref{sec:App31}), 
that, 
 at {\em any } order $\ell$, the functions  $\eta_\ell(x)$ and $\psi_\ell(x)$ are given by
\begin{equation}\label{Stokes.intro}
 \eta_\ell(x) = \sum_{\kappa=0 \atop \kappa \equiv \ell \, \textup{mod}\,  2}^\ell\eta_\ell^{[\kappa]}\cos(\kappa x) \, , \quad \psi_\ell(x) = \sum_{\kappa=1 \atop \kappa \equiv \ell \, \textup{mod}\,  2}^\ell\psi_\ell^{[\kappa]}\sin(\kappa x)\, ,
\end{equation}
with {\em   harmonics $\kappa$
not larger than $\ell$, and 
with the same parity of $ \ell $}. 
Consequently, the functions $a_\e(x)$, $p_\e(x)$,  entering in the definition \eqref{WWinit} of $\cL_{\mu,\e}$ expand as
$a_\e(x) = \sum_{\ell \geq 1} 
\e^\ell a_\ell(x) $ , $p_\e(x) = \sum_{\ell \geq 1} 
\e^\ell p_\ell(x) $, where 
each $a_\ell(x), p_\ell(x) $ has 
the form  \eqref{Stokes.intro} as well.
This structure  in turn implies 
 \eqref{LijF.intro}.

\smallskip

The fact that the jets $\mathfrak{B}_{i,n}$ belong to $\mathfrak{F}_n $, for any $ i$,  i.e.\ fulfill \eqref{LijF.intro}, 
implies immediately the expansion \eqref{betaexpansion}, since 
$ ( \mathfrak{B}_{i,n} f_0^-,f_\tp^+) = 0 $
whenever
$ n  < \tp$.
In Section \ref{sec:up}, this property is 
exploited 
to establish 
an upper bound 
of $ \lesssim \e^\tp $ 
for the real part of the eigenvalues on the unstable 
spectral bands
(if they exist). 
Summarizing

\begin{itemize}
\item
the ``weak conjecture" 
\eqref{weakconj}, as well as the 
``improved weak conjecture"  
\eqref{impweakconj}, 
is proved by the algebraic properties  of the spaces  $ \mathfrak{F}_n $, inherited by properties of the Stokes wave. 
\end{itemize}

In order to prove also the  strong conjecture \eqref{strongconj} 
we need an explicit expression of $\beta_1^{(\tp)}(\tth)$.
Property \eqref{LijF.intro} plays a crucial role also in
proving that $\beta_1^{(\tp)}(\tth)$ depends only on the {\it maximal}  Fourier-Taylor  coefficients $a_\ell^{[\ell]}, p_\ell^{[\ell]}$ with $ |\ell | \leq \tp $, 
see Theorem \ref{lem:expansionL}. Let us explain why  with an example: a term of order 
$\e^\tp$ in the Taylor expansion in \eqref{betaph} has the form
\begin{equation}\label{1.17}
    ( A_{0, n_1}\circ \ldots \circ A_{0, n_q} f_0^-, f_\tp^+) \quad \mbox{ with } n_1 + \ldots + n_q = \tp \  \textup{and } \ 
    f_0^- \,,\; f_\tp^+ \textup{ in }\eqref{autovettorikernel0}\, ,
\end{equation}
for some operators 
 $A_{0, n_a} \in \mathfrak{F}_{n_a}$.
In the scalar product \eqref{1.17} 
only the band $(A_{0, n_1}\circ \ldots \circ A_{0, n_q})^{[\tp]}$ is relevant. 
But since each $A_{0, n_a}$ has bands $|\kappa_a| \leq n_a$, the only possibility to reach band $\tp$ is to take the highest band of each operator, namely 
$$
 ( A_{0, n_1}\circ \ldots \circ A_{0, n_q} f_0^-, f_\tp^+)  =  ( A_{0, n_1}^{[n_1]}\circ \ldots \circ A_{0, n_q}^{[n_q]} f_0^-, f_\tp^+) \, . 
 $$
 This principle is used in Lemma \ref{lem:tBp}. 
Furthermore, by the properties \eqref{Stokes.intro} of the Stokes wave, each  $A_{0,n}^{[n]}$ depends only on the highest maximal Fourier-Taylor  coefficients $a_n^{[n]}, p_n^{[n]}$.
Summarizing 
 \begin{itemize}
\item 
the fact that the expansions 
\eqref{beta1exp}-\eqref{entanglementfdj} of $ \beta_1^{(\tp)} (\tth) $ depend  of only the maximal Taylor-Fourier coefficients $a_\ell^{[\ell]}, p_\ell^{[\ell]}$ for 
$ |\ell| \leq \tp $ of the Stokes waves, deeply exploits this   ``maximal harmonics" principle of the spaces $ \mathfrak{F}_n $.
\end{itemize}
  
\noindent 
{\bf (2) Proof that $\beta_1^{(\tp)} (\tth) \neq 0 $ for any $ \tp \geq 2 $.} 
The function 
$\tth \mapsto \beta^{(\tp)}(\tth)$ is analytic by direct inspection
of the formulas 
\eqref{beta1exp}-\eqref{entanglementfdj}
since each $a_\ell^{[\ell]} (\tth), p_\ell^{[\ell]} (\tth) $ are analytic in $ \tth $, see Lemma \ref{pastructure}.
Such formulas  
remain quite complex, involving $3^{\tp-1}$ terms.
Furthermore, explicit expressions for the
Taylor-Fourier coefficients $ \eta_\ell^{[\ell]} (\tth) $ and $  \psi_\ell^{[\ell]} (\tth) $
of the Stokes waves are unknown for any 
$ \ell $, and consequently, the same holds for 
$ a_\ell^{[\ell]} (\tth) $ and
$ p_\ell^{[\ell]} (\tth) $. 
Due to analyticity, proving Theorems 
\ref{thm:main}-\ref{thm:main2} clearly 
reduce to show that 
$\beta_1^{(\tp)} (\tth) \neq 0 $ at some depth. 
The terms in the sum for
$\beta_1^{(\tp)} (\tth)  $  could cancel each other out, resulting in an identically zero function. 
Actually we demonstrate in Section \ref{infinitelim} 
that this occurs, for {\it any} $\tp \geq 2 $,  
in  the deep water limit $ \tth =  +\infty $. 

Thus we focus on the shallow water limit $ \tth \to 0^+ $. 
Surprisingly
in Section \ref{sec:limit0} 
we find that 
\begin{itemize}
\item   the limits
of the maximal 
Fourier-Taylor coefficients  
$ \eta_\ell^{[\ell]} (\tth) $, 
$ \psi_\ell^{[\ell]} (\tth) $ of the Stokes wave as $ \tth \to 0^+ $  can be computed using an inductive procedure, based on the iterative bifurcation structure of the Stokes wave, resulting in 
 hypergeometric quantities, 
see  \eqref{expetajjpsijj}, \eqref{indHP2}.
\end{itemize}
Thanks to these limits we deduce  the 
Laurent asymptotic expansion for {\it any} $ \ell $ of 
\begin{equation}
\label{pall}
\begin{aligned}
&p_\ell^{[\ell]} (\tth) =  -2\ell \Big(\frac38 \Big)^{\ell-1} \tth^{3-3\ell-\frac12} - \ell \Big(\frac38 \Big)^{\ell-1}\,  \frac{11\ell^2-9\ell+1}{9} \tth^{5-3\ell-\frac12}+O( \tth^{7-3\ell-\frac12})\, ,\\ 
&a_\ell^{[\ell]} (\tth) = - \ell \Big(\frac38 \Big)^{\ell-1} \tth^{2-3\ell} - \ell\Big(\frac38 \Big)^{\ell-1} \, \frac{31\ell^2-9\ell+2}{18}\tth^{4-3\ell} +O(\tth^{6-3\ell}) \, .
\end{aligned} 
\end{equation}
In the shallow water limit the formula 
\eqref{beta1exp}-\eqref{entanglementfdj}
for $ \beta_1^{(\tp)} (\tth )$ 
has remarkable simplifications since 
many addends are negligible, see Lemma 
\ref{ordinivincenti},
and, using \eqref{pall}, we finally deduce  
in  Section \ref{asymptotics}, 
the 
Laurent 
asymptotic expansion 
of
$ \beta_1^{(\tp)} (\tth) $ as $ \tth \to 0^+ $, 
\begin{equation}\label{beta.asy.intro}
\beta_1^{(\tp)} (\tth) =  
\underbrace{{\widetilde A}^{(\tp)}}_{= 0} \tth^{\tfrac72 - 3 \tp} + 
\underbrace{{\widetilde C}^{(\tp)}}_{<0} \tth^{\tfrac{11}{2} - 3 \tp} +  O( \tth^{\tfrac{15}{2} - 3 \tp} ) \quad 
\text{as} \quad  \tth \to 0^+
 \, .
\end{equation}
As we prove in 
Lemma \ref{lem:cozero} 
\begin{itemize}
\item 
the first order term
$ \widetilde A^{(\tp)} $ of such  expansion   {\it vanishes} 
for every $\tp$, see 
Appendix \ref{sec:canc}. 
\end{itemize}
That is why we  compute 
explicitly also the second order Laurent expansion term in \eqref{beta.asy.intro}. 
In this way, the initial problem  
of proving the 
non-degeneracy of the analytic functions
$ \beta_1^{(\tp) } (\tth) $  for any $ \tp \geq 2  $,
has been reduced to determining 
whether $  {\widetilde C}^{(\tp)} $,  
which is proportional to the 
combinatorial sum \eqref{Cp}, in non-zero. In conclusion

\begin{itemize}
\item 
the first terms of 
$ \widetilde C^{(\tp)}$ are easily proved to be proportional 
to $ \frac{\tp(\tp+1)^2}{3}$.
Computer assisted computations 
confirm this calculus up to, say, the first 
$ \tp \leq 20 $, 
suggesting that formula \eqref{primavera2024} 
actually  holds for any  $\tp$. 
This is proved rigorously by the   computer algebra methods in Zeilberger-van Hoeij-Koutschan \cite{vHZ}.  
We provide in  \cite{BCMVadd} 
a self-contained exposition of this approach.
\end{itemize}

The paper is organized as follows. In Section \ref{sec:HFI} 
we first present the water waves equations and in 
Theorem \ref{LeviCivita} the  structural properties
of the Taylor-Fourier 
expansion of the Stokes waves at {\it any} order. Then we state  
Theorem \ref{thm:finalmat}
that implies, with the results of Section \ref{Katoapp},   
Theorems 
\ref{thm:main}
and 
\ref{thm:main2}.
In Section \ref{sec:App31}
 we  prove  
Theorem \ref{LeviCivita}, 
and  the asymptotic limits  of 
the Stokes  waves 
as $ \tth \to 0^+ $ and 
$ \tth \to + \infty $. 
Section \ref{Katoapp} describes 
the 
splitting of the eigenvalues of
$ \cL_{\mu,\e} $. 
In Section  
\ref{matrixcomputation} we find  the expression of the function $ \beta_1^{(\tp)} (\tth) $ in terms of the Taylor-Fourier 
Stokes waves coefficients
and we prove the weak
conjecture \eqref{weakconj}. 
Finally in Section \ref{combinatoricsincoming} we 
show the limits \eqref{beta1plimit}
and prove the strong conjecture \eqref{strongconj}. 

\medskip

\noindent {\bf Notation}.
The symbol
$ a_\e \sim b_\e $ 
means that
$ a_\e b_{\e}^{-1} \to 1 $ as 
$ \e \to 0 $, whereas 
$ a_\e \propto b_\e $ means that $ a_\e \sim C b_\e   $ for some constant
$ C  $. 
 

\medskip
\noindent {\bf Acknowledgments.}
We thank W. Strauss, B. Deconinck and V. Hur for several insightful conversations about modulational instability. We thank D. Zeilberger, M. van Hoeij and C. Koutschan for the combinatorics of Lemma \ref{primasommma}.

\section{High frequency instabilities of Stokes wave}\label{sec:HFI}

We consider the Euler equations for a 2-dimensional  
incompressible 
and irrotational fluid under the action of  gravity which fills the time dependent region 
$$
{ \mathcal D}_\eta := \big\{ (x,y)\in \bT\times \bR\;:\; -\tth < y< \eta(t,x) \big\} \, , 
\quad  \bT :=\bR/2\pi\bZ  \, , 
$$ 
with  depth $\tth >0 $ 
and  space periodic boundary conditions. 
The irrotational velocity field is the gradient  
of the  scalar potential $\Phi(t,x,y) $  
determined 
as the unique harmonic solution of 
\begin{equation*}
    \Delta \Phi = 0 \  \text{ in }\, {\mathcal D}_\eta \, , \quad
     \Phi(t,x,\eta(t,x)) = \psi(t,x) \, ,  \quad
   \Phi_y(t,x, - \tth ) = 0  \, . 
\end{equation*}
The time evolution of the fluid is determined by two boundary conditions at the free surface. 
The first is that the fluid particles  remain, along the evolution, on the free surface   
(kinematic boundary condition), 
and the second one is that the pressure of the fluid  
is equal, at the free surface, to the constant atmospheric pressure 
 (dynamic boundary condition). 
Then, as shown by Zakharov \cite{Zak1} and Craig-Sulem \cite{CS}, 
the time evolution of the fluid is determined by the 
following equations for the unknowns $ (\eta (t,x), \psi (t,x)) $,  
\begin{equation}\label{WWeq}
 \eta_t  = G(\eta)\psi \, , \quad 
  \psi_t  =  
- g \eta - \dfrac{\psi_x^2}{2} + \dfrac{1}{2(1+\eta_x^2)} \big( G(\eta) \psi + \eta_x \psi_x \big)^2 \, , 
\end{equation}
where $g > 0 $ is the gravity constant and $G(\eta):=G(\eta,\tth)$ denotes 
the Dirichlet-Neumann operator 
$$
 [G(\eta)\psi](x) := \Phi_y(x,\eta(x)) -  \Phi_x(x,\eta(x)) \eta _x(x) \, . 
 $$ 
 With no loss of generality we set the gravity constant $g=1$.
The equations \eqref{WWeq} are the Hamiltonian system
$$
 \pa_t \vet{\eta}{\psi} = \cJ \vet{\nabla_\eta \mathscr{H}}{\nabla_\psi \mathscr{H}}, \quad \quad \cJ:=\begin{bmatrix} 0 & \uno \\ -\uno & 0 \end{bmatrix},
$$ 
 where $ \nabla $ denote the $ L^2$-gradient, and the Hamiltonian
$  \mathscr{H}(\eta,\psi) :=  \frac12 \int_{\mathbb{T}} \left( \psi \,G(\eta)\psi +\eta^2 \right) \de x
$
is the sum of the kinetic and potential energy of the fluid. 
 
In addition of being Hamiltonian, the water waves 
system \eqref{WWeq}  is time reversible with respect to the involution 
\begin{equation}\label{revrho}
\rho\vet{\eta(x)}{\psi(x)} := \vet{\eta(-x)}{-\psi(-x)}, \quad \text{i.e. }
\mathscr{H} \circ \rho = \mathscr{H} \, ,
\end{equation}
and it  is invariant under space translations. 
\\[1mm]{\bf Stokes waves.}
The nonlinear water waves equations \eqref{WWeq} 
admit 
periodic traveling  solutions --called Stokes waves-- 
$$
\eta(t,x)=\breve \eta(x-ct) \, , 
\quad \psi(t,x)= \breve \psi(x-ct) \,  , 
$$ 
 for some speed  $c$ (which depends on the amplitude of the wave) 
 and  $2\pi$-periodic profiles  $\breve \eta (x), \breve \psi (x) $. 
 Bifurcation of small amplitude Stokes waves 
  has been rigorously proved by Levi-Civita \cite{LC} and  
Nekrasov \cite{Nek} in infinite depth, 
and Struik \cite{Struik} for finite depth.
Currently, this can be derived through a direct  application
of the Crandall-Rabinowitz bifurcation theorem from a  simple eigenvalue as in  \cite{BMV2}. 
 The speed of  Stokes waves  of small amplitude  is close to the one predicted  by the linear theory,  
  $$
 \ch := \sqrt{\tanh(\tth)} \,  . 
$$  
To establish the existence of all instability isolas of Stokes waves, we require structural insights into their Taylor expansion at every order. The knowledge of merely a finite number of Taylor coefficients is clearly insufficient for this purpose.
We introduce the following definitions. 

\begin{sia}\label{def:even}
Let $\ell \in \bN$. We denote by $ \mathtt{Evn}_\ell $ 
the vector space of real trigonometric polynomials of the form
 \begin{equation}\label{evenodd}
f(x) =  
\begin{cases}
f^{[0]}+ f^{[2]}\cos(2x)+\dots+ f^{[\ell]}\cos(\ell\, x) & \text{if }\ell \text{ is even} \, , \\
f^{[1]}\cos(x)+ f^{[3]}\cos(3x)+\dots+ f^{[\ell]}\cos(\ell\,x)  & \text{if }\ell\text{ is odd} \, ,  \end{cases} 
\end{equation}
with real coefficients $ f^{[i]}$, $i = 0 , \ldots, \ell $ and by 
$ \mathtt{Odd}_\ell$ the vector space of real trigonometric polynomials of the form
\begin{equation}\label{evenodd2}
 g(x) =  \begin{cases} g^{[2]}\sin(2x) + g^{[4]}\sin(4x)+\dots+ 
 g^{[\ell]}\sin(\ell\,x) & \text{if }\ell\text{ is even} \, ,  \\
g^{[1]}\sin(x)+ g^{[3]}\sin(3x)+\dots+ g^{[\ell]}\sin(\ell\,x)  & \text{if }\ell\text{ is odd} \, ,  
 \end{cases}
 \end{equation}
with real coefficients $ g^{[i]}$, $i = 1, \ldots, \ell $. 
\end{sia}
Note that a function $f\in\mathtt{Evn}_\ell$, resp. $ \mathtt{Odd}_\ell$,  
is not just an even, resp. odd, 
trigonometric real  polynomial of degree less or equal to $ \ell $, 
 but  it possesses only  harmonics of the  {\it same}  parity of  $\ell$. 
 
The coefficients of the Stokes wave are found to be analytic functions of the depth  $ \tth >0 $, exhibiting a structure captured by the following definition.

 \begin{sia}\label{rationalstructure}
An analytic function $g:(0,+\infty)\to \bR$ belongs to
$ {\cal Q}(\ch^2) $ if there exist polynomials $p(x)$
and $ q(x)$  with {\it integer} coefficients  and without common factors, such that $q(x)\neq 0$ for any $0<x\leq 1$,  and
\begin{equation}\label{rationalfunctions}
g(\tth)  = \frac{p(\ch^2)}{q(\ch^2)} \,,\  \forall \tth >0 \, 
 ,\qquad \ch^2  = \tanh(\tth)\, . 
\end{equation}
Note that $\lim\limits_{\tth \to +\infty} g(\tth) = \frac{p(1)}{q(1)}\in \bQ $. The 
set ${\cal Q}(\ch^2)$ is closed under the product of functions. 
\end{sia}

The subsequent result provides a detailed  description
of the Stokes waves at {\it any} order in $ \e $.

 \begin{teo}{\bf (Stokes waves)}\label{LeviCivita} 
For any $\tth_* > 0 $  there exist $\e_* :=\e_*(\tth_* )  >0 $ and for any $\tth \geq \tth_*$ 
 a unique family  of real analytic solutions $(\eta_\e(x), \psi_\e(x), c_\e)$, 
defined for any  $|\e|<\e_* $, of 
\begin{equation}\label{travelingWW}
\begin{cases}
 c \, \eta_x+G(\eta,\tth)\psi = 0 \, , \\
 c \, \psi_x -  \eta - \dfrac{\psi_x^2}{2} + \dfrac{1}{2(1+\eta_x^2)} \big( G(\eta, \tth) \psi + \eta_x \psi_x \big)^2  = 0 \, , 
\end{cases}
\end{equation}
such that
 $\eta_\e (x) $ and $\psi_\e (x) $ are $2\pi$-periodic,  $\eta_\e (x) $ is even and $\psi_\e (x) $ is odd, of the form
\begin{equation}\label{Stokespwseries}
\begin{aligned} 
&  \eta_\e(x) = \e \cos(x) + 
{\mathop \sum}_{\ell \geq 2} \e^\ell \eta_\ell (x)\, , \\
& \psi_\e(x) =  \e\ch^{-1}
 \sin(x) + {\mathop \sum}_{\ell \geq 2} \e^\ell \psi_\ell(x)\, , \\ 
& c_\e = \ch+
{\mathop \sum}_{\ell \geq 2, \ell \textup{ even}} 
\!\!\!\! \e^\ell c_\ell \quad \text{with} \ \ch \ \text{defined in} \ \eqref{rationalfunctions} \, . 
\end{aligned}
\end{equation}
 The functions
$ \eta_\e (x), \psi_\e (x)  $ and the speed
$ c_\e $ 
depend analytically w.r.t. $ \tth > 0 $.  
Furthermore the following properties hold: for any  $ \ell\in\bN $, the Taylor coefficients $  \eta_\ell(x), \psi_\ell(x) $ are trigonometric polinomials
\begin{equation}
\begin{aligned} \label{leadingStokes}
&\text{1. $  \eta_\ell(x) \in \mathtt{Evn}_\ell $
and $ \psi_\ell(x) \in \mathtt{Odd}_\ell $
according to Def. \ref{def:even}; }\\ 
& \text{2. the coefficients 
 $  \eta_\ell^{[0]}\, ,\dots, \eta_\ell^{[\ell]} $,  
 $ \ch \psi_\ell^{[1]}\,,\dots \ch \psi_\ell^{[\ell]} $ 
 and 
 $ \ch c_\ell $  belong to
 $ {\cal Q}(\ch^2) $, see Def. \ref{rationalstructure}.} 
  \end{aligned}
  \end{equation}
 The  $ \eta_\ell^{[i]}, \psi_\ell^{[i]} $ are referred
 as the Fourier-Taylor coefficients of the Stokes wave. 
\end{teo}

The proof of the structural properties \eqref{leadingStokes} of the Stokes waves is given in  Section  
\ref{sec:App31}.


Proposition \ref{lem:asympetapsi} will provide
further properties  
of  Stokes waves at {\it any} order of their Taylor expansion.
Specifically, we will demonstrate inductively the
second-order asymptotic expansion in the shallow-water limit 
$ \tth \to 0^+ $ of the Taylor-Fourier 
coefficients $ \eta_\ell^{[\ell]}, \psi_\ell^{[\ell]} $ 
valid for any $\ell \in \bN$. 
This latter result is crucial in proving that $ \beta_1^{(\tp)}(\tth)
\to - \infty $ as $ \tth \to 0^+ $. 

\begin{rmk}
The local branches of Stokes wave solutions 
of Theorem \ref{LeviCivita} have  been  
extended 
to  global branches encompassing  extreme 
waves, see e.g. \cite{KN,To},  
through the application of analytic and topological bifurcation theory. More recently,  
the existence of quasi-periodic  traveling Stokes waves --which 
represent the nonlinear superposition of periodic Stokes waves
with rationally independent speeds--
 has also been demonstrated 
 in \cite{BFM1,FG,BFM2}  using  KAM methods. 
\end{rmk}

\noindent{\bf Linearization at the Stokes waves.}
In a reference frame moving with the speed 
$ c_\epsilon $,  
the linearized water waves equations 
at the Stokes waves turn out to be
(after conjugating with the ``good unknown of Alinhac" and the ``Levi-Civita" transformations as in \cite{NS,BMV1,BMV2,BMV_ed}) 
the linear system 
\begin{equation}\label{linWW}
h_t = \cL_\e h   
\end{equation}
where 
$\cL_\e:  H^1(\bT,\bR^2) 
\to L^2(\bT, \bR^2) $ 
is the Hamiltonian and reversible real operator
\begin{equation}
\begin{aligned}
\label{cLepsilon}
\cL_\e := \cL_\e(\tth)   
& :  =  
\begin{bmatrix} \pa_x \circ (\ch+p_\e(x)) &  |D|\tanh((\tth+\mathtt{f}_\e) |D|) \\ - (1+a_\e(x)) &   (\ch+p_\e(x))\pa_x \end{bmatrix}   \\
& = \cJ  \begin{bmatrix}   1+a_\e(x) &   -(\ch+p_\e(x)) \pa_x \\ 
\pa_x \circ (\ch+p_\e(x)) &  |D|\tanh((\tth+\mathtt{f}_\e) |D|)  \end{bmatrix} 
\end{aligned}
\end{equation} 
where  $\ttf_\e$ is a real number, analytic in $ \e $,  and
$p_\e (x), a_\e (x) $ are real  even analytic functions.
As a corollary of the properties 
\eqref{leadingStokes}
of the Stokes waves we prove in Section \ref{sec:App31}
the following
lemma.

\begin{lem}\label{pastructure}
The functions $p_\e(x)$,  $a_\e(x)$  and the constant $\ttf_\e$ in \eqref{cLepsilon} admit the Taylor expansions
 \begin{equation}\label{leadingLin}
 \begin{aligned}
& p_\e(x) = \sum_{\ell \geq 1} \e^\ell p_\ell(x)\, , \quad a_\e(x) = \sum_{\ell \geq 1} \e^\ell a_\ell(x)\, ,\quad
 \ttf_\e = \sum_{\substack{\ell \geq 2\,,\; \ell\textup{ even}}} \e^\ell \ttf_\ell \, ,
\end{aligned}
\end{equation}
they  are analytic with respect to $ \tth > 0 $ and,   for any  $  \ell\in\bN $, 
   \begin{enumerate}
 \item 
\label{leadingStokesp1}
 the functions $ p_\ell(x) $,
$  a_\ell(x) $ are trigonometric polynomials in
$  \mathtt{Evn}_\ell $ (see Definition \ref{def:even}); 
  \item the coefficients 
$\ch p_\ell^{[0]}, \ldots, \ch p_\ell^{[\ell]}  $,   
  $ a_\ell^{[0]}, \ldots, a_\ell^{[\ell]} $ and 
$\ttf_\ell  $ belong to 
 $ {\cal Q}(\ch^2) $  (see Definition \ref{rationalstructure}).
   \label{leadingStokesa1}
  \end{enumerate}
\end{lem}

\noindent{\bf Bloch-Floquet analysis.}\label{Bloch-Floquet}
In view of  \eqref{BFtheory}  
we have to determine the $ L^2 (\bT) $-spectrum of the  Floquet operator $ \cL_{\mu,\e} $. It turns out to be  the complex  \emph{Hamiltonian} and \emph{reversible} operator
\begin{align}\notag 
 \cL_{\mu,\e}   & := \begin{bmatrix} (\pa_x+\im\mu)\circ (\ch+p_\e(x)) & 
 |D+\mu| \tanh\big((\tth + \ttf_\e) |D+\mu| \big) \\ -(1+a_\e(x)) & (\ch+p_\e(x))(\pa_x+\im \mu) \end{bmatrix} \\ 
 &= \underbrace{\begin{bmatrix} 0 & \uno\\ -\uno & 0 \end{bmatrix}}_{\displaystyle{=\cJ}} \underbrace{\begin{bmatrix} 1+a_\e(x) & -(\ch+p_\e(x))(\pa_x+\im \mu) \\ (\pa_x+\im\mu)\circ (\ch+p_\e(x)) & |D+\mu| \tanh\big((\tth + \ttf_\e) |D+\mu| \big) 
 \end{bmatrix}}_{\displaystyle{=:
 \cB(\mu,\e)}=\cB (\mu,\e)^*  }   \, ,   
 \label{WW}
\end{align} 
that we regard  as an operator with 
domain $H^1(\bT):= H^1(\mathbb{T},\bC^2)$ and range $L^2(\bT):=L^2(\mathbb{T},\bC^2)$, 
equipped with  
the complex scalar product 
\begin{equation}\label{scalar}
(f,g) := \frac{1}{2\pi} \int_{0}^{2\pi} \left( f_1 \bar{g_1} + f_2 \bar{g_2} \right) \, \text{d} x  \, , 
\quad
\forall f= \vet{f_1}{f_2}, \ \  g= \vet{g_1}{g_2} \in  L^2(\bT, \bC^2) \, .
\end{equation} 
We also  denote $ \| f \|^2 = (f,f) $. The  
complex space $ L^2 (\bT, \bC^2) $ is equipped 
with the sesquilinear,  skew-Hermitian and non-degenerate 
{\it complex symplectic form}  
\begin{equation}\label{ses}
{\cal W}_c  \, \colon L^2 (\bT, \bC^2) \times L^2 (\bT, \bC^2) \to \bC \, , \quad  {\cal W}_c(f, g) := (\cJ f,g) \, .
\end{equation}
The complex operator 
$\cL_{\mu,\e} $ 
is also reversible, namely  
  \begin{equation}\label{Reversible}
 \cL_{\mu,\e} \circ \bro = \bro \circ \cL_{\mu,\e} \, , \qquad \text{equivalently}
 \qquad 
 \cB (\mu, \e) \circ \bro =- \bro \circ \cB (\mu,\e) \, , 
\end{equation} 
where $\bro$ is the complex involution (cfr. \eqref{revrho})
\begin{equation}\label{reversibilityappears}
 \bro \vet{\eta(x)}{\psi(x)} := \vet{\bar\eta(-x)}{-\bar\psi(-x)} \, .
\end{equation}
Since the functions $\e \mapsto a_\e$, $p_\e$ 
 are analytic as maps from $B(\e_0) \to H^1(\bT)$ 
and $\e \mapsto {\mathtt f}_\e$ 
is analytic as well,
the operator $ \cL_{\mu,\e}$ in  \eqref{WW} is analytic in $ (\mu,\e) $. 

\begin{rmk}
\label{analytic}
The operator $ \cL_{\mu,\e}$
is analytic also w.r. to $ \tth > \tth_*  $
for any $ \tth_* >0$ 
provided $\e $ is small 
depending on $ \tth_* $. 
The key claim is that 
the map 
$ \tth \mapsto \tanh (\tth |D|) \in {\cal L} (H^s_0, H^s_0)  
$ 
is analytic, for any $ s \in \bR $. 
Since 
$ \| 
\tanh (\tth |D|) 
\|_{{\cal L}(H^s_0, H^s_0)} = \sup_{j \in \bZ}
|\tanh (\tth |j|)| $ 
this amounts to prove the 
analyticity of the map 
$$ 
\varphi: U_\delta \to \ell^\infty (\bZ)  \, , \quad 
\tth \mapsto \varphi (\tth) := 
\big( \tanh (\tth |j|)\big)_{j \in \bZ}  \in \ell^\infty (\bZ) \, , 
$$
in a complex neighborhood 
$ U_\delta := \{ z = x+ \im y \, , |x-\tth_0| < \delta, |y| < \delta \} $ of any $ \tth_0 > 0 $. 
By  Theorem A.3 in \cite{KAMKdV}
 it is sufficient to show that 
 $ \varphi $ 
is bounded 
and each component  
$  
 \tanh (\tth |j|) $
 is analytic in $ U_\delta $ for some
$ \delta > 0 $ small enough. This holds true because, for any  
$ x > \tth_0/ 2 $ we have  
$$
|\tanh ( z |j|)| = 
 \frac{|1-e^{-2(x+\im y)|j|}|}{|1+e^{-2(x+\im y)|j|}|}
 \leq 
 \frac{1+e^{-2x|j|}}{
 1-e^{-2x|j|}
 } \leq \frac{2}{
 1-e^{-\tth_0}} \, ,
 \quad \forall j \in \bZ \, .
$$
\end{rmk}

Our objective is to demonstrate the existence of eigenvalues of $  \cL_{\mu,\e}  $ 
with non zero real part. 
Due to the Hamiltonian structure of $\cL_{\mu,\e}$, such eigenvalues can only emerge  from multiple
  eigenvalues of $\cL_{\mu,0}$, since   if $\lambda$ is an eigenvalue of $\cL_{\mu,\e}$ then  also $-\bar \lambda$ is. 
Consequently, simple purely imaginary eigenvalues of $\cL_{\mu,0}$ persist on the imaginary axis under perturbation.
\\[1mm]
\noindent{\bf The spectrum of $\cL_{\mu,0}$.}\label{initialspectrum} 
The spectrum of  the  Fourier multiplier matrix operator 
\begin{equation}\label{cLmu}
 \cL_{\mu,0} = 
 \begin{bmatrix} 
 \ch (\pa_x+\im\mu)  & |D+\mu|\tanh(\tth |D+\mu|) \\ -1 & \ch(\pa_x+\im\mu) \end{bmatrix} 
\end{equation}
on $L^2(\bT,\bC^2)$ is given by the purely imaginary numbers
\begin{equation}\label{omeghino}
\lambda_j^\sigma(\mu,\tth):=
\im \omega^\sigma(j+\mu,\tth)=  \im \Big( \ch(j+\mu) -\sigma \Omega(j+\mu, \tth)  \Big)  \, ,
\  \forall \sigma   = \pm \, , \,  j\in \bZ \, , 
\end{equation}
where  
\begin{equation}\label{Oomegino}
\omega^\sigma(\varphi,\tth) :=  \ch \varphi -\sigma \Omega(\varphi,\tth) \, ,  \qquad 
 \Omega(\varphi,\tth) :=\sqrt{\varphi\tanh(\tth \varphi)}  \, , 
\qquad   \varphi \in \bR\,  ,\ \ \sigma=\pm\, .
\end{equation}
Note that $\omega^+(-\varphi,\tth) = -\omega^-(\varphi,\tth) $. Sometimes   
$ \sigma $ is called the Krein signature of the eigenvalue $ \lambda_j^\sigma(\mu,\tth) $.

Restricted to  $\varphi > 0 $, the function $\Omega(\cdot,\tth) : (0,+\infty) \to (0,+\infty) $  is,
for any $ \tth>0 $,     
  increasing and concave:  
\begin{equation}\label{boomerang}
\begin{aligned}(\pa_\varphi\Omega)(\varphi,\tth) &= \dfrac{\tanh(\tth \varphi) +\tth \varphi \big(1-\tanh^2(\tth \varphi) \big)}{2\sqrt{\varphi \tanh(\tth \varphi)}} >0 \, , \\
(\pa_ {\varphi\varphi}\Omega) (\varphi,\tth) &= -\dfrac{\tanh ^2(\tth \varphi) \big[4 \tth^2 \varphi^2 \big(1-\tanh^2(\tth \varphi)\big)+\big(1- \tth \varphi \frac{1-\tanh^2( \tth \varphi)}{\tanh( \tth \varphi)}\big)^2 \big]}{4 (\varphi \tanh (\tth \varphi))^{3/2}} < 0\, .
\end{aligned}
\end{equation}
For any $ j \in \bZ $, $ \mu \in \bR $ such that $ j + \mu \neq 0 $
we have $ \Omega (j+ \mu, \tth)  \neq 0 $  and we associate 
to  the eigenvalue $\lambda_j^\sigma(\mu,\tth)$ the eigenvector
\begin{equation}\label{def:fsigmaj}
\begin{aligned}
& f^\sigma_j:= f^\sigma_j(\mu,\tth) := \frac{1}{\sqrt{2\Omega(j+\mu,\tth)}} \vet{-\sqrt{\sigma}\,\Omega(j+\mu,\tth)}{\sqrt{-\sigma}}e^{\im j x}\, , \\
&  \qquad \qquad \cL_{\mu,0} f_j^\sigma(\mu,\tth) = \lambda_j^\sigma(\mu,\tth)  f_j^\sigma(\mu,\tth) \, , 
\end{aligned}
\end{equation} 
which satisfy, recalling  \eqref{reversibilityappears},
 the reversibility property   
\begin{equation}\label{baserev}
\bar \rho f_j^+ (\mu,\tth) = f_j^+ (\mu,\tth)  \, , \quad 
\bar \rho f_j^- (\mu,\tth) = - f_j^- (\mu,\tth)  \, . 
\end{equation}
For any $\mu \notin  \bZ $ we have $\Omega(j+\mu,\tth)\neq 0$ for any  $j \in \bZ $ 
and the family of eigenvectors $ \{f_j^{\sigma} (\mu )\}_{j \in \bZ}$ in \eqref{def:fsigmaj} forms a {\it complex symplectic basis} of  $L^2(\bT,\bC^2)$, 
 with respect to the complex symplectic form ${\cal W}_c$ in \eqref{ses}, namely its elements are linearly independent, span densely $L^2(\bT,\bC^2) $ and satisfy, 
for any  pairs of integers $ j ,j' \in \bZ $, and any pair of signs $ \sigma, \sigma' 
\in \{ \pm \}$, 
\begin{equation}\label{sympbas}
{\cal W}_c \big( f_j^\sigma, f_{j'}^{\sigma'} \big) = \begin{cases}  -\im & \text{if } j=j'\text{ and }\sigma=\sigma'=+ \, , \\ 
\ \; \im & \text{if } j=j'\text{ and }\sigma=\sigma'=- \, , \\
\ \, 0 & \text{otherwise} \, .  \end{cases} 
\end{equation} 
The choice of the normalization constant in \eqref{def:fsigmaj} 
implies \eqref{sympbas}.

To describe  the multiple 
eigenvalues of $\cL_{\mu,0} $ we need the following result
proved in Appendix \ref{spectralcollisions}.

\begin{lem}[{\bf Spectral collisions}]\label{collemma}
For any depth $ \tth > 0 $ and any integer $\tp\geq 2 $,  the following holds.
\begin{itemize}
\item[i)] For any $ \sigma \in \{\pm  \} $ 
the  equations 
$ \omega^\sigma \big( \varphi ,\tth \big) = \omega^\sigma \big( \varphi+\tp,\tth \big) $
have  no solutions, see Figure \ref{fig.nocollision}.
\item[ii)] The equation 
$\omega^- \big( \varphi ,\tth \big) = \omega^+ \big( \varphi +\tp,\tth \big)    $
has a  unique positive solution $ \uphi(\tp,\tth) > 0 $,  
\begin{equation}\label{Kreincollision}
\omega^- \big( \uphi(\tp,\tth),\tth \big) = \omega^+ \big( \uphi(\tp,\tth)+\tp,\tth \big) 
=  \ch  \uphi(\tp,\tth) + \Omega( \uphi(\tp,\tth))  
=:  \omega_*^{(\tp)}(\tth) 
> 0 \, , 
\end{equation} 
see Figure \ref{fig.Kreincollision}, 
and the negative solution $ - \uphi(\tp,\tth) -\tp$, 
\begin{equation}\label{negativesolution}
 \omega^-\big(-\uphi(\tp,\tth)-\tp,\tth\big) =  \omega^+\big(-\uphi(\tp,\tth),\tth\big) =
- \omega_*^{(\tp)}(\tth) < 0 \, ,
 \end{equation}
 which is unique for any $\tp\geq 3 $.  For $\tp= 2 $ the equation 
 $\omega^- \big( \varphi ,\tth \big) = \omega^+ \big( \varphi +2,\tth \big)  $ 
 possesses also 
 the negative solution $ \varphi = - 1 $ and 
 $\omega^- \big( -1 ,\tth \big) = \omega^+ \big(1,\tth \big)  = 0  $.
 \item[iii)]
 The equation $ \omega^+(\varphi,\tth) = \omega^-(\varphi+\tp,\tth) $  has no solutions. \end{itemize} 
  The positive solution  $  \uphi(\tp,\tth)  $ of \eqref{Kreincollision}
  satisfies the following properties: 
\begin{enumerate}
\item  \label{it1}
 The function $ \uphi(\tp, \cdot ) : (0, + \infty ) \to  (0, + \infty )  $, 
 $ \tth \mapsto \uphi(\tp,\tth) $,  is analytic. 
\item  \label{it2}
For any $ \tth > 0 $ the sequence $\tp\mapsto \uphi(\tp,\tth) $ is increasing 
\begin{equation}\label{uallinfi}
0 < \uphi(2,\tth) < \ldots < \uphi(\tp,\tth) < \ldots  \qquad
\text{and} \qquad \lim_{\tp\to +\infty} \,  \uphi(\tp,\tth) =  +\infty \, . 
\end{equation}
\item   \label{it3}
For any $ \tth > 0 $ it results 
$$
\begin{cases}
\uphi(2,\tth) \in (0,1) 
\\
 \uphi(\tp,\tth) \in \big( \uphi(\tp-1,\tth)  , \frac{(\tp-1)^2}{4} \big) \, , \ \forall\tp\geq 3 \, . 
\end{cases}
$$
 Furthermore there is $ \tth_* > 0 $ such that
  $ \uphi(2,\tth) > 1/ 4 $ for any $ \tth > \tth_*   $.  
\item  \label{it4}
It results 
 \begin{equation}\label{expuphinuova}
\lim_{\tth \to 0^+} \uphi(\tp,\tth) = 0 \, , \qquad 
\lim_{\tth \to + \infty} \uphi(\tp, \tth ) = \tfrac14 (\tp-1)^2\, ,
\end{equation}
more precisely 
 \begin{equation}\label{expuphinuova1}
\uphi(\tp,\tth) = \tfrac{1}{12}\tp (\tp^2-1) \tth^2 + \frac{19 -15 \tp^2-4 \tp^4}{720} \tp\tth^4+O(\tth^6)
\qquad \text{as} \  \tth \to 0 ^+ \, , 
\end{equation}
and 
\begin{equation}\label{expaphip}
\uphi(\tp,\tth) = \frac{(\tp-1)^2}{4} + 
\begin{cases}  
\frac38 e^{-\frac{\tth}{2}} + o\big(e^{-\frac{\tth}{2}}\big) & \text{if} \ \tp= 2 \, ,  \\
- \frac83 e^{- 2 \tth} + o\big(e^{- 2 \tth}\big) & \text{if} \ 
\tp = 3 \, ,  \\
 - \frac{\tp^2-1}{2} e^{- 2 \tth } + o\big(e^{-2\tth}\big)  &\;  \forall\,\tp\geq 4 \, ,
\end{cases}
\qquad\text{as} \ \tth\to+\infty \, .
\end{equation}
\end{enumerate}
\end{lem}  

\begin{rmk}\label{rem:BF} For $\tp=0,1$, the equations 
$\omega^\sigma \big( \varphi ,\tth \big) = \omega^{\sigma'} \big( \varphi +\tp,\tth \big)    $ 
 admit other solutions {\it all} at the common 
 value $ 0 $  (follow the proof of Lemma \ref{collemma}). 
 This gives rise to  the quadruple zero eigenvalue   of $ \cL_{\mu,0} $
from which the  Benjamin-Feir instability originates,  analyzed in \cite{BMV1,BMV3,BMV_ed}. 
\end{rmk}

As a corollary of Lemma \ref{collemma} we deduce 
in Appendix \ref{spectralcollisions} 
the complete description of {\it all} 
the  multiple {\it nonzero} eigenvalues of $\cL_{\mu,0} $. 

\begin{lem}[{\bf Multiple eigenvalues of $\cL_{\mu,0} $ away from $0$}]\label{thm:unpert.coll}   For any depth $ \tth > 0 $ and any $\mu \in \bR $, the spectrum of $\cL_{\mu,0}$ away from zero contains only simple or double eigenvalues: for 
any  integer $\tp\geq 2 $, 
\begin{enumerate}
        \item \label{iteml1} for any 
        $ \mu \neq \uphi(\tp,\tth) $, \text{mod} 1,   all the eigenvalues of 
        the operator $\cL_{\mu,0}$ in \eqref{cLmu} are simple;
    \item  \label{iteml2} for any 
   $  \mu = \uphi(\tp,\tth) $, \text{mod} $1$,
the operator  $\cL_{\mu,0}$ 
has the non-zero double eigenvalue
\begin{equation}\label{intcollision}
\im\omega_*^{(\tp)}(\tth)
= \im \omega^-\big(\uphi(\tp,\tth),\tth\big) = \im \omega^+\big(\uphi(\tp,\tth)+\tp,\tth\big)  
 \, ,
\end{equation}
and
the operator  $\cL_{-\mu,0}$  has the complex conjugated non-zero double eigenvalue
$$
 -\im \omega_*^{(\tp)}(\tth) = \im \omega^-\big(-\uphi(\tp,\tth)-\tp,\tth\big) = \im \omega^+\big(-\uphi(\tp,\tth),\tth\big)\, .
$$
The sequence $\tp\mapsto \omega_*^{(\tp)}(\tth) $ is increasing 
\begin{equation}\label{intcollisioncre}
0 < \omega_*^{(2)}(\tth) < \ldots <  \omega_*^{(\tp)}(\tth) < \ldots \qquad 
\text{and} \qquad \lim_{\tp \to + \infty} \omega_*^{(\tp)}(\tth) = + \infty \, . 
\end{equation}
The eigenspace of $ \cL_{\mu,0} $ associated with the 
double eigenvalue $ \im\omega_*^{(\tp)}(\tth) $ in \eqref{intcollision} is spanned 
by the two 
eigenvectors 
\begin{equation}\label{autovettorikernel}
\begin{aligned}
& 
f_{k}^-(\mu,\tth) = 
\frac{1}{\sqrt{2\Omega(\uphi(\tp,\tth),\tth)}} \vet{- \im \,\Omega(\uphi(\tp,\tth),\tth)}{1} e^{\im k x} \, , \\
& f_{k'}^+(\mu,\tth) =
\frac{1}{\sqrt{2\Omega(\uphi(\tp,\tth)+\tp,\tth)}} \vet{- \Omega(\uphi(\tp,\tth)+\tp,\tth)}{\im}
e^{\im k' x}   \, , 
\end{aligned}
\end{equation}
where $k:= \uphi(\tp,\tth)-\mu \in \bZ $ and $k':=k+\tp $.
\end{enumerate}
\end{lem}

\begin{figure}[h!]\centering \subcaptionbox*{}[.4\textwidth]{\includegraphics[width=5.5cm]{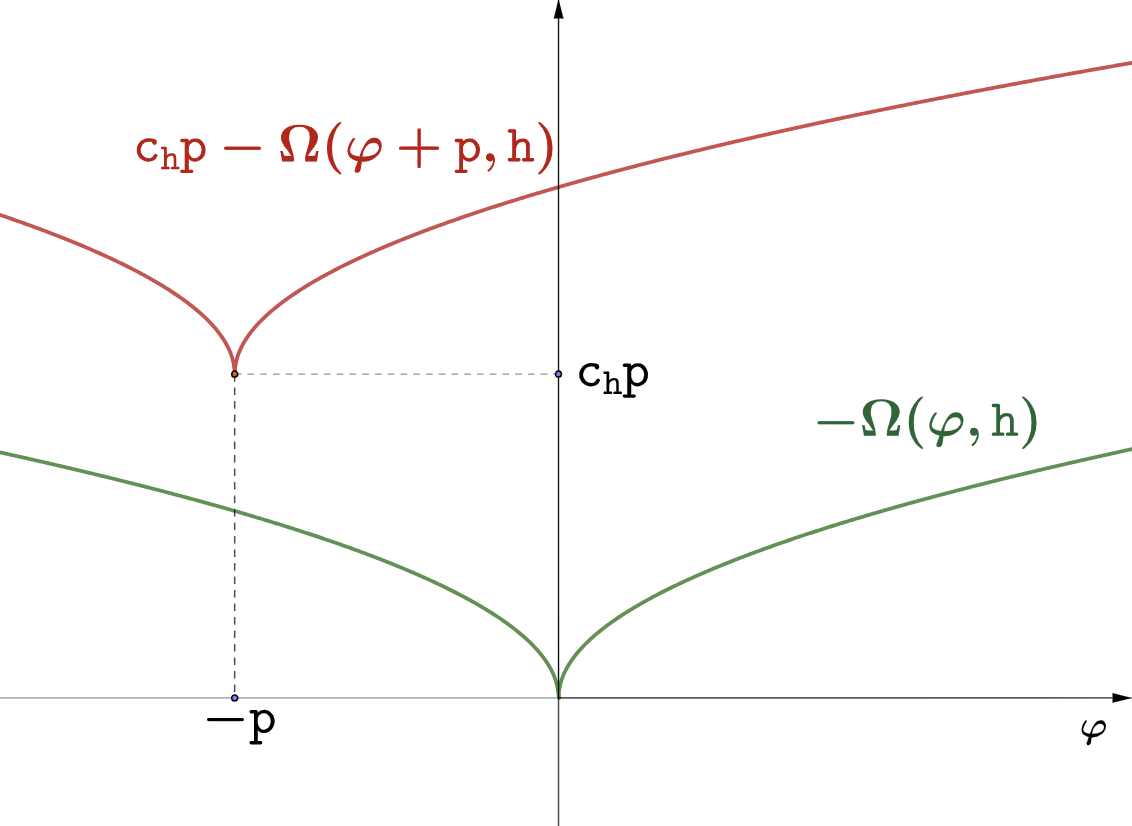}}\hspace{1cm}
\subcaptionbox*{}[.4\textwidth]{
\includegraphics[width=5.5cm]{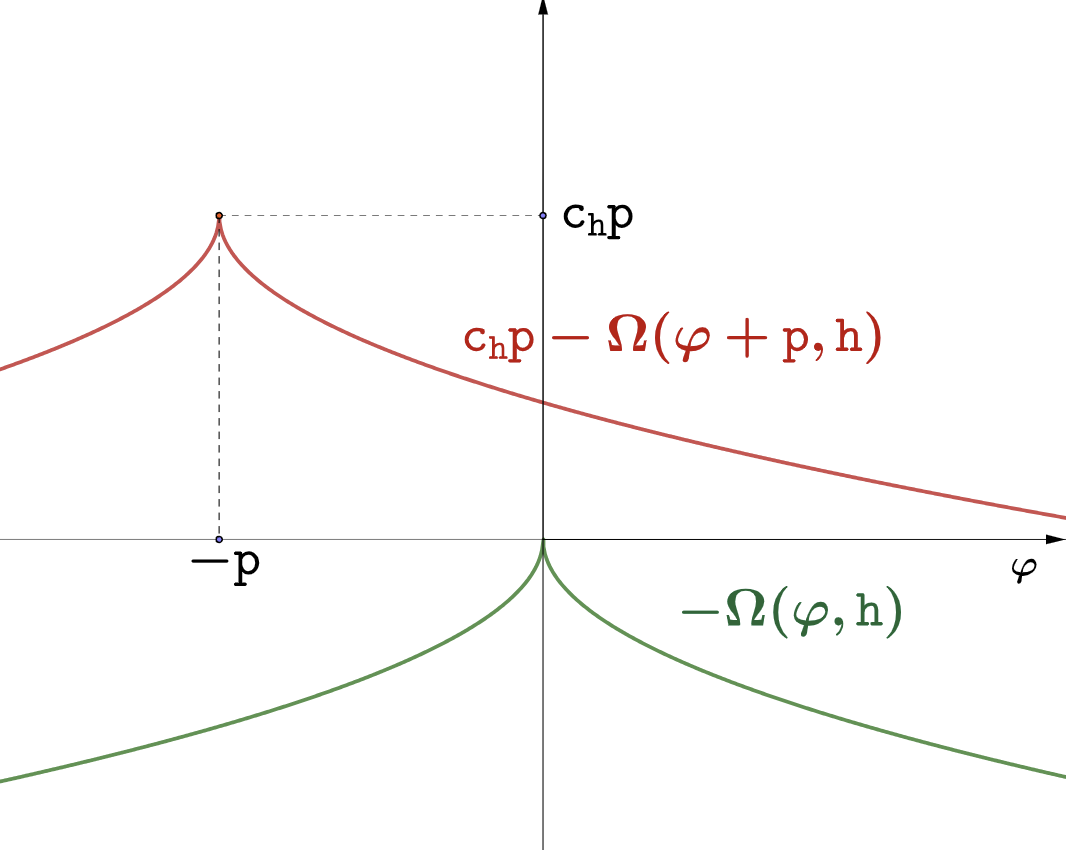} }
 \vspace{-0.8cm}
\caption{ \label{fig.nocollision} Absence of collisions for eigenvalues with same Krein signature for any $\tp\geq 2 $.}
\end{figure}
\begin{figure}[h]\centering
\includegraphics[width=7.5cm]{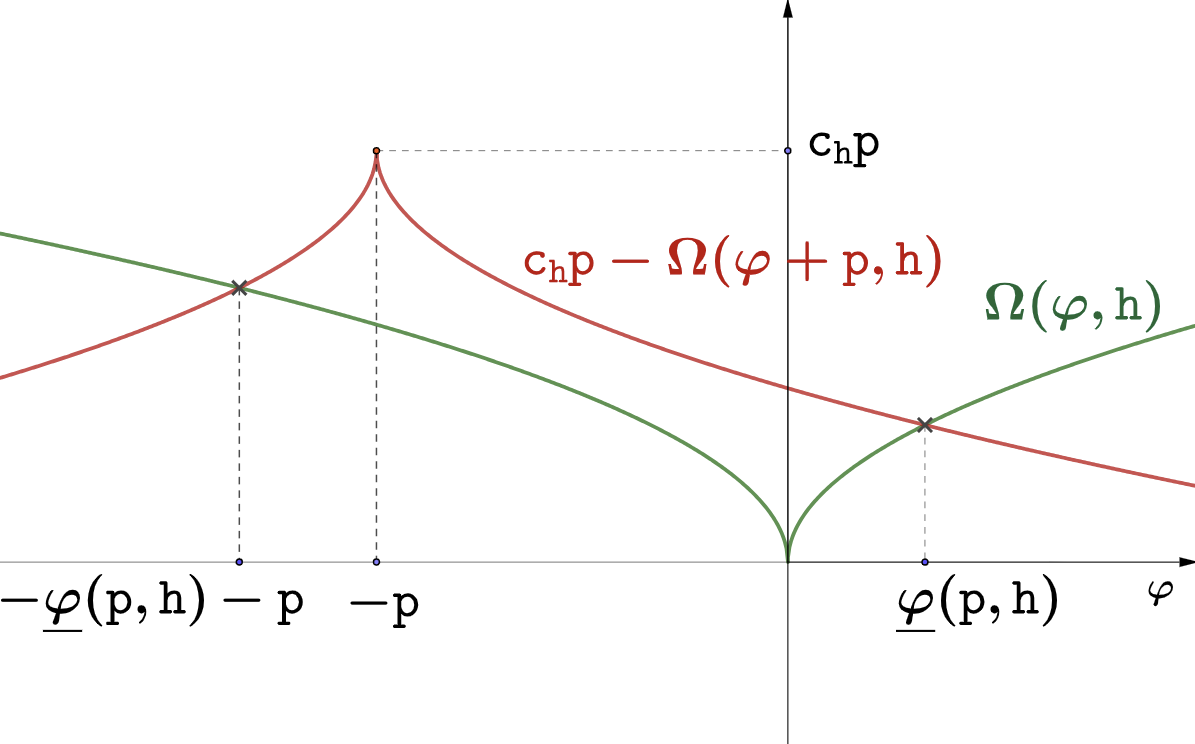} 
\caption{ \label{fig.Kreincollision} Collision of eigenvalues with opposite Krein signature for
$\tp\geq 2 $. 
}
\end{figure}

In view of  Lemma \ref{thm:unpert.coll}, {\bf given an  integer $\tp\geq 2$, we  fix} 
\begin{equation}\label{constants}
\umu := \uphi(\tp,\tth)   \, , \quad  k := 0 \, ,  \quad k':= \tp \, ,\quad 
\omega_*^{(\tp)}(\tth)  := \omega^-(\uphi(\tp,\tth),\tth)=\omega^+(\tp+\uphi(\tp,\tth),\tth)\, .
\end{equation}

\begin{rmk}
\label{rmk:depth}
We will consider depths $ \tth > 0 $ 
such that $ \umu = \uphi(\tp,\tth) \notin \bZ $,
so that 
the family of eigenvectors 
$ \{f_j^{\sigma} (\umu )\}_{j \in \bZ}$ in 
\eqref{def:fsigmaj}  form a 
complex symplectic basis of $L^2(\bT,\bC^2)$.
This is  convenient for the computations in Section 
\ref{Sec:entcoeff}, as we outline at the beginning of that section, without introducing any restrictions.
\end{rmk}

The spectrum of $\cL_{\umu,0}$ decomposes into two disjoint parts
\begin{equation}
\label{spettrodiviso0}
\sigma (\cL_{\underline{\mu},0}) = \sigma_\tp' (\cL_{\underline{\mu},0}) \cup \sigma_\tp'' (\cL_{\underline{\mu},0})  \qquad \text{where }  \quad \sigma_\tp'(\cL_{\underline{\mu},0}):= \big\{\im\omega_*^{(\tp)}(\tth) \big\} \, . 
\quad  \end{equation}
By Kato's perturbation theory (see  Lemma \ref{lem:Kato1} below)
for any $ (\mu, \e) $  sufficiently close to $ (\umu, 0) $, the perturbed spectrum
$\sigma\left(\cL_{\mu,\e}\right) $ still 
admits a disjoint decomposition  
$$
\sigma\left(\cL_{\mu,\e}\right) = \sigma_\tp'\left(\cL_{\mu,\e}\right) \cup \sigma_\tp''\left(\cL_{\mu,\e}\right) \, , \quad \forall\tp\geq 2 \, , 
$$
where $ \sigma_\tp'\left(\cL_{\mu,\e}\right) 
$  consists of 2 eigenvalues close to the double eigenvalue  $\im\omega_*^{(\tp)} (\tth)$ 
 of $\cL_{\umu,0} $. 
We denote by $\cV_{\mu, \e}^{(\tp)}$   the spectral subspace associated to  $\sigma_\tp'\left(\cL_{\mu,\e}\right) $, which   has  dimension 2 and it is  invariant under $\cL_{\mu, \e}$.

\begin{teo}\label{thm:finalmat} For any integer  $\tp\in \bN$, $\tp\geq 2$, for any depth $\tth>0$,
there exist $\e_0^{(\tp)} >  0 $, $ \delta_0^{(\tp)} >0$ such that
the operator $\cL_{\mu,\e} \colon \cV_{\mu, \e} \to \cV_{\mu, \e}$, defined  for any  $(\mu,\e) = (\umu+\delta, \e) \in B_{\delta_0^{(\tp)}}(\umu)
\times B_{\e_0^{(\tp)}}(0)$ with
$\umu := \uphi (\tp, \tth ) > 0 $  in \eqref{constants},  can be  represented 
by a $2\times 2$ 
 matrix with identical real off-diagonal entries and purely imaginary diagonal entries  of the form
$$
 \begin{pmatrix} \im\omega^+(\umu+\tp+\delta,\tth)  
 +\im r_1(\e^2) & 
\beta_1^{(\tp)}(\tth)\, \e^\tp+ r_2 (\e^{\tp+2}, \delta\e^\tp)  \\[2mm]
\beta_1^{(\tp)}(\tth)\, \e^\tp+r_2 (\e^{\tp+2}, \delta\e^\tp)  & \im\omega^-(\umu+\delta,\tth)  + 
  \im r_3 (\e^2) 
\end{pmatrix}\, ,
$$
where $\omega^\pm(\varphi,\tth)$ is given in \eqref{Oomegino} and $\beta_1^{(\tp)}
(\tth) $ is  a real analytic function satisfying
\eqref{beta1plimit}.
\end{teo} 
\noindent Theorem \ref{thm:finalmat}  implies Theorem \ref{thm:main2},  by using Theorem  \ref{primeisola}. 
The form \eqref{isolaintro} of the approximated ellipse  follows by 
Lemma 
\ref{lem:appell}
with 
$ 
E^{(\tp)}(\tth) :=
\frac{\gamma_1^{(\tp)}(\tth) - \alpha_1^{(\tp)}(\tth) }{\gamma_1^{(\tp)}(\tth) + \alpha_1^{(\tp)}(\tth) }
$ 
and 
$ \alpha_1^{(\tp)}(\tth), \gamma_1^{(\tp)}(\tth) $ in  \eqref{omegappomega0m}-\eqref{a1g1}. 


\section{Expansion of the Stokes wave at any order}\label{sec:App31}

Section \ref{sec31} presents the proofs for the structural properties \eqref{leadingStokes} of Stokes waves, as stated in
Theorem \ref{LeviCivita}
and Lemma \ref{pastructure}. 
Furthermore, we prove Propositions \ref{lem:asympetapsi} and
\ref{degenstokesinfinity}, which offer key information about the Stokes wave's Fourier-Taylor expansion at any order in both shallow and deep water.

\subsection{
Structure of the Stokes wave}\label{sec31}

We demonstrate that the Taylor expansion of the Stokes wave maintains the structure outlined in \eqref{leadingStokes} at every order. 
This is achieved by iteratively constructing the solution of system \eqref{travelingWW}.
It is advantageous to introduce formal power series of trigonometric polynomials.

\begin{sia}\label{spaziEvnOdd}
We denote $ \mathtt{Evn}(\e)$, respectively $ \mathtt{Odd}(\e)$, the module over the ring ${\cal Q}(\ch^2)$ (introduced in Definition \ref{rationalstructure}) of formal power series  
\begin{equation}\label{fexgex}
\begin{aligned}
& f_\e(x) = \sum_{\ell\geq 1} \e^\ell f_\ell (x) 
\quad \text{with} 
\quad f_\ell (x) \in \mathtt{Evn}_\ell   \quad \text{with coefficients }      \quad  
f_\ell^{[0]},\dots,f_\ell^{[\ell]} \in {\cal Q}(\ch^2),  \\
& g_\e(x) = \sum_{\ell\geq 1} \e^\ell  g_\ell (x) 
\quad \text{with} \quad 
g_\ell (x) \in \mathtt{Odd}_\ell
\quad \text{with coefficients }   \quad 
g_\ell^{[1]},\dots, g_\ell^{[\ell]}  \in {\cal Q}(\ch^2)  \, ,
\end{aligned}
\end{equation}  
where 
$\mathtt{Evn}_\ell $
and $\mathtt{Odd}_\ell $
are the spaces of trigonometric polynomials of degree $ \ell $  introduced 
in Definition \ref{def:even}.
We denote by $ \bR \cap 
\mathtt{Evn}(\e) $ the one-dimensional submodule of constants  in
$ \mathtt{Evn}(\e) $.
Given a function $ f_\e (x) $ we denote $ [f_\e(x)]_\ell := f_{\ell} (x) $
its Taylor coefficient of order
$ \ell $ in \eqref{fexgex} and, for any
$ \kappa \in \bN $ we denote 
 $$ 
f_{\ell}^{[\kappa]} := [f_\e(x)]_\ell^{[\kappa]} =
[f_\ell(x)]^{[\kappa]} = 
 \begin{cases}
 \frac{1}{\pi} \int_{\bT} f_\ell(x) \cos (\kappa x) \, dx
 \quad \text{if} \ \ell \ \text{is even,} 
 \\ 
 \frac{1}{\pi} \int_{\bT} f_\ell(x) \sin (\kappa x) \,d x 
  \quad \text{if} \ \ell \ \text{is  odd,} 
 \end{cases} \quad f_\ell^{[0]} := \frac{1}{2\pi} \int_\bT f_\ell(x) \de x\, ,
 $$
 the $ \kappa $-harmonic coefficient 
 of $ f_{\ell} (x) $. 
\end{sia}

We shall also consider the ``shifted'' modules over ${\cal Q}(\ch^2)$ given by $\ch \mathtt{Evn}(\e)$ and $\ch \mathtt{Odd}(\e)$. Note that $\ch^2 \mathtt{Evn}(\e) = \mathtt{Evn}(\e)$ and $\ch^2 \mathtt{Odd}(\e) = \mathtt{Odd}(\e)$.

We now list some properties of the formal power series in 
 $\mathtt{Evn}(\e)$ and $\mathtt{Odd}(\e)$.
 Similar results hold, mutatis mutandis, by replacing  $ \mathtt{Evn}(\e)$ with $\ch \mathtt{Evn}(\e)$ or $ \mathtt{Odd}(\e)$ with  $\ch \mathtt{Odd}(\e)$. 
 
 The first result follows directly from Definitions \ref{def:even}-\ref{rationalstructure} and \ref{spaziEvnOdd}. 
\begin{lem}[{\bf Product}]\label{product}
For any  $f_\e (x), \tilde f_\e (x) \in \mathtt{Evn}(\e)$ and $g_\e (x), \tilde g_\e (x) \in \mathtt{Odd}(\e)$, one has  
\begin{equation}
f_\e (x) \tilde f_\e (x) \, ,\; g_\e (x) \tilde g_\e (x) \in \mathtt{Evn}(\e)\, ,\qquad f_\e (x) \tilde g_\e (x) \, ,\; g_\e (x) \tilde f_\e (x) \in \mathtt{Odd}(\e)\, .
\end{equation}
 For any $\ell_1, \ell_2\in \bN$ and $f_{\ell_1} (x)  \in \mathtt{Evn}_{\ell_1}$, $f_{\ell_2} (x) \in \mathtt{Evn}_{\ell_2}$, $g_{\ell_1}(x) \in \mathtt{Odd}_{\ell_1}$ and $g_{\ell_2}(x)\in \mathtt{Odd}_{\ell_2}$ one has
\begin{align}\label{productruleevnodd}
f_{\ell_1} (x) f_{\ell_2} (x) \, ,\; g_{\ell_1}(x) g_{\ell_2}(x) \in \mathtt{Evn}_{\ell_1+\ell_2}\,, \qquad f_{\ell_1} 
(x)g_{\ell_2}(x) \, ,\; g_{\ell_1}(x)  f_{\ell_2}(x) \in \mathtt{Odd}_{\ell_1+\ell_2}\, .
\end{align}
\end{lem}

We now consider the action of
Fourier multipliers. 

\begin{lem}
\label{FourierMultipliers}
{\bf (Fourier multipliers)}
The differential operator $ \pa_x $,
the flat-sea Dirichlet-Neumann operator
$ G_0 = |D| \tanh{(\tth |D|)} $, the
Hilbert transform 
$ \cH:= -\im \sgn(D)$ and $|D|$ satisfy, for any 
$ \ell \in \bN $,  
\begin{equation}\label{paacts}
\begin{aligned}
& \pa_x: \mathtt{Evn}_{\ell}\to \mathtt{Odd}_{\ell} \, ,\   \ \pa_x: \mathtt{Odd}_{\ell}  \to \mathtt{Evn}_{\ell}\, , \\
& G_0: \mathtt{Evn}_{\ell} \to  \mathtt{Evn}_{\ell}  \, ,\  G_0: \mathtt{Odd}_{\ell} \to \mathtt{Odd}_{\ell}\, , \\
& \cH: \mathtt{Evn}_\ell \to \mathtt{Odd}_\ell\, ,\ \  \cH: \mathtt{Odd}_\ell \to \mathtt{Evn}_\ell\, ,\\
& |D|: \mathtt{Evn}_{\ell} \to  \mathtt{Evn}_{\ell}  \, ,\  |D|: \mathtt{Odd}_{\ell} \to \mathtt{Odd}_{\ell}\, , 
\end{aligned}
\quad\
\begin{aligned}
& \pa_x: \mathtt{Evn}(\e) \to \mathtt{Odd}(\e)\, ,\ \  \pa_x: \mathtt{Odd}(\e)\to  \mathtt{Evn}(\e)\, , \\
& G_0: \mathtt{Evn}(\e) \to  \mathtt{Evn}(\e)\, ,\ G_0: \mathtt{Odd}(\e) \to  \mathtt{Odd}(\e)\, , \\
& \cH: \mathtt{Evn}(\e) \to \mathtt{Odd}(\e) \, ,\ 
\ \cH: \mathtt{Odd}(\e) \to \mathtt{Evn}(\e) \, ,\\
& |D|: \mathtt{Evn}(\e) \to  \mathtt{Evn}(\e)\, ,\ |D|: \mathtt{Odd}(\e) \to  \mathtt{Odd}(\e)\,  \, .
\end{aligned}
\end{equation}
\end{lem}

\begin{proof} 
We prove the statement for  $G_0$.  Since
the symbol 
$ |\xi| \tanh (\tth |\xi |)  $ is real and even
in $ \xi $ we have that
$ G_0 $ preserves the parity and, 
for any  $f_\e (x) $ in $ \mathtt{Evn}(\e)$ or $\mathtt{Odd}(\e)$, for any $ \kappa \in \bN_0  $, 
\begin{equation}\label{tanhk}
[G_0 f_\e]_\ell^{[\kappa]} = \kappa \tanh(\tth \kappa) f_\ell^{[\kappa]}\, ,\quad \tanh(\tth \kappa) =\Big( \sum\limits_{\substack{j=1\\ j \, \textup{odd}}}^{\kappa} {\kappa\choose j} \ch^{2j}\Big)\, \Big/\,\Big(\sum\limits_{\substack{j=0\\ j \, \textup{even}}}^{\kappa} {\kappa \choose j} \ch^{2j}\Big) \in {\cal Q}(\ch^2) \, ,
\end{equation}
where the latter formula comes by induction using $$
\tanh(a+b) = \frac{\tanh(a)+\tanh(b)}{1+\tanh(a)\tanh(b)}  \qquad \text{and}
\qquad 
{k \choose j} + {k \choose j-1} = {k+1 \choose j} \, . 
$$
Property \eqref{paacts} 
 for $ \pa_x $, $ \cH  $ 
 and $ |D | $
is similar.
\end{proof}

A key property used repeatedly hereafter is that the leading Fourier coefficients of a product of functions within
$ \mathtt{Evn}(\e) $ or $ \mathtt{Odd}(\e) $
depend solely on the leading Fourier coefficients of each factor, according to the following lemma.  

\begin{lem}\label{lem:prod}
For any $\ell \in \bN $ it results 
$$
\big[ 
f_\e (x) g_\e (x)
\big]_\ell^{[\ell]}
= 
\begin{cases}
\frac12 \sum_{\ell_1+\ell_2= \ell}
f_{\ell_1}^{[\ell_1]}
g_{\ell_2}^{[\ell_2]} 
\qquad 
\text{if}
\ f_\e (x) \in \mathtt{Evn}(\e) \, , \ 
 g_\e (x) \in \mathtt{Evn}(\e) \, \\
 \qquad \qquad 
 \qquad \qquad \qquad \, \text{or} 
\ f_\e (x) \in \mathtt{Evn}(\e) \, , \
 g_\e (x) \in \mathtt{Odd}(\e) \, , 
\\
- \frac12 \sum_{\ell_1+\ell_2= \ell}
f_{\ell_1}^{[\ell_1]}
g_{\ell_2}^{[\ell_2]}\, 
\quad \text{if}
\ f_\e (x) \in \mathtt{Odd}(\e) \, ,  \
 g_\e (x) \in \mathtt{Odd}(\e) \, . 
\end{cases}
$$
\end{lem}

\begin{proof}
We make the proof for
$ f_\e (x), g_\e (x)  \in \mathtt{Evn} (\e) $. For any $ \ell \in \bN $ we have 
$$
\big[ f_\e (x) g_\e (x) \big]_\ell^{[\ell]}
= \sum_{\ell_1+\ell_2 = \ell, \atop j_1 \leq \ell_1, j_2 \leq \ell_2} \big[ f_{\ell_1}^{[j_1]} 
g_{\ell_2}^{[j_2]}  
\cos (j_1 x)
\cos (j_2 x) \big]^{[\ell]}
= \frac12 \sum_{\ell_1+\ell_2= \ell}
f_{\ell_1}^{[\ell_1]}
g_{\ell_2}^{[\ell_2]} 
$$ 
because $ \cos ( (j_1+j_2) x) =  \cos ( \ell x) $
just for $ j_1 = \ell_1 $
and $ j_2 = \ell_2 $. The other cases are similar. 
\end{proof}

We now consider the composition of
functions belonging to the spaces in Definition \ref{spaziEvnOdd}. 

\begin{lem}[{\bf Composition}]
\label{composition}
Let  $f_\e (x) $ be a function  in $ \mathtt{Evn}(\e)$ and $g_\e (x) $ in $\mathtt{Odd}(\e)$. 
Then the composite
 function 
$  f_\e\big(x+g_\e(x)\big)$  belongs to 
$\mathtt{Evn}(\e) $.
In addition, if $ \tilde{f}_\e (x)  \in \mathtt{Evn}(\e)$ satisfies 
$\tilde{f}_\ell^{[\ell]}= f_\ell^{[\ell]}$
for any $\ell\in \bN$, then 
\begin{equation}\label{composition2}
\big[f_\e(x+g_\e(x))\big]_\ell^{[\ell]} = \big[\tilde f_\e(x+g_\e(x))\big]_\ell^{[\ell]} \, , \quad \forall \ell \in \bN\, .
\end{equation}
\end{lem}
\begin{proof}
For any $ \ell \in \bN $, we have  
\begin{equation}\label{compogf}
\big[ f_\e\big(x+g_\e(x)\big) \big]_\ell = \sum_{k=1}^{\ell} \sum_{\ell_1+\dots+\ell_{k}=\ell \atop 
\ell_1, \ldots, \ell_k \geq 1 } \frac{1}{(k-1)!}
(\pa_x^{k-1} f_{\ell_1}) (x) g_{\ell_2}(x)\ldots g_{\ell_{k}}(x)  
\end{equation}
where we set  $ \pa_x^0 := \text{Id} $. By Lemmata \ref{product} and
\ref{FourierMultipliers} 
we deduce that $\big[ f_\e\big(x+g_\e(x)\big) \big]_\ell $ belongs to $ \mathtt{Evn}_\ell$. In view of the product structure \eqref{compogf} and using Lemma \ref{lem:prod}, 
the leading harmonic  
$ \big[ f_\e\big(x+g_\e(x)\big) \big]_\ell^{[\ell]} 
$ depends on
$ f_\e (x) $ only through  the 
leading coefficients 
$ f_{\ell_1}^{[\ell_1]} $.
This proves  \eqref{composition2}. 
\end{proof}

\begin{lem}
\label{naroba} Let $f_\e (x), \tilde{f}_\e (x) 
\in \mathtt{Evn}(\e)$ and $g_\e (x)  \in \mathtt{Odd}(\e)$. 
Then $
\dfrac{f_\e}{1+ \tilde{f}_\e}  \in \mathtt{Evn}(\e)$ and $\dfrac{g_\e}{1+ \tilde{f}_\e}  \in \mathtt{Odd}(\e)$. 
\end{lem}
\begin{proof} We have  $$\dfrac{f_\e}{1+ \tilde{f}_\e}  = f_\e + \sum\limits_{k\geq 1} (-1)^k f_\e \tilde{f}_\e^k = \sum\limits_{\ell \geq 1}\e^\ell \Big(f_\ell(x) + \sum_{k=1}^{\ell-1} (-1)^k \sum\limits_{\ell_1+\dots+\ell_{k+1}=\ell} f_{\ell_1}(x) \tilde{f}_{\ell_2}(x)\dots \tilde{f}_{\ell_{k+1}}(x)  \Big) $$
and we deduce that $\Big[ \dfrac{f_\e}{1+ \tilde{f}_\e}\Big]_\ell $ belongs to $ \mathtt{Evn}_\ell $ by \eqref{productruleevnodd}. Similarly one proves that $\Big[ \dfrac{g_\e}{1+ \tilde{f}_\e}\Big]_\ell \in \mathtt{Odd}_\ell $.
\end{proof}
We now consider the  Dirichlet-Neumann operator $G(\eta)$ 
that can be represented as 
$$
G(\eta) = G_0 + \sum_{\ell \geq 1} G_\ell(\eta) 
\quad\text{where}
\quad
G_0 = |D| \tanh{(\tth |D|)} \, . 
$$

\begin{lem}[{\bf Dirichlet-Neumann}]\label{lem:DN}
Let  $\eta_\e \in \mathtt{Evn} (\e) $.
The Dirichlet-Neumann operator
$ G(\eta_\e) $  and its Taylor components 
$ G_\ell (\eta_\e ) $   
satisfy  
$$
G(\eta_\e) \, , 
G_\ell (\eta_\e ) \, : \mathtt{Odd} (\e) \to \mathtt{Odd} ( \e) \, \, ,
\qquad \forall \ell \in \bN \, .
$$ 
\end{lem}

\begin{proof}
The Taylor coefficient 
$ G_\ell (\eta) $
are recursively given by   \cite[formulae (39)-(40)]{CS} 
\begin{equation}\label{Craig}
\begin{aligned}
G_{2r}(\eta) &= -\frac1{(2r)!}G_0 |D|^{2r-2} \pa_x \eta^{2r} \pa_x - \sum_{s=0}^{r-1}  \frac1{(2r-2s)!} |D|^{2r-2s} \eta^{2r-2s} G_{2s}(\eta)  \\ 
&\quad - \sum_{s=0}^{r-1}  \frac1{(2r-2s-1)!} G_0 |D|^{2r-2s-2} \eta^{2r-2s-1} G_{2s+1}(\eta) \, ,   \\ 
G_{2r-1}(\eta) &= -\frac1{(2r-1)!} |D|^{2r-2} \pa_x \eta^{2r-1} \pa_x - \sum_{s=0}^{r-1}  \frac1{(2r-2s-1)!} G_0 |D|^{2r-2s-2} \eta^{2r-2s-1} G_{2s}(\eta)\\  
&\quad - \sum_{s=0}^{r-2}  \frac1{(2r-2s-2)!} |D|^{2r-2s-2} \eta^{2r-2s-2} G_{2s+1}(\eta) \, . 
\end{aligned}
\end{equation}
By induction and exploiting Lemmata \ref{product}-\ref{FourierMultipliers} one shows that $G_\ell(\eta_\e): \mathtt{Odd}(\e)\to \mathtt{Odd}(\e)$. 
\end{proof}

We also use the following lemma.
We remind that  $ \ttf_\e \in 
\bR \cap \mathtt{Evn} (\e) $. 

\begin{lem}\label{lem:coth}
$
\cH \coth ( (\tth + \ttf_\e)|D|): \mathtt{Evn} (\e) \to \mathtt{Odd} ( \e) 
$
as well as 
$ [ \cH \coth ( (\tth + \ttf_\e)|D|)]_\ell  $ for any $\ell \in \bN $. 
\end{lem}

\begin{proof}
Let $ g(z) := \coth (z) $.
By \eqref{compogf} we have 
\begin{equation}
\big[ g \big( \tth |D| + \ttf_\e |D| \big) \eta_\e \big]_\ell = \sum_{k=1}^{\ell} \sum_{\ell_1+\dots+\ell_{k}=\ell} \frac{1}{(k-1)!}
|D|^{k-1} (\pa_z^{k-1} g)  ( \tth |D|) \eta_{\ell_1}(x)  
\tf_{\ell_2}\ldots \tf_{\ell_{k}} \, . 
\end{equation}
It results that 
$ (\pa_z^{k-1} g) ( \tth |j|) $ belong to
$ {\cal Q}(\ch^2)  $.
This follows by \eqref{tanhk} for 
$ k = 1 $ and inductively for the higher order derivatives.  
\end{proof}


\noindent{\bf Proof of Theorem \ref{LeviCivita}}.
The existence of a 
solution of \eqref{travelingWW}
follows as in \cite{BMV2}
by the Crandall-Rabinowitz
bifurcation theorem from a simple eigenvalue.  The fact that $\e_* > 0 $ is uniform with respect to $\tth >\tth_*$ follows by Remark \ref{uniformity}. 
 We now prove the structural properties \eqref{leadingStokes}.

We are going to prove inductively the following statements 
about the coefficients of the Stokes waves in \eqref{Stokespwseries}: 
for any  $ \ell\in\bN $ 
 \begin{itemize}
 \item 
[$(S \eta )_\ell$]   the functions $  \eta_{\ell'}(x) $, $ \ell' = 1, \ldots, \ell  $, 
 belong to
$ \mathtt{Evn}_{\ell'}$ with coefficients 
$  \eta_{\ell'}^{[0]}\, ,\ldots, \eta_{\ell'}^{[\ell']} \in{\cal Q}(\ch^2)  $; 
 \item[$(S \psi)_\ell$]  
the functions $ \psi_{\ell'}(x) $, $ \ell' = 1, \ldots, \ell  $, 
belong to
$ \mathtt{Odd}_{\ell'}$ with coefficients 
$\ch  \psi_1^{[1]}\,,\dots,  \ch\psi_{\ell'}^{[\ell']} \in {\cal Q}(\ch^2) $;
\item[$(S c )_\ell$]  the coefficients 
   $\ch c_{\ell'} $, $ \ell' = 1, \ldots, \ell - 1 $
   belong to
 $ {\cal Q}(\ch^2) $ and  vanish for any $ \ell' $  odd. 
  \end{itemize}

{\sc Initialization.}
  For $ \ell = 1 $ the statements $(S \eta )_1$,  $(S \psi)_1 $ 
   hold since $ \eta_1 (x) =  \cos (x)  $,
  $ \psi_1 (x) = \ch^{-1} \sin (x)  $, and    $(S c )_1 $ is empty. 
\\[1mm]
{\sc Inductive step.} Assuming  the statements  $(S \eta )_{\ell -1} $,  $(S \psi)_{\ell -1} $, $(S c )_{\ell-1} $
for some $ \ell \geq 2 $  we now  prove 
$(S \eta )_{\ell} $,  $(S \psi)_{\ell} $, $(S c )_{\ell} $. 
We write the water waves system \eqref{travelingWW} as 
\begin{equation}\label{travelingsys}
\begin{cases}
\eta_\e-\ch(\psi_\e)_x  = \widetilde c_\e(\psi_\e)_x - \dfrac12 (\psi_\e)_x^2 + \dfrac12 \dfrac{(\eta_\e)_x^2 [(\psi_\e)_x -c_\e]^2}{1+(\eta_\e)_x^2}  \\
\ch(\eta_\e)_x + G_0 \psi_\e = -\widetilde c_\e (\eta_\e)_x  - \widetilde G(\eta_\e) \psi_\e  
\end{cases}
\end{equation}
where 
\begin{equation}\label{G0}
\begin{aligned}
& \widetilde c_\e := c_\e- \ch\,, \quad \widetilde G(\eta) := G(\eta) - G_0 = \sum_{j\geq 1} G_j(\eta) \, 
\end{aligned}
\end{equation}
and $G_j(\eta)$ is defined in \eqref{Craig}.
In view of \eqref{travelingsys} the Taylor coefficients 
$\eta_\ell(x), \psi_\ell(x), {c_{\ell-1}} $ of the Stokes wave in \eqref{Stokespwseries} are determined recursively 
for any $ \ell \geq 2 $ by the relation  
\begin{equation}\label{cell-1}
\cB_0  
\vet{\eta_\ell}{\psi_\ell} =
c_{\ell-1}  
\vet{(\psi_1)_x}{-(\eta_1)_x} +
\vet{F_\ell}{G_\ell} \, , \quad
\vet{F_\ell}{G_\ell}  := 
\sum_{\substack{ \ell_1+\ell_2 = \ell \atop  \ell_1 \text{even},  \  
2 \leq \ell_1 \leq \ell-2}} c_{\ell_1} 
\vet{(\psi_{\ell_2})_x}{-(\eta_{\ell_2})_x} +
\vet{f_\ell}{g_\ell}  
\end{equation}
(the second sum is empty for $ \ell = 2, 3 $)
where $\cB_0   $ is the self-adjoint operator 
$$
\cB_0  := \begin{bmatrix}
1 & -\ch \pa_x \\ 
 \ch \pa_x & G_0
\end{bmatrix}  \, , \qquad  G_0 = |D| \tanh (\tth |D|)  \, , 
$$
and $ f_\ell(x) $, resp. $ g_\ell(x) $, are the $ \e^\ell $ Taylor coefficients of the functions
\begin{equation}\label{fgapp}
 f_\e :=- \dfrac12 (\psi_\e^{< \ell})_x^2 + 
 \dfrac12 \dfrac{(\eta_\e^{< \ell})_x^2 
 \big[ (\psi_\e^{< \ell})_x -c_\e^{< \ell} \big]^2}{1+(\eta_\e^{< \ell})_x^2}  \, , 
\qquad g_\e := - \big( G(\eta_\e^{< \ell})  - G_0 \big) \psi_\e^{< \ell} \, , 
\end{equation}
with 
\begin{equation}
\begin{aligned} \label{Stokespwseries1}
 &\eta_\e^{<\ell}(x) := \e \cos(x) + \sum_{2 \leq \ell' < \ell} \e^{\ell'} \eta_{\ell'} (x),  \ \ \ \ 
 \psi_\e^{<\ell}(x) :=  \e\ch^{-1}
 \sin(x) + \sum_{2 \leq \ell' <\ell} \e^{\ell'} \psi_{\ell'}(x), \\
 & c_\e^{<\ell} := \ch+\sum_{2 \leq \ell' < \ell,  \,  \ell' \textup{ even}} \e^{\ell'} c_{\ell'}.
\end{aligned} 
\end{equation}
The next  lemma follows by 
 the inductive assumptions $(S \eta )_\ell $,  $(S \psi)_\ell $, 
  $(S c )_\ell $ 
and Lemmata
\ref{product}-\ref{lem:DN}. 
\begin{lem} \label{FellGell}  
The Taylor coefficients 
$ F_\ell (x), G_\ell (x) $
in \eqref{cell-1}
satisfy,
for any $ \ell \geq 2 $,  
\begin{enumerate}
\item 
$ F_\ell (x)  \in \mathtt{Evn}_{\ell}$ with  coefficients 
$  F_\ell^{[0]}\, ,\ldots, F_\ell^{[\ell]}\in {\cal Q}(\ch^2)$;
\item 
 $ G_\ell (x)  \in \mathtt{Odd}_{\ell}$ with
 $  \ch G_\ell^{[1]}\, ,\ldots, \ch G_\ell^{[\ell]}\in  {\cal Q}(\ch^2)$.
\end{enumerate}
\end{lem}

In order to solve the linear system \eqref{cell-1} we report the following lemma. 
\begin{lem}\label{kernelB0}
The Kernel of $ \cB_0 $ is
$$
\text{Ker} \, \cB_0 = \text{span} \vet{\cos(x)}{\ch^{-1} \sin (x)} \, .  
$$
The Range of $ \cB_0 $   is the $ L^2 $-orthogonal to 
$ \text{Ker} \, \cB_0 $ and it is given by 
$$
R  =R_0 \oplus R_1\oplus R_{\geq 2} 
\quad \text{where} \quad 
 \quad R_0 := \text{span}\,\Big\{\vet{1}{0} \Big\}\, ,\ \ 
 R_1 := \text{span}\,\Big\{\vet{-\cos(x)}{\ch\sin(x)} \Big\}\, 
$$ 
and 
$ R_{\geq 2} := \overline{\bigoplus_{k=2}^\infty R_k} $ with 
$ R_k := \text{span}\,\Big\{\vet{\cos(kx)}{0}\, ,\;\vet{0}{\sin(kx)} \Big\} $.
The  linear operator $\cB_0^{-1}:R\to R $ defined by 
\begin{align} 
& \label{cB0inv} 
\cB_0^{-1} \vet{1}{0} = \vet{1}{0}\, ,\quad \cB_0^{-1} \vet{-\cos(x)}{\ch \sin(x)} = \frac{1}{1+\ch^2}\vet{-\cos(x)}{\ch \sin(x)} \quad  \text{and} \quad \forall k \geq 2  \\
&\cB_0^{-1} \vet{f^{[k]} \cos (k x)}{g^{[k]} \sin (k x)} = 
\vet{\eta^{[k]} \cos (k x)}{\psi^{[k]} \sin (k x)}  , \ 
\vet{\eta^{[k]}}{\psi^{[k]}} :=  
\frac{1}{k \tanh(\tth k) - \tanh (\tth) k^2 }
\begin{bmatrix} k \tanh(\tth k) &  \ch k \\ \ch k & 1 \end{bmatrix} 
\vet{f^{[k]}}{g^{[k]}}   ,  \notag 
\end{align}
solves $\cB_0\cB_0^{-1}=\cB_0^{-1}\restr{\cB_0}{R}=\uno_R $.
For any $ k \geq 2 $ 
the function $( k \tanh(\tth k) -  k^2 \tanh (\tth)  )^{-1}$   in \eqref{cB0inv} 
belongs to $ {\cal Q} (\ch^2 ) $ and,
for any $ \tth_* > 0 $, 
\begin{equation}\label{khpos}
\begin{aligned}
\inf_{\tth \geq \tth_*}\inf_{k \geq 2} k \big( k \tanh (\tth)  -  \tanh (\tth k)  \big) & = 
\inf_{\tth \geq \tth_*}  2 \big( 2 \tanh (\tth)  -  \tanh (\tth 2)  \big) \\
& = 
2 \big( 2 \tanh (\tth_*)  -  \tanh (\tth_* 2)  \big) > 0 \, . 
\end{aligned}
\end{equation}
\end{lem}

\begin{proof}
The first part is proved in 
\cite[Lemma A.3]{BMV_ed}. 
Let  us prove \eqref{khpos} and the preservation of the rational structure ${\cal Q}(\ch^2)$.
By the strict convexity of $ z \mapsto \tanh (\tth z) $ for $ z > 0 $ and $ \tanh (0) = 0 $
we deduce that, for any $ k \geq 2 $, the function  
$ \tth \mapsto \tanh (\tth k) - k \tanh (\tth)  $ is strictly increasing and positive. 
Furthermore, 
for any $ \tth > 0 $,  the sequence 
$ k \mapsto \tanh (\tth k) - k \tanh (\tth) $ is increasing.  
This proves \eqref{khpos}. Finally, by \eqref{tanhk}, the function 
$$ 
k \tanh(\tth k) -  k^2 \tanh (\tth)   =
\frac{p(\ch^2)}{q(\ch^2)}
$$
with polynomials
$ p(x), q(x)$ with integer coefficients  
and  $p(x), q(x) \neq 0$ for any $0<x\leq 1$. 
\end{proof}
\begin{rmk}\label{uniformity}
 The lower bound in \eqref{khpos}, which is uniform for all depths $\tth>\tth_*>0$, provides the existence of the Stokes wave for any amplitude  $|\e|<\e_*(\tth_*)$.  Clearly if $\tth_* \to 0$ then $\e_*(\tth_*) 
 \to 0 $.
\end{rmk}

We now solve  \eqref{cell-1}.   
We first determine $ c_{\ell-1} $. 
The right hand side in \eqref{cell-1} has to be orthogonal to 
the Kernel of $ \cB_0 $,
\begin{equation}\label{celletapsi}
c_{\ell-1}
\Big( \vet{(\psi_1)_x}{-(\eta_1)_x} , \vet{\cos(x)}{\ch^{-1} \sin (x)} \Big)  
+  
\Big( \vet{F_\ell (x)}{G_\ell (x)} , \vet{\cos(x)}{\ch^{-1} \sin (x)} \Big) = 0 \, . 
\end{equation}
It results 
\begin{equation}\label{scap1e1}
\Big( \vet{(\psi_1)_x}{-(\eta_1)_x} , \vet{\cos(x)}{\ch^{-1} \sin (x)} \Big) = \ch^{-1} \,.    
\end{equation}
By Lemma \ref{FellGell}, if $ \ell $ is even then 
\begin{equation}\label{ortogl}
\Big( \vet{F_\ell (x)}{G_\ell(x)} , \vet{\cos(x)}{\ch^{-1} \sin (x)} \Big) = 0  
\end{equation}
and therefore $ c_{\ell-1} = 0 $ by
\eqref{celletapsi}. If $ \ell $ is odd then
$ \Big( \vet{F_\ell(x) }{G_\ell(x)} , \vet{\cos(x)}{\ch^{-1} \sin (x)} \Big) \in  {\cal Q}(\ch^2) $ 
and so, by \eqref{scap1e1} and 
\eqref{celletapsi}, we deduce that 
$\ch  c_{\ell-1} \in  {\cal Q}(\ch^2) $. This proves   $(S c )_{\ell} $. 

Then we determine $ \eta_\ell(x) $ and $ \psi_\ell(x) $.  
By Lemma \ref{FellGell}
and Definition \ref{def:even}
if $ \ell $ is even then,
since $ c_{\ell-1} = 0 $,
the coefficients 
$ \eta_\ell^{[j]}$, $ \psi_\ell^{[j]} $ 
are determined 
as the unique solutions of 
\begin{equation}\label{etajpsij}
\begin{bmatrix}
1 & -\ch j \\ 
- \ch j & j \tanh (\tth j)  
\end{bmatrix} 
\vet{\eta_\ell^{[j]}}{\psi_\ell^{[j]}} = \vet{F_\ell^{[j]}}{G_\ell^{[j]}} \, . 
\end{equation}
For any $ j $ odd 
then 
$ F_\ell^{[j]} = G_\ell^{[j]} = 0 $ by 
Lemma \ref{FellGell} and  
so for  $ j \neq 1 $ we have
$ \eta_\ell^{[j]} = \psi_\ell^{[j]} = 0 $.
For $ j = 1 $ 
we choose
$ \eta_\ell^{[1]} = \psi_\ell^{[1]} = 0 $. For $ j $ even 
system \eqref{etajpsij} 
is solved, in view of Lemmata 
 \ref{kernelB0} and  \ref{FellGell}, by  
\begin{equation}\label{lainvj}
\vet{\eta_\ell^{[0]}}{\psi_\ell^{[0]}} = F_\ell^{[0]} \vet{1}{0} \, , \quad  
\vet{\eta_\ell^{[j]}}{\psi_\ell^{[j]}} = \frac{1}{j \tanh (\tth j)   - j^2 \tanh (\tth )}
\begin{bmatrix}
 j \tanh (\tth j)   & \ch j \\ 
 \ch j & 1
\end{bmatrix} 
\vet{F_\ell^{[j]}}{G_\ell^{[j]}} 
\end{equation}
for any $ j \geq 2 $.
If  $ j $ is odd 
then 
$ F_\ell^{[j]} = G_\ell^{[j]} = 0 $ 
and so 
$ \eta_\ell^{[j]} = \psi_\ell^{[j]} = 0 $.
Note that $ j \neq 1 $ since $ j $ is even. By Lemmata 
\ref{FellGell} and \ref{kernelB0}
we deduce 
 $(S \eta )_{\ell} $, $(S \psi )_{\ell} $ for any $ \ell $ even. 
 
 If $ \ell $ is odd  then,  similarly,  we have to determine 
only the $ \eta_\ell^{[j]}$, $ \psi_\ell^{[j]} $ 
for  $ j $ odd as solutions of \eqref{etajpsij}, which, 
for any $ j \neq 1 $, are given by the right hand side in \eqref{lainvj}.  
On the other hand, for $ j  = 1 $,
 the first Fourier component of  
the right hand side in \eqref{cell-1},
$ \vet{c_{\ell-1}(\psi_1)_x+F_\ell (x)}{-c_{\ell-1}(\eta_1)_x+ G_\ell (x)} $,  is equal to 
$ \frac12 (\ch G_\ell^{[1]} - F_\ell^{[1]})  
\vet{-  \cos (x) }{ \sin (x) } $ by \eqref{celletapsi}. Therefore, 
by  \eqref{cB0inv} we have 
$$
\vet{\eta_\ell^{[1]}}{\psi_\ell^{[1]}} = 
   \frac{1}{2(1+ \ch^2)}  \vet{- (\ch G_\ell^{[1]} - F_\ell^{[1]})}{ \ch (\ch G_\ell^{[1]} - F_\ell^{[1]})}\, .
$$
By Lemmata 
\ref{FellGell} and \ref{kernelB0}
we deduce that 
 $(S \eta )_{\ell} $, $(S \psi )_{\ell} $ also for any $ \ell $ odd. This concludes the proof of Theorem \ref{LeviCivita}.

\begin{rmk}
There is no loss of generality in selecting the solution 
 \eqref{cB0inv} on the mode $ 1 $.  
Adding an element of the kernel $ \vet{ \cos (x)}{ \ch^{-1} \sin (x) } $ just amounts 
to a reparametrization of the $ \e $ of the Stokes wave. 
\end{rmk}

\noindent{\bf Proof of Lemma \ref{pastructure}}.
We first recall \cite[Section 2]{BMV3} that
the functions 
$p_\e (x) $,  $a_\e (x) $ 
are given by
\begin{equation}\label{def:pa}
\ch+p_\e(x) :=  \displaystyle{\frac{ c_\e-V(x+\mathfrak{p}_\e (x))}{ 1+\mathfrak{(p_\e)}_x(x)}} \, , \quad 1+a_\e(x):=   \displaystyle{\frac{1+ (V(x + \mathfrak{p}_\e (x)) - c_\e)
 B_x(x + \mathfrak{p}_\e(x))  }{1+\mathfrak{(p_\e)}_x(x)}} \, ,
\end{equation}
where 
$(V,B) :=\restr{(\Phi_x,\Phi_y)}{y=\eta_\e(x)}$ are the horizontal and vertical components of the velocity field  evaluated at the Stokes wave, given explicitly by
\begin{equation}\label{espVB}
V :=  V_\e = -B (\eta_\e)_x + (\psi_\e)_x  = \frac{ (\psi_\e)_x + c_\e(\eta_\e)_x^2}{1+(\eta_\e)_x^2} \, , \qquad
 B := B_\e =\frac{ (\psi_\e)_x- c_\e}{1+(\eta_\e)_x^2}(\eta_\e)_x
  \, ,
\end{equation}
and 
$ \mathfrak{p}_\e(x)$ is the solution of the fixed-point equation
\begin{equation} \label{def:ttf}
\mathfrak{p}_\e(x) =  \frac{\cH }{\tanh \big((\tth + \ttf_\e)|D| \big)}[\eta_\e ( x + \mathfrak{p}_\e(x))] \, , 
 \qquad 
 \ttf_\e:= \frac{1}{2\pi} \int_\bT \eta_\e (x +  \mathfrak{p}_\e(x)) \de x \, ,
 \end{equation}
 where $ \cH =  -\im \sgn(D) $
is the Hilbert transform. 
Note that 
\begin{equation}\label{hilbert}
 \cH \cos(kx) = \sin(kx) \, ,   \, \quad  \cH \sin(kx) = -\cos(kx)  \, ,
 \ \  \forall k \in \bN  \, , \quad   
 \cH (1) = 0 \,  . 
 \end{equation}
Lemma \ref{pastructure} directly descends from the following one. 
\begin{lem}
\label{lacunary}
The  $ 2 \pi $-periodic functions $V_\e(x)$, $B_\e(x)$, $\mathfrak{p}_\e(x)$, $p_\e(x)$ and $a_\e(x)$
are real analytic w.r.t. 
$x$ and $ \e $
and 
$V_\e \in \ch^{-1} \mathtt{Evn}(\e) $, $B_\e
\in \ch \mathtt{Odd}(\e)$,
$\mathfrak{p}_\e \in \mathtt{Odd}(\e)$, 
$\ttf_\e
\in \bR \cap  \mathtt{Evn}(\e)  $,
$ p_\e \in\ch \mathtt{Evn}(\e) $ and 
$ a_\e \in\mathtt{Evn}(\e) $.
\end{lem}
We shall denote
$$
 B_\e(x) = \sum_{\ell \geq 1} \e^\ell B_\ell (x)\, , 
\quad V_\e(x) =  
 \sum_{\ell \geq 1} \e^\ell V_\ell(x)\, , \quad 
\mathfrak{p}_\e(x) = \sum_{\ell \geq 1} \e^\ell \mathfrak{p}_\ell(x)\, . 
$$
\begin{proof}
Using that 
$ \eta_\e \in 
\mathtt{Evn}(\e) $, 
$ \psi_\e \in \ch^{-1} \mathtt{Odd}(\e) $
and $ c_\e \in \bR 
\cap \ch \mathtt{Evn}(\e)  $, 
we deduce, by  Lemmata \ref{product}-\ref{naroba},  that 
the functions $V_\e$ and $B_\e$ in \eqref{espVB} belong to $\ch \mathtt{Evn}(\e)$ and $\ch \mathtt{Odd}(\e)$ respectively.

We deduce that 
$\mathfrak{p}_\e$, resp.  
$\ttf_\e $ in \eqref{def:ttf},  belong to $\mathtt{Odd}(\e)$, resp. 
 $\bR \cap  \mathtt{Evn}(\e)  $, by 
the following inductive statements: for any  $ \ell\in\bN $ 
 \begin{itemize}
 \item 
[$(S \mathfrak{p} )_\ell $]   the functions $  \mathfrak{p}_{\ell'}(x) $, $ \ell' = 1, \ldots, \ell  $, 
 belong to
$ \mathtt{Odd}_{\ell'}$ with coefficients 
$  \mathfrak{p}_{\ell'}^{[1]}\, ,\ldots, \mathfrak{p}_{\ell'}^{[\ell']} \in {\cal Q}(\ch^2) $;
\item 
[$(S \ttf )_\ell$]  the coefficients 
   $ \ttf_{\ell'} $, $ \ell' = 1, \ldots, \ell $, belong to  $ {\cal Q}(\ch^2) $ and vanish for any $ \ell' $  odd. 
  \end{itemize}
The statements 
$(S \mathfrak{p} )_1$ 
and $(S \ttf )_1$ follow 
by \eqref{def:ttf}, \eqref{hilbert}
with $ \mathfrak{p}_1 (x) = 
\ch^{-2} \sin (x) $ and
$\ttf_1 = 0 $. 
Then supposing
$(S \mathfrak{p} )_{\ell-1} $ 
and $(S \ttf )_{\ell-1} $ hold  
for some $ \ell \geq 2 $
we now prove 
$(S \mathfrak{p} )_{\ell} $ 
and $(S \ttf )_\ell $. 
By \eqref{def:ttf}  we have 
$$
\mathfrak{p}_\ell (x)  =  
\Big[ \frac{\cH }{\tanh \big((\tth+
\ttf_\e^{< \ell} ) |D|)}[\eta_\e^{\leq \ell} ( x + \mathfrak{p}_\e^{<\ell} (x))]  \Big]_\ell  
$$
where $ \eta_\e^{\leq \ell}(x) $ is defined as in \eqref{Stokespwseries1}
summing up to $ \ell' \leq \ell $
and 
$$
\mathfrak{p}_\e^{<\ell}(x) := 
 \sum_{1 \leq \ell' < \ell} \e^{\ell'} \mathfrak{p}_{\ell'} (x),  \quad 
 \ttf_\e^{<\ell} := 
 \sum_{2 \leq \ell' < \ell, \atop  \ell' \textup{ even}} \e^{\ell'} \ttf_{\ell'} \, .
$$
Moreover
$$
\ttf_\ell  =
\frac{1}{2\pi} \int_\bT \big[ 
\eta_\e^{\leq \ell} (x +  \mathfrak{p}_\e^{<\ell} (x))\big]_\ell \de x \, . 
$$
By the inductive assumption 
and Lemma \ref{composition}
the function 
$ \eta_\e^{ \leq \ell} (x +  \mathfrak{p}_\e^{<\ell} (x)) $
belongs to 
$   \mathtt{Evn}(\e) $. 
Then if $ \ell  $ is odd then  
$ \ttf_\ell = 0 $ whereas if
$ \ell  $ is even then  
$ \ttf_\ell  $ belongs to 
$ \bR \cap \mathtt{Evn}(\e) $.
Moreover 
by Lemma \ref{lem:coth} the function
$  \mathfrak{p}_{\ell}(x)  $ 
is in
$ \mathtt{Odd}_{\ell}$ with coefficients 
$  \mathfrak{p}_{\ell}^{[1]}\, ,\ldots, \mathfrak{p}_{\ell}^{[\ell]} $  in  $ {\cal Q}(\ch^2) $.

Finally we deduce that the functions $p_\e (x), a_\e (x) $ in \eqref{def:pa} are respectively in $\ch \mathtt{Evn}(\e)$ and $\mathtt{Evn}(\e)$ by Lemmata \ref{product}, \ref{FourierMultipliers},  \ref{composition} and \ref{naroba}. 
\end{proof}

\subsection{Asymptotics of the Stokes wave as $\tth\to 0^+$}\label{sec:limit0}

We now prove a second-order  asymptotic expansion of the leading Fourier coefficients
$ \eta_\ell^{[\ell]}, \psi_\ell^{[\ell]} $  
of the Stokes waves $(\eta_\e,\psi_\e)$ in \eqref{Stokespwseries}
as $ \tth \to 0^+ $, at {\it any} order $\ell \in \bN$. 

\begin{prop}\label{lem:asympetapsi}
For any $\ell \in \bN$, the $\ell$-th Fourier coefficients 
$ \eta_\ell^{[\ell]}, \psi_\ell^{[\ell]}  $ of the $ 2 \pi $-periodic 
coefficients $\eta_\ell(x)$, $\psi_\ell(x)$ of the Stokes wave 
in \eqref{Stokespwseries} have the asymptotic expansion as $ \tth\to 0^+ $ 
\begin{equation}\label{expetajjpsijj}
\begin{aligned}
&  \eta_\ell^{[\ell]} =  \ell \, x_\ell \tth^{3-3\ell} + z_\ell\tth^{5-3\ell} + O(\tth^{7-3\ell}) \, , \\
& \psi_\ell^{[\ell]} = x_\ell \tth^{3-3\ell-\frac12} + y_\ell \tth^{5-3\ell-\frac12} + 
O(\tth^{7-3\ell-\frac12})  \, , 
\end{aligned}
\end{equation}
where
\begin{align}\label{indHP2}
x_\ell := \Big(\frac38\Big)^{\ell-1}\, , \quad y_\ell := \frac{5\ell^2+3\ell-5}{18}  \Big(\frac38\Big)^{\ell-1} \, ,\quad z_\ell:=  \frac{\ell (\ell-1)(
\ell+2)}{6}\Big(\frac38\Big)^{\ell-1} \, .
\end{align}
As a consequence 
 for any $\ell_1, \ell_2 \in \bN$, setting $\ell:=\ell_1+\ell_2 $, as 
 $ \tth\to 0^+ $,  
\begin{equation}\label{multeffect}
\eta_{\ell_1}^{[\ell_1]}\eta_{\ell_2}^{[\ell_2]} = O(\tth^{6-3\ell}) \, ,\quad \eta_{\ell_1}^{[\ell_1]}\psi_{\ell_2}^{[\ell_2]} = O(\tth^{6-3\ell-\frac12}) \, ,\quad  \psi_{\ell_1}^{[\ell_1]}\psi_{\ell_2}^{[\ell_2]} = x_{\ell_1} x_{\ell_2} \tth^{5-3\ell} + O(\tth^{7-3\ell})\, .
\end{equation}
\end{prop}

\begin{proof}
 The case $\ell=1$ holds by \eqref{Stokespwseries} and since
 \begin{equation}\label{expch} 
 \ch = \sqrt{\tanh(\tth)} \, , 
 \quad \tanh(\tth)=\tth\big(1-\tfrac13 \tth^2 + r(\tth^4)\big) \quad  \text{as} \quad  \tth \to 0^+ \, .
 \end{equation}
We now suppose  
by induction
that 
\eqref{expetajjpsijj}, \eqref{indHP2} hold for some $\ell \geq 2$ and any  
$\ell'=1,\dots, \ell-1$. Consequently formula \eqref{multeffect} holds for any $\ell_1+\ell_2=\ell$, since $1\leq \ell_1, \ell_2 <\ell $.

Recall that the pair $\big(\eta_\ell(x),\psi_{\ell}(x)\big)$  fulfills \eqref{cell-1}. We now compute the equations fulfilled by the $\ell$-th harmonics 
$\big( \eta_\ell^{[\ell]}\cos(\ell x), \psi_{\ell}^{[\ell]}\sin(\ell x) \big)$. 
Since each  vector   $c_{\ell_1} 
\vet{(\psi_{\ell_2})_x}{-(\eta_{\ell_2})_x} $ do not possess $\ell$-harmonics (recall that $\ell_1 + \ell_2 = \ell$, $\ell_1 \geq 1$ and that  $  \eta_\ell(x) $, resp. $ \psi_\ell(x) $, 
   is a trigonometric polynomial in
$ \mathtt{Evn}_\ell $, resp. 
in $ \mathtt{Odd}_\ell $), we get  the linear system
\begin{equation}\label{rangeeqell}
\mathcal{B}_0 \vet{\eta_\ell^{[\ell]}\cos(\ell x)}{\psi_\ell^{[\ell]}\sin(\ell x)} = \vet{f_\ell^{[\ell]}\cos(\ell x)}{g_\ell^{[\ell]}\sin(\ell x)}\, ,\quad  \mathcal{B}_0 = \begin{bmatrix}1 & -\ch\pa_x \\ \ch\pa_x & G_0 \end{bmatrix} \, ,
\end{equation}
where
\begin{equation}\label{fellgellprima}
 f_\ell^{[\ell]}:= \Big[
 - \dfrac12 (\psi_\e^{< \ell})_x^2 + 
 \dfrac12 \dfrac{(\eta_\e^{< \ell})_x^2 
 \big[ (\psi_\e^{< \ell})_x -c_\e^{< \ell} \big]^2}{1+(\eta_\e^{< \ell})_x^2}
 \Big]_\ell^{[\ell]} \, , \quad
 g_\ell^{[\ell]} :=  
 \Big[
 - \big( G(\eta_\e^{< \ell})  - G_0 \big) \psi_\e^{< \ell}
 \Big]_\ell^{[\ell]} 	\, .
\end{equation}
Its solution, given by  \eqref{cB0inv}, is 
\begin{equation}\label{etalpsiell}
    \vet{\eta_\ell^{[\ell]}}{\psi_\ell^{[\ell]}} :=  
\frac{1}{\ell \tanh(\tth \ell) - \tanh (\tth) \ell^2 }
\begin{bmatrix} \ell \tanh(\tth \ell) &  \ch \ell \\ \ch \ell & 1 \end{bmatrix} 
\vet{f_\ell^{[\ell]}}{g_\ell^{[\ell]}}  \, .
\end{equation}
We are left to compute the expansion when $\tth \to 0^+$ of the terms in the right hand side.
We will use systematically Lemma \ref{lem:prod}.  

\noindent\underline{Expansion of $f_\ell^{[\ell]}$.}
By \eqref{G0} we have $[ c_\e^{<\ell}]_j^{[\ell']} = 0 $ for any $\ell'\neq 0$ and $j\in \bN$, and by exploiting \eqref{Stokespwseries}-\eqref{leadingStokes} and the multiplicative structure of the spaces $\mathtt{Evn}_j $ and $\mathtt{Odd}_j$ in \eqref{evenodd}-\eqref{evenodd2}, we obtain
\begin{equation}\label{fellgell}
\begin{aligned}
 f_\ell^{[\ell]} &= \Big[ - \frac12 (\psi_\e^{<\ell})_x^2 + \dfrac12 \dfrac{(\eta_\e^{<\ell})_x^2 [\ch-(\psi_\e^{<\ell})_x]^2}{1+(\eta_\e^{<\ell})_x^2}\Big]_\ell^{[\ell]} \\
 &=\Big[ - \frac12 (\psi_\e^{<\ell})_x^2 + \frac12 {\ch^2}   (\eta_\e^{<\ell})_x^2 +  \frac12 \dfrac{  (\psi_\e^{<\ell})_x^2 (\eta_\e^{<\ell})_x^2-2\ch (\psi_\e^{<\ell})_x (\eta_\e^{<\ell})_x^2 - \ch^2 (\eta_\e^{<\ell})_x^4}{1+(\eta_\e^{<\ell})_x^2}\Big]_\ell^{[\ell]} \, . 
 \end{aligned}
\end{equation}
 By 
 \eqref{multeffect} and 
 \eqref{expch} we have, for any $\ell'=1,\dots,\ell$,
 $$
 \begin{aligned}
 &[(\eta_\e^{<\ell'})_x^2 (\psi_\e^{<\ell'})_x^2]_{\ell'}^{[\ell']}=  O(\tth^{11-3\ell'})\,, \quad [\ch (\eta_\e^{<\ell'})_x^2 (\psi_\e^{<\ell'})_x]_{\ell'}^{[\ell']} =  O(\tth^{9-3\ell'})\, , \\ & [ \ch^2(\eta_\e^{<\ell'})_x^4 ]_{\ell'}^{[\ell']} =  O(\tth^{9-3\ell'})\, ,\quad \Big[\frac{1}{1+(\eta_\e^{<\ell'})_x^2}\Big]_{\ell'}^{[\ell']} = O(\tth^{6-3\ell'}) \, , 
 \end{aligned}
 $$
and we deduce from \eqref{fellgell} that 
\begin{equation}
f_\ell^{[\ell]} = \Big[ - \frac12 (\psi_\e^{<\ell})_x^2 + \frac12 \ch^2 (\eta_\e^{<\ell})_x^2 \Big]_\ell^{[\ell]} + O(\tth^{9-3\ell})\, .
\end{equation}
Thus, by the inductive assumption \eqref{expetajjpsijj}-\eqref{indHP2}
and Lemma \ref{lem:prod}, we have 
\begin{align}
\notag
f_\ell^{[\ell]}  &= -\frac12 \sum_{\ell_1 +\ell_2 =\ell} \ell_1 \ell_2\, \psi_{\ell_1}^{[\ell_1]} \psi_{\ell_2}^{[\ell_2]} [\cos(\ell_1 x) \cos(\ell_2 x)]^{[\ell]} + \frac\tth 2 \sum_{\ell_1 +\ell_2 =\ell} \ell_1 \ell_2\, \eta_{\ell_1}^{[\ell_1]} \eta_{\ell_2}^{[\ell_2]} [\sin(\ell_1 x) \sin(\ell_2 x)]^{[\ell]} + O(\tth^{9-3\ell})  \\ \notag 
&= -\frac{\tth^{5-3\ell}}4 \Big(\frac38\Big)^{\ell-2}\sum_{\ell_1 +\ell_2 =\ell}\ell_1 \ell_2  - \frac{\tth^{7-3\ell}}2 \sum_{\ell_1 +\ell_2 =\ell} \ell_1 \ell_2\, x_{\ell_1} y_{\ell_2}  - \frac{\tth^{7-3\ell}}4 \Big(\frac38\Big)^{\ell-2}\sum_{\ell_1 +\ell_2 =\ell}\ell_1^2 \ell_2^2  +  O(\tth^{9-3\ell})
\\ \label{fellellgellell}
 &  = - \frac{3^{\ell-3}}{8^{\ell-1}} (\ell^2-1)\ell\, \tth^{5-3\ell} - \frac{3^{\ell-4}}{8^{\ell-1}}\frac{(\ell^2-1) \ell (11\ell^2+5\ell-14)}{10} \tth^{7-3\ell}  + O(\tth^{9-3\ell}) \, , 
\end{align}
using in the last step the  algebraic identities
\begin{equation}\label{sommenonDoron}
{\begin{matrix}
\sum\limits_{\ell_1 + \ell_2 = \ell \atop 
1 \leq \ell_1, \ell_2  \leq \ell-1 } \ell_1 \ell_2 =  \frac{\ell (\ell^2 -1)}{6} \ , 
\sum\limits_{\ell_1 + \ell_2 = \ell \atop 
1 \leq \ell_1, \ell_2  \leq \ell-1}   \ell_1 \ell_2  (5\ell_2^2 + 3\ell_2 -5 ) = \frac{ (\ell^2-1) \ell  (\ell^2+\ell-4)}{4}\, , \sum\limits_{\ell_1 + \ell_2 = \ell \atop 
1 \leq \ell_1, \ell_2  \leq \ell-1 } \ell_1^2 \ell_2^2  =\frac{(\ell^2-1) \ell (\ell^2+1) }{30}\end{matrix}} \ . 
\end{equation}
\noindent{\underline{Expansion of $g_\ell^{[\ell]}$.}}
The  operators $G_j(\eta) $ fulfill, for $\tth \to 0^+$,
\begin{equation}\label{Gj}
[G_{j}(\eta_\e^{< \ell})\psi_\e^{< \ell}]_\ell^{[\ell]} = O(  \tth^{3j+4-3\ell- \frac12})
\ \mbox{ if }j\mbox{ is even}\, ,\quad [G_{j}(\eta_\e^{< \ell})\psi_\e^{< \ell}]_\ell^{[\ell]} = O( \tth^{3j+3-3\ell- \frac12}) \ \mbox{ if }j\mbox{ is odd}\, ,
\end{equation}
as one can prove by induction using  \eqref{Craig}, 
\eqref{expetajjpsijj}
and  Lemma 
\ref{lem:prod}.  
Thus  by
 \eqref{fellgellprima}
 and \eqref{G0}
\begin{equation}
g_\ell^{[\ell]} = \big[ - \big( G(\eta_\e^{< \ell})  - G_0 \big) \psi_\e^{< \ell} \big]_\ell^{[\ell]} =  \Big[-  G_1(\eta_\e^{< \ell}) \psi_\e^{< \ell} \Big]_\ell^{[\ell]} + O(\tth^{10-3\ell-\frac12})\, .
\end{equation}
Recall that
$$
G_1(\eta)\psi = - \pa_x (\eta \psi_x) - G_0[\eta G_0 \psi)] \ . 
$$
By the inductive assumption \eqref{expetajjpsijj}-\eqref{indHP2} and exploiting that, in view of \eqref{G0}, for any $ \ell_1+\ell_2 =\ell $, 
$$
\begin{aligned}
\big[\pa_x ( \cos(\ell_1 x) \pa_x \sin(\ell_2 x) )\big]^{[\ell]} = -\ell \frac{\ell_2}{2} \sin(\ell x) \, , 
\ \  
\big[G_0 (\cos(\ell_1 x) G_0 \sin(\ell_2 x))  \big]^{[\ell]} =\tth^2 \ell^2  \frac{\ell_2^2}{2}  \sin( \ell x) \big( 1+O(\tth^2)\big) \, ,
\end{aligned}
$$
 we have
\begin{align} \notag
g_\ell^{[\ell]} &= \sum_{\ell_1 +\ell_2 =\ell} \big[\pa_x \eta_{\ell_1}^{[\ell_1]} \cos(\ell_1 x) \pa_x \psi_{\ell_2}^{[\ell_2]} \sin(\ell_2 x)  \big]^{[\ell]} +  \sum_{\ell_1 +\ell_2 =\ell} \big[G_0 \eta_{\ell_1}^{[\ell_1]} \cos(\ell_1 x) G_0 \psi_{\ell_2}^{[\ell_2]} \sin(\ell_2 x)  \big]^{[\ell]} + O(\tth^{10-3\ell-\frac12})  \\ \notag
&=-\ell \sum_{\ell_1 +\ell_2 =\ell} \eta_{\ell_1}^{[\ell_1]} \psi_{\ell_1}^{[\ell_1]} \frac{\ell_2}{2}  +   \tth^2 \ell^2\sum_{\ell_1 +\ell_2 =\ell} \eta_{\ell_1}^{[\ell_1]} \psi_{\ell_1}^{[\ell_1]} \frac{\ell_2^2}{2}  + O(\tth^{10-3\ell-\frac12})  \\ \notag
&= -\ell \sum_{\ell_1 +\ell_2 =\ell} \frac{\ell_2}{2} \big(\ell_1 x_{\ell_1} \tth^{3-3\ell_1} + z_{\ell_1} \tth^{5-3\ell_1} + O(\tth^{7-3\ell_1}) \big)\big(x_{\ell_2} \tth^{3-3\ell_2-\frac12} + y_{\ell_2} \tth^{5-3\ell_2-\frac12} + O(\tth^{7-3\ell_2-\frac12}) \big)   \\ \notag
&\quad +\tth^2 \ell^2 \sum_{\ell_1 +\ell_2 =\ell} \frac{\ell_2^2}{2} \big(\ell_1 x_{\ell_1} \tth^{3-3\ell_1}  + O(\tth^{5-3\ell_1}) \big)\big(x_{\ell_2} \tth^{3-3\ell_2-\frac12} + O(\tth^{5-3\ell_2-\frac12}) \big) + O(\tth^{10-3\ell-\frac12})  \\ \notag
&= -\ell \sum_{\ell_1 +\ell_2 =\ell}  \frac{\ell_1 \ell_2}2 x_{\ell_1} x_{\ell_2} \tth^{6-3\ell -\frac12 } + \Big( \ell^2 \sum_{\ell_1 +\ell_2 =\ell}  \frac{\ell_1\ell_2^2}{2} x_{\ell_1} x_{\ell_2}-\ell \sum_{\ell_1 +\ell_2 =\ell}  \frac{\ell_2}{2} (\ell_1 x_{\ell_1} y_{\ell_2} + z_{\ell_1} x_{\ell_2} ) \Big)  \tth^{8-3\ell-\frac12 } \\
& \quad + O(\tth^{10-3\ell-\frac12}) \notag \\ \notag
&= -2 \frac{3^{\ell-3}}{8^{\ell-1}} (\ell^2-1)\ell^2  \tth^{6-3\ell-\frac12} + \Big[\frac{3^{\ell-3}}{8^{\ell-1}}  \ell^4 (\ell^2-1) - \frac{3^{\ell-4}}{8^{\ell-1}} \frac{(\ell^2-1)\ell^2 (\ell^2+\ell-4)}{2}  \\ \notag &\quad - \frac{3^{\ell-4}}{8^{\ell-1}}\frac{(\ell-2)(\ell^2-1)\ell^2 (3\ell+11)}{10}  \Big] \tth^{8-3\ell-\frac12 } + O(\tth^{10-3\ell-\frac12})  \\ \label{fellellgellell2}
&= -2 \frac{3^{\ell-3}}{8^{\ell-1}} (\ell^2-1)\ell^2  \tth^{6-3\ell-\frac12} + \frac{3^{\ell-4}}{8^{\ell-1}} \frac{(\ell^2-1)\ell^2(11\ell^2 -5\ell + 21)}{5}\tth^{8-3\ell-\frac12 } + O(\tth^{10-3\ell-\frac12})
\end{align}
where in the final step we used \eqref{sommenonDoron} together with the following identities
\begin{equation}\label{sommenonDoron2}
\sum_{\ell_1 + \ell_2 = \ell \atop 1 \leq \ell_1, \ell_2 \leq \ell -1}
\ell_2 \ell_1 (\ell_1 -1) (\ell_1 +2)  = \frac{(\ell^2-1) \ell (\ell-2)(3\ell+11)}{60}
 \ ,
\quad 
\sum_{\ell_1 + \ell_2 = \ell \atop 1 \leq \ell_1, \ell_2 \leq \ell -1} \ell_1 \ell_2^2 = \frac{(\ell^2-1) \ell^2 }{12} \ . 
\end{equation}
\noindent\underline{Expansion of the matrix elements in \eqref{etalpsiell}.}
As $\tth \to 0^+$, the matrix elements in \eqref{etalpsiell} expand as 
\begin{align}\label{exp.matrix}
  &  \frac{1}{\ell \tanh(\tth \ell) - \tanh (\tth) \ell^2 }
\begin{bmatrix} \ell \tanh(\tth \ell) &  \ch \ell \\ \ch \ell & 1 \end{bmatrix} \\
\notag
& = 
\frac{3}{\ell^2 (1-\ell^2) \tth^3 }  \big(1+\tfrac25 \tth^2 (1+\ell^2) + O(\tth^4) \big) \begin{pmatrix}
  \ell^2\,\tth \big(1-\frac{\ell^2 \tth^2}{3} + O(\tth^4) \big) & \ell\,\tth^{\frac12} \big(1-\frac{\tth^2}{6} +O(\tth^4) \big) \\[2mm] \ell\, \tth^{\frac12}\big(1-\frac{\tth^2}{6} +O(\tth^4) \big)  & 1 \end{pmatrix} \, . 
\end{align}
In view of \eqref{etalpsiell} and using the expansion in \eqref{exp.matrix}, we obtain
\begin{equation}\label{grazieMathematica}
 \begin{pmatrix} \eta_\ell^{[\ell]} \\[2mm] \psi_\ell^{[\ell]} \end{pmatrix} 
=  {\footnotesize  \begin{matrix}\begin{pmatrix}  \dfrac{3\big(\ell\, \tth^{\frac12} f_\ell^{[\ell]} + g_\ell^{[\ell]} \big)}{\ell (1-\ell^2) \tth^{\frac52} }  \\[10pt]
 \dfrac{3 \big(\ell\, \tth^{\frac12} f_\ell^{[\ell]} + g_\ell^{[\ell]} \big)}{\ell^2 (1-\ell^2) \tth^3 }
\end{pmatrix} -
\begin{pmatrix}  \dfrac{2\ell^3\, \tth^{\frac12} f_\ell^{[\ell]} + g_\ell^{[\ell]}}{2\ell (1-\ell^2) \tth^{\frac12} } \\[10pt]
 \dfrac{f_\ell^{[\ell]} }{2\ell (1-\ell^2) \tth^{\frac12} }  
\end{pmatrix} +\begin{pmatrix}  \dfrac{6 (1+\ell^2 )\big(\ell\, \tth^{\frac12} f_\ell^{[\ell]} + g_\ell^{[\ell]} \big)}{5\ell (1-\ell^2) \tth^{\frac12} }  \\[10pt]
 \dfrac{6 (1+\ell^2) \big(\ell\, \tth^{\frac12} f_\ell^{[\ell]} + g_\ell^{[\ell]} \big)}{5\ell^2 (1-\ell^2) \tth } 
\end{pmatrix} +  \begin{pmatrix} O( \tth^2 f_\ell^{[\ell]} + \tth^{\frac32} g_\ell^{[\ell]} ) \\[5pt] O(\tth^{\frac32} f_\ell^{[\ell]} + \tth g_\ell^{[\ell]} ) \end{pmatrix}\end{matrix}} \, .
\, 
\end{equation}
By \eqref{fellellgellell}, \eqref{fellellgellell2} and \eqref{grazieMathematica} we obtain the asymptotic expansions  \eqref{expetajjpsijj} for $\psi_\ell^{[\ell]}$ and $\eta_\ell^{[\ell]}$. 
\end{proof}

We deduce a similar property for the coefficients $ p_\e, a_\e  $
of the linearized operator $\cL_\e$ in \eqref{cLepsilon}. 

\begin{lem}\label{paasymashto0+}
For any $\ell \in \bN$ the $ \ell$-th Fourier 
coefficients $p_\ell^{[\ell]}$, $a_\ell^{[\ell]}$ of the $ 2 \pi $-periodic 
functions 
$ p_\ell (x) $, $a_\ell (x)  $
in \eqref{leadingLin} have the asymptotic expansion, as $ \tth\to 0^+ $, 
\begin{equation}\label{exppjjajj}
\begin{aligned}
&p_\ell^{[\ell]}=  -2\ell \Big(\frac38 \Big)^{\ell-1} \tth^{3-3\ell-\frac12} - \ell \Big(\frac38 \Big)^{\ell-1}\,  \frac{11\ell^2-9\ell+1}{9} \tth^{5-3\ell-\frac12}+O( \tth^{7-3\ell-\frac12})\, ,\\ 
&a_\ell^{[\ell]} = - \ell \Big(\frac38 \Big)^{\ell-1} \tth^{2-3\ell} - \ell\Big(\frac38 \Big)^{\ell-1} \, \frac{31\ell^2-9\ell+2}{18}\tth^{4-3\ell} +O(\tth^{6-3\ell}) \, .  
\end{aligned} 
\end{equation}
\end{lem}

\begin{proof}
We shall use systematically Lemma 
\ref{lem:prod}.
We first deduce the following expansions, as $\tth\to 0^+$, 
 \begin{equation}\label{Basymp}
 \begin{alignedat}{2}
 &B_\ell^{[\ell]} \stackrel{\eqref{espVB}}{=}  &&\Big[ \frac{ (\psi_\e)_x- c_\e }{1+(\eta_\e)_x^2}(\eta_\e)_x \Big]_\ell^{[\ell]} \stackrel{\eqref{multeffect}}{=}  [-\ch ( \eta_\e)_x +( \psi_\e)_x ( \eta_\e)_x  ]_\ell^{[\ell]} + O(\tth^{8-3\ell-\frac12})\\
 &\quad\ \ \stackrel{ \eqref{expch}}{=} && \tth^{\frac12} \Big(1-\frac{\tth^2}{6}+O(\tth^4)\Big) \ell \eta_\ell^{[\ell]} - \frac12 \sum_{\ell_1+\ell_2=\ell}\ell_1 \ell_2 \psi_{\ell_1}^{[\ell_1]} \eta_{\ell_2}^{[\ell_2]} + O(\tth^{8-3\ell-\frac12}) \\
  &\quad\ \ \, \stackrel{\eqref{expetajjpsijj}}{=} && \ell^2 \Big(\frac38 \Big)^{\ell-1} \tth^{4-3\ell-\frac12} + \ell^2  \Big(\frac38 \Big)^{\ell-1} \  \frac{\ell^2+3\ell-7}{18} \tth^{6-3\ell-\frac12} + O(\tth^{8-3\ell-\frac12})  
 \end{alignedat}
 \end{equation}
 using also the second identity in \eqref{sommenonDoron2}. Moreover
 \begin{equation}\label{Vasymp}
 \begin{aligned}
  V_\ell^{[\ell]} &\stackrel{\eqref{espVB}}{=}\big[- B (\eta_\e)_x + (\psi_\e)_x \big]_\ell^{[\ell]} \stackrel{\eqref{expetajjpsijj},\eqref{Basymp}}{=} [(\psi_\e)_x]_\ell^{[\ell]} + O(\tth^{7-3\ell-\frac12}) \\  & \; \stackrel{\eqref{expetajjpsijj}}{=}  \ell \Big(\frac38 \Big)^{\ell-1}  \tth^{3-3\ell-\frac12} + \ell \Big(\frac38 \Big)^{\ell-1} \frac{5\ell^2+3\ell-5}{18} \tth^{5-3\ell-\frac12} + O(\tth^{7-3\ell-\frac12})  \, .
 \end{aligned}
 \end{equation}
We now claim that 
for any $ \ell \in \bN  $ the $\ell$-th Fourier coefficient of the function 
$\mathfrak{p}_\ell(x)$  defined by  \eqref{def:ttf} has the asymptotic expansion as $\tth\to 0^+$
\begin{equation}\label{claimpgot}
\mathfrak{p}_{\ell}^{[\ell]} = \Big(\frac38 \Big)^{\ell-1} \tth^{2-3\ell}+\Big(\frac38 \Big)^{\ell-1} \frac{17\ell^2-9\ell-2}{18} \tth^{4-3\ell} + O(\tth^{6-3\ell})\, .
\end{equation}
For $\ell = 1$ we have,  by  \eqref{def:ttf}, 
$$
\mathfrak{p}_1(x) 
\stackrel{\eqref{Stokespwseries}}{=} \frac{\sin(x)}{\tanh(\tth)} \stackrel{\eqref{expch}}{=} \frac{\sin(x)}{\tth} \big(1+\frac{\tth^2}{3}+O(\tth^4)\big) 
$$
and \eqref{claimpgot}
for $ \ell  = 1 $ follows.  Then we argue by induction supposing that \eqref{claimpgot} holds for  some $\ell \geq 2$ and any 
$\ell' =1,\dots \ell-1$. 
Consequently
$\mathfrak{p}_{\ell_1}^{[\ell_1]}\mathfrak{p}_{\ell_2}^{[\ell_2]} = O(\tth^{4-3(\ell_1+\ell_2)}) $ for any $1\leq \ell_1,\ell_2<\ell$.
Recall that $\mathfrak{p}_\e$ is $\mathtt{Odd}(\e)$,  and $\eta_\e (x) $, $\eta_\e(x + \mathfrak{p}_\e(x))$ are in  $\mathtt{Evn}(\e)$ (cfr Lemma \ref{composition}), so 
 we get, by exploiting formula \eqref{compogf},
 \begin{align*}
 \mathfrak{p}_\ell^{[\ell]} &\quad\,\stackrel{\eqref{def:ttf}}{=}\quad\, \Big[  \frac{\cH }{\tanh \big( (\tth+\ttf_\e) |D| \big)}\big[\eta_\e^{\leq \ell} \big( x + \mathfrak{p}^{<\ell}_\e(x)\big)\big]  \Big]_\ell^{[\ell]}\\
 \notag
 & \quad\,\stackrel{\eqref{hilbert}}{=} \quad\, \frac{1}{\tanh \big( \tth \ell \big)}\big[\eta_\e^{\leq \ell} \big( x + \mathfrak{p}_\e^{<\ell}(x)\big)\big]_\ell^{[\ell]}\\
 &\stackrel{\eqref{expetajjpsijj}, \eqref{expch}}{=}  \frac{1+\frac{\ell^2\tth^2}{3}+O(\tth^4)}{\ell\tth}  \Big( \eta_\ell^{[\ell]}  +\frac12 \sum_{\ell_1+\ell_2= \ell }\ell_1 \eta_{\ell_1}^{[\ell_1]} \mathfrak{p}_{\ell_2}^{[\ell_2]} + O(\tth^{7-3\ell}) \Big) \\
&\quad\,\stackrel{\eqref{claimpgot}}{=} \quad\,\Big( \frac38\Big)^{\ell-1} \,\frac{1+\frac{\ell^2\tth^2}{3}+O(\tth^4)}{\tth} \Big( \tth^{3-3\ell} +  \frac{(\ell-1)(\ell+2)}{6}\tth^{5-3\ell}  +  \frac{4\ell^2-6\ell+2}{9} \tth^{5-3\ell}  + O(\tth^{7-3\ell} )\Big) 
 \end{align*}
using also  $\sum\limits_{\ell_1+\ell_2=\ell}\ell_1^2 = \tfrac16 \ell(\ell-1)(2\ell-1) $. 
 This proves the asymptotic expansion \eqref{claimpgot} for $\mathfrak{p}_\ell^{[\ell]} $.
 
Finally we have, denoting $\mathfrak{p}\equiv \mathfrak{p}_\e$ for brevity,
\begin{align*}
p_\ell^{[\ell]} &\quad \, \stackrel{\eqref{def:pa}}{=}\quad \,  \Big[ \displaystyle{\frac{ c_\e-V^{\leq \ell}(x+\mathfrak{p}^{\leq \ell}(x))}{ 1+\mathfrak{p}^{\leq \ell}_x(x)}} \Big]_\ell^{[\ell]} \\ 
&\stackrel{\eqref{Basymp}-\eqref{claimpgot}}{=}\Big[ -\ch \mathfrak{p}^{\leq \ell}_x(x) + \ch (\mathfrak{p}^{< \ell}_x)^2(x) - V^{\leq \ell}(x) - V_x^{< \ell}(x) \mathfrak{p}^{<\ell}(x) + V^{< \ell}(x)\mathfrak{p}^{< \ell}_x(x) \Big]_\ell^{[\ell]}+ O(\tth^{7-3\ell-\frac12}) \\
& = 
- \ch\, \ell\,  \mathfrak{p}_{\ell}^{[\ell]}
+ 
\frac{\ch}{2} \sum_{\ell_1 + \ell_2 = \ell} \ell_1 \ell_2 \mathfrak{p}_{\ell_1}^{[\ell_1]}\mathfrak{p}_{\ell_2}^{[\ell_2]}
-  V_{\ell}^{[\ell]}
+\frac{1}{2} \sum_{\ell_1 + \ell_2 = \ell} (\ell_2 -\ell_1) V_{\ell_1}^{[\ell_1]}\mathfrak{p}_{\ell_2}^{[\ell_2]}
+ O(\tth^{7-3\ell-\frac12})\\
& =
-\ell  \Big( \frac38\Big)^{\ell-1} \tth^{3-3 \ell-\frac12} -  \ell  \Big( \frac38\Big)^{\ell-1}\Big( \frac{17 \ell^2-9 \ell-2 }{18} - \frac16\Big)\tth^{5-3\ell  - \frac12} 
 +\frac29 \ell  \Big( \frac38\Big)^{\ell-1} (\ell^2-1)  \tth^{5-3\ell- \frac12} \\
& 
\quad -\ell \Big(\frac38 \Big)^{\ell-1}  \tth^{3-3\ell-\frac12} - \ell \Big(\frac38 \Big)^{\ell-1} \frac{5\ell^2+3\ell-5}{18} \tth^{5-3\ell-\frac12}  - \frac29  \ell \Big(\frac38 \Big)^{\ell-1}(\ell^2-3\ell+2)\tth^{5-3\ell-\frac12}+O(\tth^{7-3\ell - \frac12}) \, ,
\end{align*}
where in the last step we used \eqref{expch},\eqref{Basymp}-\eqref{claimpgot} and \eqref{sommenonDoron}.
Summing up we obtain the first line in \eqref{exppjjajj}.

Similarly,
\begin{align*}
 a_\ell^{[\ell]}  &\quad \, \stackrel{\eqref{def:pa}}{=}  \quad\, \Big[\displaystyle{\frac{1+ (V(x + \mathfrak{p}(x)) - c_\e)
 B_x(x + \mathfrak{p}(x))  }{1+\mathfrak{p}_x(x)}} \Big]_\ell^{[\ell]} \\&\quad \,\stackrel{\eqref{def:pa}}{=}\quad\, \Big[\displaystyle{\frac{1}{1+\mathfrak{p}_x^{\leq \ell}(x)}}- (\ch+ p_\e^{\leq \ell}(x)) B_x(x+\mathfrak{p}^{\leq \ell}(x))\Big]_\ell^{[\ell]} \\
 &\stackrel{\eqref{exppjjajj}-\eqref{claimpgot}}{=} \Big[- \mathfrak{p}_x^{\leq \ell}(x)+(\mathfrak{p}_x^{\leq \ell})^2(x)-\ch B_x^{\leq \ell}(x) \Big]_\ell^{[\ell]} +O(\tth^{6-3\ell})\,\\
 &\qquad  = \quad\,
-  \ell\,  \mathfrak{p}_{\ell}^{[\ell]}
+ 
\frac{1}{2} \sum_{\ell_1 + \ell_2 = \ell} \ell_1 \ell_2 \mathfrak{p}_{\ell_1}^{[\ell_1]}\mathfrak{p}_{\ell_2}^{[\ell_2]}
-  \ch \ell B_{\ell}^{[\ell]} +O(\tth^{6-3\ell}) \\
&\qquad = \quad\,
 -\ell \Big(\frac38 \Big)^{\ell-1}  \tth^{2-3\ell} - \ell
 \Big(\frac38 \Big)^{\ell-1} \frac{17 \ell^2 - 9\ell -2}{18}\tth^{4-3\ell}\\
 &\qquad \qquad+ \frac29 \ell \Big(\frac38 \Big)^{\ell-1}  (\ell^2-1) \tth^{4-3\ell}  - \ell^3 \Big(\frac38 \Big)^{\ell-1}   \tth^{4-3\ell}+O(\tth^{6-3\ell})
\end{align*}
where in the last step we used \eqref{Basymp}-\eqref{claimpgot} and \eqref{sommenonDoron}.
This gives the second line in \eqref{exppjjajj}.
\end{proof}

\subsection{Limit of the Stokes wave as $ \tth \to + \infty $}\label{lastokeswavesva}

We now prove another 
structural property of the 
leading Fourier coefficients
$ \eta_\ell^{[\ell]}, \psi_\ell^{[\ell]} $ of the Stokes waves at {\it any} order in the deep water limit.

\begin{prop} \label{degenstokesinfinity} For any $\ell \in \bN$, the $\ell$-th Fourier coefficients 
$   \eta_\ell^{[\ell]}, \psi_\ell^{[\ell]} $ of the $ 2 \pi $-periodic 
coefficients $\eta_\ell(x)$, $\psi_\ell(x)$ of the Stokes wave 
in \eqref{Stokespwseries} have a finite identical limit as $\tth\to +\infty$, 
\begin{equation}\label{etapsiugualiinf}
\lim_{\tth \to +\infty }\eta_1^{[1]} =  \lim_{\tth \to +\infty } \psi_1^{[1]} = 1\,,\qquad
\lim_{\tth \to +\infty }\eta_\ell^{[\ell]} =  \lim_{\tth \to +\infty } \psi_\ell^{[\ell]} = \lim_{\tth \to +\infty } -\frac{f_\ell^{[\ell]}}{\ell-1} \in \bQ \, ,\quad \forall  \ell\geq 2\, ,
\end{equation}
with $f_\ell^{[\ell]}$ in \eqref{fellgell}.
\end{prop}
In order to prove Proposition \ref{degenstokesinfinity} the key result is 
Lemma \ref{lemchiave},  
where for any function $f_\e\in \mathtt{Evn}(\e)$ and $h_\e \in \mathtt{Odd}(\e)$ we denote for brevity 
 $$
  f_\ell^{[\ell]}\piuinf := \lim_{\tth\to \infty} f_\ell^{[\ell]} \in \bQ \, ,\quad  g_\ell^{[\ell]}\piuinf := \lim_{\tth\to \infty} g_\ell^{[\ell]} \in \bQ\, .
 $$

\begin{lem}\label{lemchiave}
Let $f_\e\in \mathtt{Evn}(\e)$ and $h_\e \in \mathtt{Odd}(\e)$ according to  Definition \ref{spaziEvnOdd}. Then, for any  
$j,\ell \in \bN $,  
\begin{equation}\label{3identitiesonG}
\lim_{\tth \to +\infty}[G_0 h_\e ]_\ell^{[\ell]} = \ell\, h_\ell^{[\ell]}\piuinf \,,\quad\lim_{\tth \to +\infty}[\pa_x f_\e \pa_x h_\e + G_0  f_\e G_0 h_\e]_\ell^{[\ell]} = 0 \, ,\quad
\lim_{\tth \to +\infty}[G_j(f_\e) h_\e ]_\ell^{[\ell]} = 0 \, .
\end{equation}  
\end{lem}

\begin{proof}
The first identity directly  follows by
$$
[G_0 h_\e ]_\ell^{[\ell]}  = \big[|D|\tanh(\tth |D|) h_\ell^{[\ell]} \sin(\ell x) \big]_\ell^{[\ell]} = \ell\tanh(\tth \ell) h_\ell^{[\ell]}   \overset{\tth \to +\infty}{\longrightarrow}\ell h_\ell^{[\ell]}\piuinf \, . 
$$
The second identity 
follows,
recalling that 
$ f_\ell \in \mathtt{Evn}_\ell $
and $ h_\ell \in \mathtt{Odd}_\ell $, see Definition \ref{def:even}, 
by  Lemma 
\ref{lem:prod}, 
$$
\lim_{\tth\to +\infty }[\pa_x f_\e \pa_x h_\e ]_\ell^{[\ell]}  
= -\frac{\ell}2  \sum_{\ell_1+\ell_2=\ell} \ell_2  f_{\ell_1}^{[\ell_1]}\piuinf h_{\ell_2}^{[\ell_2]}\piuinf \, ,
$$
and, similarly,
$$
    \lim_{\tth\to +\infty }[ G_0  f_\e G_0 h_\e]_\ell^{[\ell]}   = \frac{\ell}{2} \sum_{\ell_1+\ell_2=\ell} \ell_2 f_{\ell_1}^{[\ell_1]}\piuinf  h_{\ell_2}^{[\ell_2]}\piuinf\, . 
$$
The third identity follows by induction on $j\in \bN$. Indeed, the case $j=1$ is exactly the second identity  in 
\eqref{3identitiesonG}, 
since
$
G_1(\eta) = - \pa_x \eta \pa_x - G_0 \eta G_0 $, 
cfr.  \eqref{G0}. Let $j>1$ and suppose that \eqref{3identitiesonG} holds for $j'=1,\dots,j-1$. 
Then, in view of the recursive expression \eqref{Craig} of the 
Dirichlet-Neumann operator,
the fact that 
$ f_\e^j \in \mathtt{Evn}(\e)$, $ G_\ell (f_\e^j)
h_\e \in \mathtt{Odd}(\e) $, and inductive 
hypothesis, 
one has 
\begin{equation}\label{zeroinf}
\lim_{\tth \to +\infty}[G_j(f_\e) h_\e ]_\ell^{[\ell]} \stackrel{\eqref{Craig}}{=} - \frac{1}{j!} \ell^{j-1} \lim_{\tth \to +\infty}[\pa_x f_\e^j \pa_x h_\e+|D| f_\e^j G_0 h_\e  ]_\ell^{[\ell]}\, ,
\end{equation}
which comes from
$$
\begin{aligned}
& - \frac{1}{j!} 
G_0 |D|^{j-2} \pa_x h_\e
f_\e^{j} \pa_x   h_\e
- \frac{1}{j!} 
|D|^{j} 
f_\e^{j} G_0 h_\e \, \quad 
\text{if} \ j  \
\text{is even} \, ,  \\
& - \frac{1}{j!} 
 |D|^{j-1} \pa_x 
f_\e^{j} \pa_x   h_\e
- \frac{1}{j!} 
G_0 |D|^{j-1} 
f_\e^{j} G_0 h_\e \,  \quad 
\text{if}  \ j  \
\text{is odd} \, , 
\end{aligned}
$$
whereas the other
terms vanish. 
Finally, 
the limit in
\eqref{zeroinf} 
is zero since $ f_\e^j \in \mathtt{Evn}(\e)$, recalling that 
$ G_0 = |\tth|\tanh (\tth |D|)$
and using the second identity in \eqref{3identitiesonG}.
\end{proof}
\noindent{\it Proof of Proposition \ref{degenstokesinfinity}}.
 The case $\ell=1$ holds by \eqref{Stokespwseries}. 
 For any $ \ell \geq 2 $ 
the pair $(\eta_\ell^{[\ell]},\psi_\ell^{[\ell]})$ solves the linear system \eqref{rangeeqell} with $f_\ell^{[\ell]}$, $g_\ell^{[\ell]} $ in \eqref{fellgell} where, 
by Lemma \ref{lemchiave},  
\begin{equation}\label{claimMax} 
\lim_{\tth \to +\infty} g_\ell^{[\ell]}  = \lim_{\tth \to +\infty} \Big[- \widetilde G(\eta_\e) \psi_\e\Big]_\ell^{[\ell]} = - \sum_{j=1}^\ell \lim_{\tth \to +\infty} \Big[G_j(\eta_\e) \psi_\e\Big]_\ell^{[\ell]} =  0\, ,\qquad \forall \ell \in \bN\, .
\end{equation} 
By \eqref{claimMax} and \eqref{rangeeqell} we have
\begin{align*}
\lim_{\tth \to +\infty} \vet{\eta_\ell^{[\ell]}\cos(\ell x)}{\psi_\ell^{[\ell]}\sin(\ell x)} = \lim_{\tth \to +\infty} (G_0 + \ch^2 \pa_x^2)^{-1 }\begin{bmatrix}G_0 & \ch\pa_x \\ -\ch\pa_x & 1 \end{bmatrix}^{-1 }\vet{f_\ell^{[\ell]}\cos(\ell x)}{0} \,,
\end{align*}
namely, passing to the limit  $\tth\to +\infty$,
\begin{equation}\label{etafell}
\begin{pmatrix}\eta_\ell^{[\ell]}\piuinf
\\[1mm] \psi_\ell^{[\ell]}\piuinf\end{pmatrix} = -\frac{1}{\ell( \ell-1)} \begin{pmatrix}\ell & \ell \\ \ell & 1 \end{pmatrix} \begin{pmatrix}f_\ell^{[\ell]}\piuinf
\\[1mm] 0 \end{pmatrix} =-\frac{1}{ \ell-1}   \begin{pmatrix}f_\ell^{[\ell]}\piuinf 
\\[1mm] f_\ell^{[\ell]}\piuinf  \end{pmatrix} \, .
\end{equation}
We have  proved that $\eta_\ell^{[\ell]}\piuinf
= \psi_\ell^{[\ell]}\piuinf $ 
are equal in the deep-water limit.\hfill \qed \smallskip

We deduce a similar property for the coefficients $ p_\e, a_\e  $
of the linearized operator $\cL_\e$ in \eqref{cLepsilon}.
\begin{lem}\label{apinf}
For any $\ell \in \bN$ the $ \ell$-th Fourier 
coefficients $p_\ell^{[\ell]}$, $a_\ell^{[\ell]}$ of the $ 2 \pi $-periodic 
functions 
$ p_\ell (x) $, $a_\ell (x)  $
in \eqref{leadingpa} have a finite identical limit as $\tth\to +\infty$, i.e.
\begin{equation}\label{paugualiinf}
 \lim_{\tth\to +\infty} p_\ell^{[\ell]}=\lim_{\tth\to +\infty} a_\ell^{[\ell]} = :\varpi_\ell \in \bQ \, .
\end{equation}
\end{lem}

\begin{proof} {\bf Step 1}.
We now prove inductively that
for any $ \ell \in \bN $ we have 
\begin{equation}\label{allequalsl}
\eta_\ell^{[\ell]}\piuinf = \psi_\ell^{[\ell]} \piuinf=  B_\ell^{[\ell]}\piuinf = 
V_\ell^{[\ell]} \piuinf \, .  
\end{equation}
The case $ \ell = 1 $ follows by comparison of formulas \eqref{Stokespwseries} with \cite[(B.12)-(B.13)]{BMV3}. 
Let us assume that \eqref{allequalsl}
holds up to $ \ell  - 1 $. We have
\begin{equation}\label{lapriBV}
- c_\e (\eta_\e)_x \stackrel{\eqref{espVB}}{=} B - (\eta_\e)_x V \, , 
\quad (\psi_\e)_x \stackrel{\eqref{espVB}}{=}  V + 
(\eta_\e)_x B \, , 
\quad G(\eta_\e) \psi_\e \stackrel{\eqref{travelingWW}}{=}  - c_\e (\eta_\e)_x  \, .
\end{equation}
The first relation 
in \eqref{lapriBV} gives  
$ - [\eta_x]_\ell^{[\ell]}\piuinf = 
B_\ell^{[\ell]}\piuinf - [(\eta_\e)_x V]_\ell^{[\ell]}\piuinf $
and so,
by  Lemma 
\ref{lem:prod}, 
\begin{equation}\label{intemVB}
\ell \eta_\ell^{[\ell]}\piuinf = 
B_\ell^{[\ell]}\piuinf +  \frac12
\sum_{\ell_1+\ell_2= \ell}
\ell_1 \eta_{\ell_1}^{[\ell_1]}\piuinf
V_{\ell_2}^{[\ell_2]}\piuinf =
B_\ell^{[\ell]}\piuinf + 
\frac12\sum_{\ell_1+\ell_2= \ell}
\ell_1 \eta_{\ell_1}^{[\ell_1]}\piuinf
\eta_{\ell_2}^{[\ell_2]}\piuinf
\end{equation}
by the inductive hypothesis.
Similarly, 
the second relation 
in \eqref{lapriBV} gives  
\begin{equation}\label{intemVB2}
\ell \psi_\ell^{[\ell]}\piuinf = 
V_\ell^{[\ell]}\piuinf +  \frac12
\sum_{\ell_1+\ell_2= \ell}
\ell_1 \eta_{\ell_1}^{[\ell_1]}\piuinf
B_{\ell_2}^{[\ell_2]}\piuinf =
V_\ell^{[\ell]}\piuinf + 
\frac12\sum_{\ell_1+\ell_2= \ell}
\ell_1 \eta_{\ell_1}^{[\ell_1]}\piuinf
\eta_{\ell_2}^{[\ell_2]}\piuinf \, . 
\end{equation}
Subtracting 
\eqref{intemVB}-\eqref{intemVB2}
we deduce, since $ \eta_\ell^{[\ell]}\piuinf
= \psi_\ell^{[\ell]}\piuinf $,  
that 
$ V_{\ell}^{[\ell]}\piuinf = 
B_{\ell}^{[\ell]}\piuinf$. 
We now claim that 
\begin{equation}\label{intemVB1} 
(\ell -1) \eta_\ell^{[\ell]}\piuinf
=  
\frac12 \sum_{\ell_1+\ell_2= \ell}
\ell_1 \eta_{\ell_1}^{[\ell_1]}\piuinf
\eta_{\ell_2}^{[\ell_2]}\piuinf \, 
\end{equation}
which implies, via \eqref{intemVB},
that $ \eta_\ell^{[\ell]}\piuinf = B_\ell^{[\ell]}\piuinf $ and then proves \eqref{allequalsl} at order $ \ell $.\\
To prove claim \eqref{intemVB1}, we observe that, by  \eqref{etapsiugualiinf},
\begin{equation}\label{ilpassaggio}
(\ell-1) \eta_\ell^{[\ell]} \piuinf = 
- f_\ell^{[\ell]}\piuinf \, ,  
\end{equation}
where, using 
\eqref{fellgellprima}, \eqref{fellgell}, 
and Lemma 
\ref{lem:prod}, 
$$
\begin{aligned}
f_\ell^{[\ell]}\piuinf
&  =
\frac12 \Big[
- (\psi_\e)_x^2 + \frac{(\eta_\e)_x^2 (c_\e- (\psi_\e)_x)^2}{1 + (\eta_\e)_x^2} \Big]_\ell^{[\ell]}  \piuinf
\stackrel{\eqref{espVB}} = \frac12 \Big[
- (\psi_\e)_x^2 + B^2 (1+ (\eta_\e)_x^2) \Big]_\ell^{[\ell]}\piuinf \\
& \stackrel{\eqref{espVB}} = 
\frac12 \Big[
- (\psi_\e)_x^2 + B^2 + (V-(\psi_\e)_x)^2 \Big]_\ell^{[\ell]}\piuinf
= 
\frac12 \Big[
B^2 + V^2 - 2 V (\psi_\e)_x \Big]_\ell^{[\ell]}\piuinf \\
& = 
\frac14 
\sum_{\ell_1+\ell_2 = \ell}
- B_{\ell_1}^{[\ell_1]}\piuinf
B_{\ell_2}^{[\ell_2]}\piuinf +
V_{\ell_1}^{[\ell_1]}\piuinf
V_{\ell_2}^{[\ell_2]}\piuinf 
- 2 \ell_1 \psi_{\ell_1}^{[\ell_1]}\piuinf
V_{\ell_2}^{[\ell_2]}\piuinf \\
& = 
\frac14 
\sum_{\ell_1+\ell_2 = \ell}
- 2 \ell_1 \eta_{\ell_1}^{[\ell_1]}\piuinf
\eta_{\ell_2}^{[\ell_2]}\piuinf
 = 
- \frac12 \sum_{\ell_1+\ell_2 = \ell}
 \ell_1 \eta_{\ell_1}^{[\ell_1]}\piuinf
\eta_{\ell_2}^{[\ell_2]}\piuinf
\end{aligned}
$$
by the inductive assumption \eqref{allequalsl}. 
In view of \eqref{ilpassaggio} 
this proves 
\eqref{intemVB1} and concludes the proof of \eqref{allequalsl}. \\[1mm]
{\bf Step 2}. We finally prove \eqref{paugualiinf}.
 We have 
\begin{equation}
\begin{aligned}
V_\ell^{[\ell]}\piuinf &\stackrel{\eqref{espVB}}{=} \ell\psi_\ell^{[\ell]}\piuinf - \sum_{\ell_1+\ell_2=\ell}\ell_2 B_{\ell_1}^{[\ell_1]} \piuinf \eta_{\ell_2}^{[\ell_2]}\piuinf \\
&\stackrel{\eqref{allequalsl}}{=} \ell B_\ell^{[\ell]}\piuinf - \sum_{\ell_1+\ell_2=\ell}\ell_2 V_{\ell_1}^{[\ell_1]}\piuinf B_{\ell_2}^{[\ell_2]}\piuinf   =\big[ (c_\e -V)B_x  \big]_\ell^{[\ell]}\piuinf\, .
\end{aligned}\vspace{-3mm}
\end{equation}
By \eqref{composition2} we then have
$$
\big[V \big(x+\mathfrak{p}(x)\big)\big]_\ell^{[\ell]}\piuinf  =\big[ \big(c_\e -V \big(x+\mathfrak{p}(x)\big)\big)B_x\big(x+\mathfrak{p}(x)\big) \big]_\ell^{[\ell]}\piuinf\, ,
$$
which, in view of \eqref{def:pa},
proves \eqref{paugualiinf}.
\end{proof}
 
\section{Isolas of unstable eigenvalues}\label{Katoapp}

To analyize the splitting of the 
double eigenvalue $ \im\omega_*^{(\tp)}(\tth) $ of $ \cL_{\underline \mu,0} $ in \eqref{spettrodiviso0} for 
 $ (\mu, \epsilon) $ close to $ ( \underline \mu,0) $,
 we employ  Kato's similarity transformation 
theory  developed 
in \cite{BMV1,BMV3,BMV4} and summarized below.  
We recall that $ \cL_{\mu,\e}  : Y \subset X \to X $   
has domain $Y:=H^1(\mathbb{T}):=H^1(\mathbb{T},\bC^2)$ and range $X:=L^2(\mathbb{T}):=L^2(\mathbb{T},\bC^2)$.

\begin{lem}\label{lem:Kato1}
{\bf (Kato theory for separated eigenvalues of Hamiltonian operators)} 
Fix $\tp\in \bN $, $\tp\geq 2 $. 
Let $ \Gamma $ be a closed counter-clocked wise oriented curve winding once around
 the double eigenvalue 
 $ \im\omega_*^{(\tp)}(\tth) $  in \eqref{spettrodiviso0} of $ \cL_{\underline \mu,0} $, 
  separating $\sigma_\tp' (\cL_{\underline \mu, 0}) $ 
  from the other part of the spectrum $\sigma_\tp'' (\cL_{\underline \mu,0})$.
Then there exist $\e_0, \delta_0>0$  such that for any $(\mu, \e) \in B_{\delta_0}(\underline \mu)\times B_{\e_0}(0)$  the following hold.
\\[1mm]
1. 
The curve $\Gamma$ belongs to the resolvent set of 
the operator $\cL_{\mu,\e} : Y \subset X \to X $ defined in \eqref{WW}. 
The operators
\begin{equation}\label{Pproj} 
 P(\mu,\e) := -\frac{1}{2\pi\im}\oint_\Gamma (\cL_{\mu,\e}-\lambda)^{-1} \de\lambda : X \to Y 
\end{equation} 
are projectors commuting  with $\cL_{\mu,\e}$ on $Y$,  i.e.\ 
$ P(\mu,\e)^2 = P(\mu,\e) $ and 
\begin{equation}\label{commuPL}
P(\mu,\e)\cL_{\mu,\e}[y] = \cL_{\mu,\e} P(\mu,\e)  [y] \, ,
\quad \forall y\in Y \, . 
\end{equation}
The map  $(\mu, \epsilon)\mapsto P(\mu,\epsilon) $ is  
analytic from $B_{\delta_0}(\underline \mu)\times B_{\e_0}(0)$ to $ {\cal L}(X,Y)$. 
The projectors $P(\mu,\e) $ are 
 skew-Hamiltonian and reversibility preserving, i.e. 
 \begin{equation}\label{propPU}
  \cJ P(\mu,\e) =P(\mu,\e)^*\cJ \, , \quad 
\bro P(\mu,\e) = P(\mu,\e)  \bro \, .
\end{equation} 
2.  
The domain $ Y $ of  $\cL_{\mu,\e}$ decomposes as  the direct sum 
$ Y = \mathcal{V}_{\mu,\e} \oplus \mathcal{K}_{\mu,\e} $ of the closed 
subspaces 
$ \mathcal{V}_{\mu,\e} :=\text{Rg}(P(\mu,\e)) $, 
$ \mathcal{K}_{\mu,\e}:=\text{Ker}(P(\mu,\e)) \cap Y$, 
which are  invariant under $\cL_{\mu,\e} $, and  
$$
\begin{aligned}
&\sigma(\cL_{\mu,\e})\cap \{ z \in \bC \mbox{ inside } \Gamma \} = \sigma(\cL_{\mu,\e}\vert_{{\mathcal V}_{\mu,\e}} )  = \sigma'(\cL_{\mu, \e}) \, . 
\end{aligned}
$$
3.   The projectors $P(\mu,\e)$ 
are similar one to each other: the  transformation operators
\begin{equation} \label{OperatorU} 
U_{\mu,\e}   := 
\big( \uno-(P(\mu,\e) -P(\umu,0))^2 \big)^{-1/2} \big[ 
P(\mu,\e) P(\umu,0)  + (\uno - P(\mu,\e))(\uno-P(\umu,0) ) \big] 
\end{equation}
are bounded and  invertible in both $ Y $ and $ X $, 
 and 
$ U_{\mu,\e} P(\umu,0) U_{\mu,\e}^{-1} =  P(\mu,\e)  $.  
The map $(\mu, \epsilon)\mapsto  U_{\mu,\e}$ is analytic from 
$ B_{\delta_0}(\underline \mu)\times B_{\e_0}(0)$ to $\cL\big( Y \big) $.
The transformation operators $U_{\mu,\e}$  are symplectic and reversibility preserving, namely
\begin{equation}\label{Usire}
   U_{\mu,\e}^* \cJ U_{\mu,\e}= \cJ \, , \qquad 
\bro U_{\mu,\e} = U_{\mu,\e}  \bro \, .
\end{equation}
Finally 
$ \mathcal{V}_{\mu,\e}=  U_{\mu,\e}\mathcal{V}_{\umu,0}  $
and $\dim \mathcal{V}_{\mu,\e} = \dim \mathcal{V}_{\umu,0}= 2 $ for any 
 $(\mu, \e) \in B_{\delta_0}(\underline \mu)\times B_{\e_0}(0)$.
\end{lem}

We consider the basis  of the subspace 
$ {\mathcal V}_{\mu,\e} = \mathrm{Rg} (P(\mu,\e) )$, 
\begin{equation}\label{basisF}
{\cal F} := \{ 
f^+(\mu,\e),   f^- (\mu,\e)\} \, , \quad 
f^+(\mu,\e) := U_{\mu,\e} f_{\tp}^+ , 
\ \ 
f^-(\mu,\e) := U_{\mu,\e} f_{0}^- \, , 
\end{equation}
obtained applying  the transformation operators $ U_{\mu,\e} $ in \eqref{OperatorU} 
to the eigenvectors   $f_{0}^-:=f_{0}^-(\umu,\tth)$  and 
$f_{\tp}^+:=f_{\tp}^+(\umu,\tth)$
of $ \cL_{\umu,0}$  defined in 
\eqref{autovettorikernel}, which are a symplectic and reversible basis of 
$ \mathcal{V}_{\underline \mu,0} =\mathrm{Rg} (P(\underline \mu,0) )$. 
Thanks to \eqref{Usire} 
and \eqref{sympbas}
the basis 
$ {\cal F} $ is symplectic and reversible as well, namely 
$$
\begin{aligned}
& {\cal W}_c \big( f^+
(\mu,\e), 
f^+
(\mu,\e) \big)  
= - \im \, , 
\quad 
{\cal W}_c \big( f^-
(\mu,\e), 
f^-
(\mu,\e) \big) = \im \, ,
\quad 
{\cal W}_c \big( f^+
(\mu,\e), 
f^-
(\mu,\e) \big) = 0 \, , \\
\end{aligned}
$$
and $ \bar \rho f^\sigma 
(\mu,\e) = \sigma f^\sigma
(\mu,\e) $ for any $ \sigma = \pm $. 
This directly implies the  following lemma, cfr. \cite[Lemma 2.3]{BMV4}.

\begin{lem} {\bf (Matrix representation of $ \cL_{\mu,\e}$ on $ \mathcal{V}_{\mu,\e}$)}
\label{lem:katored} For any integer $\tp\geq 2 $,
the operator $\cL_{\mu,\e}: \mathcal{V}_{\mu,\e}\to\mathcal{V}_{\mu,\e} $ in 
\eqref{WW} 
 is  represented w.r.t. the basis ${\cal F}$ in \eqref{basisF} by the $2\times 2$ Hamiltonian and reversible matrix
\begin{equation} \label{tocomputematrix}
\tL^{(\tp)}(\mu,\e) =\tJ\tB^{(\tp)}(\mu,\e), \
\tJ:=\begin{pmatrix} -\im & 0 \\ 0 & \im \end{pmatrix} \, , \quad \forall
(\mu, \e) \in B_{\delta_0}(\underline \mu)\times B_{\e_0}(0) \, , 
\end{equation}
where
\begin{equation}\label{tocomputematrixB}
\tB^{(\tp)}(\mu,\e)= \begin{pmatrix}
( \mathfrak{B}
(\mu,\e) f^+_{\tp}, f^+_{\tp}) 
& 
(\mathfrak{B}(\mu,\e) f^-_0, f^+_{\tp}) \\
(\mathfrak{B}(\mu,\e) f^+_{\tp}, f^-_{0}) 
&
(\mathfrak{B}(\mu,\e) f^-_{0}, f^-_{0})
\end{pmatrix} =: \begin{pmatrix} \alpha^{(\tp)}(\mu,\e) & \im \beta^{(\tp)}(\mu,\e) \\ -\im \beta^{(\tp)}(\mu,\e) & \gamma^{(\tp)}(\mu,\e) \end{pmatrix}  \, ,
\end{equation}
with 
\begin{equation}\label{Bgotico}
\mathfrak{B} (\mu,\e) :=
\mathfrak{B}^{(\tp)} (\mu,\e) := [P_{\underline \mu,0}]^* \, 
[U_{\mu,\e}]^* \, {\cal B}(\mu,\e) \, U_{\mu,\e} \, P_{\underline \mu,0}
\, . \end{equation} 
It results 
\begin{equation}\label{unperturbed}
\tL^{(\tp)}(\mu,0)=\begin{pmatrix} \im \omega_{\tp}^+(\mu)  & 0 \\ 0 &   \im \omega_{0}^-(\mu) \end{pmatrix}, \quad \text{i.e.} \quad 
\tB^{(\tp)}(\mu,0)= \begin{pmatrix} -\omega_{\tp}^+(\mu)  & 0 \\ 0 &  \omega_{0}^-(\mu) \end{pmatrix} \, ,
\end{equation}
where 
$\omega_j^{\sigma}(\mu):=\omega^\sigma(j+\mu,\tth)$ is in \eqref{omeghino} and 
$\omega_{\tp}^+(\umu)=\omega_0^-(\umu)=\omega_*^{(\tp)}(\tth) $ in \eqref{constants}.
Finally the functions $ \alpha^{(\tp)},  \beta^{(\tp)},  \gamma^{(\tp)} $ 
are analytic w.r.t.  $(\mu,\e)\in B_{\delta_0}(\umu)\times B_{\e_0}(0) $.
\end{lem}

\begin{rmk}
The functions $ \alpha^{(\tp)}(\mu, \e),  \beta^{(\tp)}(\mu, \e) $ and $  \gamma^{(\tp)}(\mu, \e) $ 
are also analytic  
w.r.t. $ \tth > \tth_*  $
for any $ \tth_* > 0 $,  
provided $\e $ is small 
(depending on $ \tth_* $).  This analyticity property 
is inherited  
from the projectors $ P(\mu,\e) $ 
and the similarity operators 
$ U_{\mu,\e} $ of Lemma \ref{lem:Kato1} which are 
are themselves analytic in $ \tth > 0 $ due to the properties of 
$ {\cal L}_{\mu,\e} 
(\tth)$, as detailed in  Remark \ref{analytic}. 
\end{rmk}

The eigenvalues of the matrix $\tL^{(\tp)}(\mu,\e)$ 
in \eqref{tocomputematrix} are 
\begin{equation}\label{eigenvalues}
\lambda^\pm (\mu,\e) = \tfrac\im2 S^{(\tp)}(\mu,\e)\pm\tfrac12\sqrt{
D^{(\tp)}(\mu,\e) 
}
\end{equation}
where
 \begin{align}\label{traceB}
 & S^{(\tp)}(\mu,\e):= -\im \text{Tr}\,\tL^{(\tp)}(\mu,\e) = \gamma^{(\tp)}(\mu,\e) - \alpha^{(\tp)}(\mu,\e)\, ,  \\
 & 
 D^{(\tp)}(\mu,\e) :=\text{Tr}^2 \tL^{(\tp)}(\mu,\e)  - 4 \det \tL^{(\tp)}(\mu,\e)= 4 (\beta^{(\tp)})^2 (\mu,\e) -(T^{(\tp)})^2 (\mu,\e)  \, , 
\label{discriminant} 
 \end{align}
 with  $$
 T^{(\tp)}(\mu,\e):= \text{Tr}\,\tB^{(\tp)}(\mu,\e) = \alpha^{(\tp)}(\mu,\e) + \gamma^{(\tp)}(\mu,\e) \,. 
 $$
In the following 
sections we will prove 
the existence of 
 eigenvalues of the matrix $\tL^{(\tp)}(\mu,\e)$ with nonzero real part when $(\mu,\e) $ is close to  $(\umu,0) $.
 We will apply the following abstract results.

\paragraph{Abstract instability criterion.}

Consider a $ 2 \times 2 $ Hamiltonian and reversible matrix 
\begin{equation} \label{tocomputematrix1}
\tL(\mu,\e):= \tJ\tB (\mu,\e), \quad
\tB (\mu,\e) :=\begin{pmatrix} \alpha (\mu,\e)  & \im \beta (\mu,\e) \\ 
- \im \beta (\mu,\e)  & \gamma (\mu,\e)  
\end{pmatrix} \, ,
\end{equation}
where $  \alpha (\mu,\e),  \beta (\mu,\e),  \gamma (\mu,\e) $ 
satisfy the following

\begin{ass} \label{assH}
The functions $\alpha(\mu,\e)$, $\beta(\mu,\e)$, $\gamma(\mu,\e)$ are  
 real analytic 
in the domain $ B_{\delta_0}(\underline \mu)\times B_{\e_0}(0) $
and admit the expansions, for some integer $\tp\geq 2 $ 

\begin{subequations}\label{matrixentries}
\begin{align}\label{expa}
 \alpha(\umu+\delta,\e) & = -\alpha_0(\delta)
 + \alpha_2 \e^2   +  r(\e^4,\delta\e^2)\, , \\
\label{expb}
  \beta (\umu+\delta,\e) & = \beta_1 \e^\tp+   \beta_2 \delta \e^\tp+ 
  r (\e^{\tp+2}, \delta^2\e^\tp)\, , \\
    \label{expc}
  \gamma (\umu+\delta,\e) & =\gamma_0(\delta)
  + \gamma_2 \e^2   +r(\e^4,\delta\e^2)\, ,
\end{align}
\end{subequations}
where
\begin{equation}\label{epsilonspento}
\alpha_0(\delta) = 
\omega_*^{(\tp)}-\alpha_1  \delta + r(\delta^2)\, , 	\qquad 
\gamma_0(\delta) = 
\omega_*^{(\tp)} + \gamma_1 \delta  + r(\delta^2) \, ,  
\end{equation}
and 
\begin{equation}\label{alpha1gamma1}
\alpha_1 + \gamma_1 > 0 \, . 
\end{equation}
\end{ass}
We shall prove in Section \ref{matrixcomputation} that
for any $\tp\geq 2 $ the matrix $\tL^{(\tp)}(\mu,\e)$ in \eqref{tocomputematrix}  
satisfies Assumption \ref{assH}.   

The matrix
$ \tL(\mu,\e) $ has eigenvalues with non zero real part if and only if the discriminant 
\begin{equation}\label{Dmuep}
\begin{aligned}
 D(\mu,\e) = 4 \beta^2 (\mu,\e) - T^2 (\mu,\e) & = - d_+(\mu,\e)d_-(\mu,\e)  \\
& \qquad  \text{where}
 \quad
 d_\pm  (\mu,\e) :=  T (\mu,\e) \pm 2\beta (\mu,\e)
 \end{aligned}
 \end{equation} 
is positive. This holds if
$ d_\pm (\mu,
\epsilon)$ have different signs and
$\beta(\mu,\e)$ is large enough with respect  to the
trace of $ T(\mu,\e):= \alpha (\mu,\e) + \gamma (\mu,\e) $.   
Let us analyze these quantities.

By Assumption \ref{assH}
 the trace 
 $ T(\mu,\e) = \alpha(\mu, \e) + \gamma(\mu, \e) $ admits the expansion
 \begin{equation}\label{expT}
\begin{aligned}
& T(\umu+\delta,\e)  = T(\umu+\delta,0) +T_2 \e^2 + r(\e^{4},\delta\e^2)  \quad \text{where} \quad 
T_{2}:=\alpha_2+\gamma_2 \quad \text{and}  \\
 & T(\umu+\delta,0)=  
 T_1 \delta + r(\delta^2) 
 \quad \text{where} \quad  T_{1} := \alpha_1+\gamma_1  > 0 \, .
 \end{aligned}
\end{equation}
By \eqref{expT} and the analytic implicit function theorem there exists $\e_1 \in (0, \e_0) $ such that, in 
 $ B_{\delta_0}(\umu)\times  B_{\e_1 }(0)$, 
the sets $\{ T(\mu,\e) = 0 \} $ 
and  
$ \{ d_\pm  (\mu,\e) = 0 \} $ 
 are
graphs of analytic functions 
\begin{equation}\label{impfunc}
\mu_0 \, , \mu_\pm : 
(-\e_1 ,\e_1 ) \mapsto \mu_0 (\e) \, , \ \mu_\pm  (\e) \, ,\ \text{i.e.}\ 
 T (\mu_0 (\e),\e) \equiv 0\ \text{ and }\ d_\pm  (\mu_\pm (\e),\e  ) \equiv 0 \, , 
 \end{equation}
admitting, by \eqref{expT}, 
 the Taylor expansions
\begin{equation}\label{muexpa}
 \mu_0 (\e) = \umu - \frac{T_2}{T_1} \e^2 + r(\e^{4}) \, , 
\quad 
\mu_\pm (0) =  \umu \, , \  \mu_\pm' (0) =  0 \, . 
\end{equation}
Lemma \ref{tangemu} proves that the functions $\mu_\pm (\e ) $ are 
actually $ \cO(\e^\tp) $-close to  $ \mu_0 (\e) $. 

Since $T_1 > 0 $ (cfr. \eqref{expT}), the functions $d_\pm(\mu,\e) $ are strictly positive (resp. negative) for $\mu>\mu_\pm(\e)$ (resp.  $\mu<\mu_\pm(\e)$). In addition, since $d_\pm(\mu_0(\e),\e ) =\pm 2\beta(\mu_0(\e),\e)   $  we deduce that for any $\e \in (-\e_1,\e_1) $
\begin{equation}\label{maxalternative}
\begin{cases}
\mbox{if }\beta(\mu_0(\e),\e) >0 & \mbox{then }\ \mu_+(\e) <\mu_0(\e) < \mu_-(\e)\,, \\
\mbox{if }\beta(\mu_0(\e),\e) <0 & \mbox{then }\ \mu_-(\e) <\mu_0(\e) < \mu_+(\e)\,, \\
\mbox{if }\beta(\mu_0(\e),\e) =0 & \mbox{then }\ \mu_0(\e)= \mu_+(\e) = \mu_-(\e)\,.
\end{cases}
\end{equation}
The graphs of these functions are represented  in Figure \ref{fig:instareg}. 
By \eqref{expb}
and  the expansion of $ \mu_0 (\e) $ in \eqref{muexpa},  
we have 
\begin{equation}\label{expabmuo}
\beta(\mu_0(\e),\e) = \beta_1 \e^\tp+ r(\e^{\tp+ 2})  \, .
\end{equation}
Thus, if  $\beta_1 \neq 0$ 
the function $\beta(\mu_0(\e),\e)$
is not identically zero
and only the first two cases 
in \eqref{maxalternative} occur so that,   for $ |\e| < \epsilon_1 $, we have  
$\mu_\wedge(\e) < \mu_0(\e) < \mu_\vee(\e) $ 
where
\begin{equation}\label{muminmumax}
\mu_\wedge (\e):= \min\big\{\mu_+ (\e),\mu_- (\e)\big\} \quad \text{and}\quad \mu_\vee (\e):= \max\big\{\mu_+ (\e),\mu_- (\e)\big\}\, . 
\end{equation}

\begin{teo} Assume that the  $ 2 \times 2 $ Hamiltonian and reversible 
 matrix $\tL(\mu,\e)= \tJ \tB(\mu,\e)$ in \eqref{tocomputematrix1} 
 satisfies Assumption \ref{assH}, and that the coefficient $\beta_1 $ in   \eqref{expb} is not zero.
 Then 
the spectrum of $ \, \tL(\mu,\e) $, for $(\mu,\e) $ close to $ ( \umu, 0 ) $, consists of two  eigenvalues $\lambda^\pm(\mu,\e) $ with 
non-zero real part 
if and only if $(\mu,\e)$ lies inside 
the open region 
\begin{equation}\label{regionR}
\begin{aligned}
R:&= 
\big\{ 
(\mu, \e) \in B_{\delta_0}(\umu)\times B_{\e_1}(0)  \ : \    \mu_\wedge(\e)<\mu< \mu_\vee(\e)\big\} \, , \end{aligned}
\end{equation} 
 whereas, for $(\mu,\e)\notin R $, the eigenvalues $\lambda^\pm(\mu,\e)$ are purely imaginary,
 see Figure \ref{fig:instareg}. 
 \end{teo}
 
 \begin{proof}
The eigenvalues of 
$ \tL(\mu,\e )$ 
have the form  \eqref{eigenvalues}
and have nonzero real part if and only if  
$ D(\mu,\e) > 0 $, cfr.\eqref{Dmuep}, 
which  happens if and only if $(\mu,\e)\in R$. 
 \end{proof}

It  is  convenient to 
translate the Floquet exponent 
$ \mu $ around the value $ \mu_0 (\e) $ where  $ T(\mu_0(\e) ,\e)\equiv0 $, cfr. \eqref{impfunc},   introducing the new parameter $\nu$  such that
\begin{equation}\label{def:nu}
\mu = \mu_0(\e) + \nu\, , \quad\text{i.e.} \quad \nu := \delta+\umu-\mu_0(\e)\, .
\end{equation}
Accordingly we write the functions   $\mu_\pm(\e)$ in \eqref{muexpa} as
$\mu_{\pm} (\e) = \mu_0(\e) + \nu_\pm(\e) $.

\begin{lem}\label{tangemu}
The analytic functions $ \nu_\pm(\e) := \mu_\pm(\e)-\mu_0(\e) $ admit the expansion 
\begin{equation}\label{expnuendsofficial}
\nu_\pm(\e)  =  \mp  \frac{2 \beta_1}{T_1} \e^\tp+ r(\e^{\tp+2})\,  .
\end{equation} 
\end{lem}

\begin{proof}
The functions $ \nu_\pm (\e) $ solve 
$$
d_\pm ( \mu_0 (\e) + \nu_\pm (\e), \e ) = T( \mu_0 (\e) + \nu_\pm (\e), \e) \pm 
2 \beta  ( \mu_0 (\e) + \nu_\pm (\e), \e ) = 0 \, .
$$ 
Expanding this identity at $ \mu = \mu_0 (\e) $ we get
$$
\nu_\pm (\e) = \mp \frac{2 \beta (\mu_0 (\e), \e )}{ (\pa_\mu T)(\mu_0(\e))- 
2 (\pa_\mu \beta) (\mu_0 (\e))} + r (\nu_\pm^2) (1 + r(\e)) \, , 
$$
and since 
$ (\pa_\mu T)(\mu_0(\e))- 
2 (\pa_\mu \beta) (\mu_0 (\e)) = T_1 (1 + r(\e)) $ 
(cfr.  \eqref{expT}, \eqref{expb}) we deduce, by \eqref{expabmuo},  
\eqref{expnuendsofficial}. 
\end{proof}

We define
\begin{equation}\label{numpiu}
\nu_\wedge(\e) := \min \{ \nu_+(\e), \nu_-(\e) \} < 0 \, , \quad 
\nu_\vee (\e) := \max \{ \nu_+(\e), \nu_-(\e) \} > 0 \, .  
\end{equation}
Note that $\nu_\wedge(\e)$ and $\nu_\vee(\e) $ are the points where the discriminant $ D(\mu_0(\e)+\nu,\e) $ in \eqref{Dmuep}
vanishes, whereas for $\nu_\wedge(\e) <\nu <\nu_\vee(\e)$ it results 
$ D(\mu_0(\e)+\nu,\e) >0 $.

 \begin{teo}[\bf Unstable eigenvalues] \label{primeisola}
Under Assumption \ref{assH}
the matrix $\tL(\mu,\e)$ in \eqref{tocomputematrix1} has the two  
eigenvalues 
\begin{equation}\label{eigs}
\lambda^\pm(\mu_0(\e)+\nu,\e) = \begin{cases}   \frac\im2 S(\mu_0(\e)+\nu,\e) \pm  \frac\im2 \sqrt{ |D(\mu_0(\e)+\nu,\e)|} & \text{if }\nu \leq \nu_\wedge(\e)\textup{ or }\nu \geq \nu_\vee(\e)\, , \\[2mm]
  \;\frac\im2 S(\mu_0(\e)+\nu,\e) \pm \frac12 \sqrt{ D(\mu_0(\e)+\nu,\e)}&\text{if }\nu_\wedge(\e) <\nu <\nu_\vee(\e)\, ,
 \end{cases} 
\end{equation}
where $ \mu_0 (\e ) $ is defined in \eqref{impfunc} and 
\begin{align}\label{expD}
D(\mu_0(\e)+\nu,\e) 
& = 4\beta_1^2 \e^{2\tp}-T_1^2 \nu^2+8\beta_1\beta_2 \e^{2\tp}\nu +r({\beta_1\e^{2\tp+2},
\nu \e^{2\tp+2} }, \e^{{2\tp+{4}}},\nu^2\e^{{2}},\nu^3) \\
 S(\mu_0(\e)+\nu,\e) 
& = 
2\omega_* + (\gamma_1-\alpha_1)\nu +  \Big( \gamma_2-\alpha_2 - (\gamma_1-\alpha_1) \frac{T_2}{T_1} \Big) \e^2+ r(\e^{4}, \nu\e^2, \nu^2)  \, . \label{expS}
 \end{align}
The maximal absolute value of the real  part 
of the 
 eigenvalues in \eqref{eigs} is 
\begin{equation} \label{expnuRe}
\max \mathfrak{Re}\, \lambda^\pm(\mu_0(\e)+\nu,\e) = |\beta_1|\e^\tp + r(\e^{{\tp+2}})\, .
\end{equation}
{\bf \emph{(Isola)}.} 
If  the coefficients $ \beta_1, \alpha_1, \gamma_1 $ in \eqref{expa}-\eqref{expc}
satisfy $\beta_1\neq 0$ and $\alpha_1\neq \gamma_1$,
then, for   $ \e $ small enough, 
the pair of unstable eigenvalues $\lambda^\pm(\mu_0(\e)+\nu,\e)$ depicts in the complex $\lambda$-plane, as $\nu $ varies
in  $ (\nu_\wedge(\e), \nu_\vee(\e)) $ a closed analytic  curve that
  intersects orthogonally the imaginary axis  
  and encircles a convex region. 
 \end{teo}
 
\begin{proof} 
  In view of \eqref{expb},  \eqref{muexpa} and \eqref{def:nu}, 
\begin{equation}\label{bnuexp}
\beta(\mu_0(\e)+\nu,\e) =  \beta_1 \e^\tp+ \beta_2 \nu\e^\tp+ r(\e^{\tp+{2}}, \nu^2 \epsilon^\tp) \, .
\end{equation} 
We now prove the expansion  \eqref{expD} of the discriminant, cfr. 
\eqref{Dmuep}, 
\begin{equation}\label{Dinte}
D(\mu_0(\e)+\nu,\e)= 4\beta^2(\mu_0(\e)+\nu,\e) - T^2(\mu_0(\e)+\nu,\e)\, .
\end{equation} 
By a Taylor expansion at $ \nu = 0 $, \eqref{impfunc}, \eqref{expT}, \eqref{muexpa} and  \eqref{def:nu}  we  get that 
\begin{equation}\label{Tnuexp}
T(\mu_0(\e)+\nu,\e) =(\pa_\mu T)(\mu_0(\e),\e)\nu + r(\nu^2) = T_1\nu + r(\e^{{2}}\nu,\nu^2) 
\end{equation}
and the expansion  \eqref{expD}  follows by inserting  
\eqref{Tnuexp} and \eqref{bnuexp} into \eqref{Dinte}. 
 By \eqref{eigenvalues} and \eqref{def:nu}, the eigenvalues of $ \tL(\mu,\e) $ 
 have the form \eqref{eigs}.
 The expansion  \eqref{expS} follows from \eqref{expa}-\eqref{expc}, \eqref{expT} and \eqref{muexpa}. 
 The (opposite) extremal real parts of the eigenvalues are attained at 
 $ \nu = \nu_{\mathfrak{Re}}$ where $
(\pa_\mu D)(\mu_0(\e)+\nu_{\mathfrak{Re}},\e) = 0$.
By \eqref{expD} we have the expansion
\begin{equation}\label{nuRemaxnondeg}
\nu_{\mathfrak{Re}}(\e) = 4\frac{\beta_1 \beta_2}{T_1^2} \e^{2 \tp}+r(\e^{2\tp+{2}})\, .
\end{equation}
By plugging \eqref{nuRemaxnondeg} into \eqref{expD}-\eqref{eigs} one obtains \eqref{expnuRe}, that holds if $ \beta_1 $ is different or equal to zero.

\medskip

\noindent
{\bf (Isola).} 
In view of \eqref{eigs},
for any $ \e $  small enough, 
the unstable eigenvalues  branch off from the imaginary axis at $\nu=\nu_\wedge(\e)$, evolve specularly  as $\nu$ increases and rejoin at $\nu=\nu_\vee(\e)$ thus forming a {\it closed} curve. If  $\alpha_1\neq \gamma_1$,   
the imaginary part of the eigenvalues
$$
 I(\nu,\e):= \mathfrak{Im}\,\lambda^\pm(\mu_0(\e)+\nu,\e)  
$$
 is monotone w.r.t. $\nu\in \big(\nu_\wedge(\e),\nu_\vee(\e)\big) $ for small $\e$, since 
 its derivative fulfills
 $$
  \pa_\nu I(\nu,\e)  \stackrel{\eqref{eigs}}{=} \frac{\gamma_1-\alpha_1}{2} + r(\e^2,\nu) \neq 0\, , \qquad \nu_\wedge(\e) \leq \nu \leq \nu_{\vee}(\e)\, . 
$$ 
Thus $\nu \mapsto I(\nu,\e)  $ is a diffeomorphism between the intervals $\big(\nu_\wedge(\e),\nu_\vee(\e)\big)$ and $\big(y_\wedge(\e),y_\vee(\e)\big)$ where 
$$y_\wedge(\e):=\min \big\{ I(\nu_\wedge(\e),\e)\, ,I(\nu_\vee(\e),\e)\big\}\, , \quad  y_\vee(\e):=\max \big\{ I(\nu_\wedge(\e),\e)\, ,I(\nu_\vee(\e),\e)\big\}\, .$$
Let $\nu(y,\e)$   denote  the inverse of $y=I(\nu,\e)$, with $y$ varying in $y_\wedge(\e) < y < y_\vee(\e)$. 
The curves depicted by the two unstable eigenvalues in the complex plane is hence covered by the two specular graphs 
$$
\Gamma_r := \big\{ \big( X(y,\e), y\big) \;:\; y_\wedge(\e) <y<y_\vee(\e) \big\}\, ,\quad 
\Gamma_l := \big\{ \big( -X(y,\e), y\big) \;:\; y_\wedge(\e) <y<y_\vee(\e) \big\}\, ,
$$
where 
$$
X(y,\e) := \frac12 \sqrt{D\big(\mu(y,\e), \e \big)}\, ,\quad \mu(y,\e):= \mu_0(\e) + \nu(y,\e)\, .
$$
It is easy to check 
that at the bottom and top of the isola, i.e.\ at $y=y_\wedge(\e)$ and $y=y_\vee(\e)$, the real parts
 $\pm X(y,\e) $
  of the unstable eigenvalues vanish with derivative that tends to infinity,  
and that 
the region encircled by the two graphs is convex. 
  \end{proof}

 An approximation  of 
 the  eigenvalues 
 $\lambda^\pm(\mu_0(\e)+\nu,\e)$ 
 in Theorem \ref{primeisola} is obtained  
neglecting the remainders $r(\nu^3)$ of $D(\mu_0(\e)+\nu,\e)$ in \eqref{expD} and $r(\nu^2)$ of $S(\mu_0(\e)+\nu,\e)$ in \eqref{expS}, namely by the complex numbers $\widetilde\lambda^\pm(\nu,\e) $ with 
\begin{equation}\label{truncation} 
\begin{cases}x:= \mathfrak{Re}\ \widetilde\lambda^\pm(\nu,\e) = \pm \frac12  \sqrt{4\beta_1^2 \e^{2\tp }(1+r(\e^{{2}})) -T_1^2 \nu^2(1+r(\e^{{2}}))+8\beta_1\beta_2\nu\e^{2\tp}
}\, ,\\ 
y:=  \mathfrak{Im}\ \widetilde\lambda^\pm(\nu,\e) =\omega_* + 
\Big(\dfrac{\gamma_2-\alpha_2}{2} - \dfrac{T_2(\gamma_1-\alpha_1)}{2T_1}\Big) 
\e^2  (1+r(\e^{{2}})) +\dfrac{\gamma_1-\alpha_1}{2}\nu (1+r(\e^2)) \, .
\end{cases} 
\end{equation} 
The functions $\widetilde\lambda^\pm(\nu,\e) $ are defined for $\nu$ in
 the interval $\widetilde\nu_\wedge(\e)\leq\nu\leq\widetilde\nu_\vee(\e) $ where the argument of the square root  in \eqref{truncation}  is non-negative. 
 These approximated  eigenvalues describe an ellipse  in the $(x,y)$-plane as stated in the following lemma which  directly 
 follows 
 from \eqref{truncation}. 

 \begin{lem}{\bf (Approximating ellipse)} 
 \label{lem:appell} 
 {Suppose the coefficients $ \beta_1, \alpha_1, \gamma_1 $ in \eqref{expa}-\eqref{expc}
satisfy $\beta_1\neq 0$ and $\alpha_1\neq \gamma_1$}. 
 Then, as $\nu$ varies between $\widetilde\nu_\wedge(\e)$ and $\widetilde\nu_\vee(\e)$
the approximated  eigenvalues $ \widetilde\lambda^\pm(\nu,\e)  $ in 
\eqref{truncation}  
span an ellipse of equation
$$
x^2  + \dfrac{T_1^2(1+r(\e^2))}{(\gamma_1-\alpha_1)^2}(y-y_0(\e))^2 =   \beta_1^2 \e^{2\tp} (1+r(\e^{{2}})) \, ,
$$
centered at $(0,y_0(\e))$ where 
$
y_0(\e) = \omega_*+\Big( \dfrac{\gamma_2-\alpha_2}{2}- 
 \dfrac{T_2(\gamma_1-\alpha_1)}{2T_1}\Big)\e^2+r(\e^{4}) $. 
\end{lem}

\section{Taylor expansion of $\mathtt{L}^{(\tp)}(\mu,\e)$}\label{matrixcomputation}

This section is dedicated to demonstrating that the entries of the matrix $\tL^{(\tp)}(\mu,\e)$ given by \eqref{tocomputematrix} 
admit a Taylor expansion as stated in  
Assumption \ref{assH} for any integer $\tp\geq 2 $.
This is proved in Theorem \ref{lem:expansionL}. 
To begin, we  introduce some notation.

\begin{ntt}\label{notazioneoperatori}
Given a family of linear 
operators 
$A=A(\mu,\e)$, 
analytic in $ (\mu, \e) $,  we define, for any 
$ i,n, \ell \in \bN_0 $, its jets $ A_{i,n}$ and its homogenous component $ A_\ell $ of degree $ \ell $,  
\begin{equation}\label{notazione}
A_{i,n} := \frac{1}{i!n!} \big(\pa^i_\mu\pa^n_\e A \big)(\umu,0) \delta^i \e^n \qquad 
\text{and} \qquad 
A_\ell := \sum_{\substack{i+n=\ell,
\ \atop  i,n\geq 0}} \!\!\! A_{i,n}  \, . 
\end{equation}

\noindent 
{\bf Matrix representation.}
In the sequel we identify a linear
(possibly unbounded) 
operator $A $  acting on $ C^\infty (\bT, \bC)$ as the infinite matrix $ \{ A^j_k \}_{j,k\in \bZ}$  with respect to the exponential basis, namely 
with  {\em matrix entries} 
\begin{equation}\label{def.matrix.elements}
A^{j}_{k} := ( A e^{\im j x} , \ e^{\im k x } )_{L^2(\bT)} \, , 
\quad (f,g) = \frac{1}{2\pi} \int_{\bT} f(x) \overline{g(x)} \, d x \, . 
\end{equation}
Thus, in purely algebraic terms, 
the action of $A$ 
is  given  by  
$$
 h(x) = \sum_{j \in \bZ} h_j e^{\im j x} \mapsto  
(A h)(x) = \sum_{k\in \bZ} 
\Big( \sum_{j \in \bZ} A^j_k \, h_j \Big) \, e^{\im k x} \, .
$$
We do not concern about issues of convergence
as we will solely deal with finitely many matrix entries.  


Given $\kappa \in \bZ$, we define its
``$\kappa$-band" 
 $A^{[\kappa]} $ as the operator 
$ A^{[\kappa]} \equiv \{A^{j}_{k} \}_{k - j = \kappa} $ 
with matrix coefficients  supported on the 
``band" $ k-j = \kappa $. In other words the action of 
$A^{[\kappa]} $  is  to ``shift 
the exponential $ e^{\im j x}$  of 
$ \kappa $ harmonics", namely 
\begin{equation}\label{band.def}
A^{[\kappa]} (e^{\im j x})  =
A^{j}_{ j + \kappa} e^{\im (j + \kappa) x} \, . 
\end{equation}
If $A = \begin{bmatrix}
    A_1 & A_2 \\
    A_3 & A_4
\end{bmatrix}$ is a $2\times 2$ matrix of operators acting on
$ C^\infty (\bT, \bC^2)$  
we define its $\kappa$-band as the
operator 
\begin{equation}\label{Am4}
A^{[\kappa]}:= 
\begin{bmatrix}
  A^{[\kappa]}_1 & A^{[\kappa]}_2 \\
  A^{[\kappa]}_3 & A^{[\kappa]}_4
\end{bmatrix}.
\end{equation}
We also  introduce the following ``entanglement coefficients" 
which are matrix entries of the jets 
$ {\cal B}_{0,\ell}^{[\kappa ]} $ of the operator 
$  \cB (\mu,\e) $  with respect to 
the basis of eigenvectors $ f_j^{\sigma} (\umu )$  of $ \cL_{\umu,0}  $.  

\begin{sia}[{\bf entanglement coefficients}] 
\label{entanglementcoefficients}
 For any  $ \ell\in \bN_0 $, integers $ \kappa,j,j'\in \bZ $ and signs $ \sigma,\sigma'=\pm $, 
 we define the entanglement coefficients 
\begin{equation}\label{entcoeff}
\ent{\ell}{\kappa }{j'}{j}{\sigma'}{\sigma}
:= \big({\cal B}_{0,\ell}^{[\kappa ]} f_j^\sigma, f_{j'}^{\sigma'}  \big)   
\end{equation}
where  $  \cB (\mu,\e) $ is the 
operator in \eqref{WW}
and 
$ f_j^{\sigma} := f_j^{\sigma} (\umu )$ are the  eigenvectors of $ \cL_{\umu,0}  $ in 
\eqref{def:fsigmaj}. 
\end{sia}
By \eqref{notazione}, \eqref{starij} and the self-adjointness of $\cB(\mu,\e)$, the entanglement coefficients 
fulfill
\begin{equation}\label{entprop}
\ent{\ell}{\kappa }{j'}{j}{\sigma'}{\sigma} = 0 
\quad 
\text{ if }j'\neq j+\kappa \, ,\quad\text{ and }\quad \bar{\ent{\ell}{\kappa }{j'}{j}{\sigma'}{\sigma}} =\ \ent{\ell}{-\kappa }{j}{j'}{\sigma}{\sigma'}\, . 
\end{equation} 
We further  denote
\begin{equation}\label{omegajsigma}
\omega_{j}^{\sigma}:=
\omega_{j}^{\sigma}(\tth)  := \omega^{\sigma}\big(j+\uphi(\tp,\tth),\tth\big)  \, ,\quad j \in \bZ\, ,\ \sigma=\pm\, , 
\end{equation}
where $ \omega^\sigma(\varphi,\tth)$ are defined in \eqref{Oomegino}
and 
$  \uphi(\tp,\tth) $ in Lemma  \ref{collemma}. We set 
\begin{equation}\label{tjp}
\begin{aligned}
& \Omega_j^{(\tp)} := \Omega_j^{(\tp)} (\tth) := \sqrt{(j+\uphi(\tp,\tth))\tanh(\tth(j+\uphi(\tp,\tth)))}
  \, , \\
&   t_j^{(\tp)} :=t_j^{(\tp)}(\tth) := \sqrt{\frac{j+\uphi(\tp,\tth)}{\tanh(\tth(j+\uphi(\tp,\tth)))}} \,  .
\end{aligned}
 \end{equation}

\begin{teo}{\bf (Taylor expansion of $ B^{(\tp)}(\mu,\e) $) }\label{lem:expansionL}  
 For any $\tp \geq 2$, the 
 $ 2 \times 2 $ Hermitian
 matrix 
$$
\tB^{(\tp)}(\mu,\e)= 
\begin{pmatrix} \alpha^{(\tp)}(\mu,\e) & \im \beta^{(\tp)}(\mu,\e) \\ -\im \beta^{(\tp)}(\mu,\e) & \gamma^{(\tp)}(\mu,\e) \end{pmatrix}  
$$
defined in \eqref{tocomputematrixB}  
is analytic 
in $ B_{\delta_0}(\underline \mu)\times B_{\e_0}(0) $,
its entries satisfy 
\begin{equation}\label{paritaentri}
    \e \mapsto \alpha^{(\tp)}(\mu,\e)\,,\;\gamma^{(\tp)}(\mu,\e) \text{ are even,}\qquad \e \mapsto \beta^{(\tp)}(\mu,\e)\text{ is }
    \begin{cases}
    \text{even} &\text{if }\tp\text{ is even} \, , \\
    \text{odd} &\text{if }\tp\text{ is odd} \, , 
    \end{cases}
\end{equation}
and the properties of 
Assumption \ref{assH}: \\[1mm]
{\bf \textit{I)} 
{\bf (Off-diagonal entry)}} The real function 
$\beta^{(\tp)} (\umu + \delta, \e) $  admits the Taylor expansion 
\begin{equation}\label{ciochevoglio}
 \beta^{(\tp)} (\umu+\delta,\e)  = 
 \beta_1^{(\tp)}(\tth) 
 \e^\tp 
 + 
  r ( \delta \e^\tp, \e^{\tp+2})  
\end{equation}
where 
\begin{equation}\label{beta1exp} 
\begin{aligned}
&\beta_1^{(\tp)}(\tth) \e^\tp= b_0^{(\tp)} \e^\tp+ \sum_{q=1}^{\tp-1} \sum_{0<j_1<\dots <j_q < \tp} \sum_{\sigma_1,\dots,\sigma_q = \pm }\sigma_1 \dots \sigma_q \betone{q}{j_1,\dots,j_q}{\sigma_1,\dots,\sigma_q} \e^\tp\, ,\\
&  b_0^{(\tp)} \e^\tp:=         \frac{1}{\im} \, 
\ent{\tp}{\tp}{\tp}{0}{+}{-}\, , \quad  \betonep{q}{j_1,\dots,j_q}{\sigma_1,\dots,\sigma_q}\e^\tp:= \frac{1}{\im} \, \dfrac{\ent{j_1}{j_1}{j_1}{0}{\sigma_1}{-}\ent{j_2-j_1}{j_2-j_1}{j_2}{j_1}{\sigma_2}{\sigma_1}\dots \ent{\tp-j_q}{\tp-j_q}{\tp}{j_q}{+}{\sigma_q}}{ (\omega_{j_1}^{\sigma_1}-\omega_*^{(\tp)})\dots(\omega_{j_q}^{\sigma_q}-\omega_*^{(\tp)})} \,,
\end{aligned}
\end{equation}

$ \bullet $ the $\omega_{j}^{\sigma}$ are defined in \eqref{omegajsigma} and
$\omega_*^{(\tp)}$ is defined in \eqref{intcollision};

$ \bullet $
the entanglement coefficients have the explicit expression
\begin{equation}\label{entanglementfdj}
\ent{\ell}{\ell}{j+\ell}{j}{\sigma'}{\sigma} =\frac{\e^\ell}4 \sqrt{\sigma}\bar{\sqrt{\sigma'}}\sqrt{\Omega_j^{(\tp)} \Omega_{j+\ell}^{(\tp)}} \big(a_\ell^{[\ell]} - p_\ell^{[\ell]} ( \sigma t_j^{(\tp)} +\sigma' t_{j+\ell}^{(\tp)}) \big)\, ,\quad j,\ell \in \bN_0 \, ,\ \ \sigma,\sigma'=\pm\, ,
\end{equation}
where 
 $ p_\ell^{[\ell]}$, $a_\ell^{[\ell]} \in \bR $ are the maximal Fourier coefficients of the functions $p_\ell (x) $, $a_\ell (x) $ in \eqref{leadingLin} and 
 $\Omega_j^{(\tp)}$, $t_j^{(\tp)}$ are in \eqref{tjp};

$ \bullet $ the function
$ \beta_1^{(\tp)} :
(0, + \infty) \to \bR $, $
\tth \mapsto  \beta_1^{(\tp)} (\tth) $,  
 is  analytic. 
 \\[1mm]
 {\bf \textit{II}) 
 (Diagonal entries)}. 
 The entries $\alpha^{(\tp)}(\umu + \delta ,\e)$ and $\gamma^{(\tp)}(\umu + \delta ,\e)$ admit 
 the Taylor  expansions
 \eqref{expa}, \eqref{expc}, 
 \begin{equation}
 \begin{aligned}\label{abc}
&\alpha^{(\tp)}(\umu+\delta, \e) = 
-\omega_\tp^+(\umu+\delta) 
+  \alpha_2^{(\tp)}
(\tth) \e^2  
+ r(\e^4,\delta\e^2) \,,\\
&  \gamma^{(\tp)}(\umu+\delta, \e) =  \omega_0^-(\umu+\delta) 
+ \gamma_2^{(\tp)}(\tth)  \e^2  
+   r(\e^4,\delta\e^2)\, ,
 \end{aligned}
 \end{equation} 
with   
\begin{equation}\label{omegappomega0m}
\omega_\tp^+(\umu+\delta) = \omega_*^{(\tp)}-\alpha_1^{(\tp)}(\tth)\delta + r(\delta^2) \, ,\quad 
\omega_0^-(\umu+\delta) = \omega_*^{(\tp)}+\gamma_1^{(\tp)}(\tth)\delta + r(\delta^2) \, ,
\end{equation}
and  
\begin{equation}
\label{a1g1}
\alpha_1^{(\tp)}(\tth)+\gamma_1^{(\tp)}(\tth)>0 \, , 
\quad
\gamma_1^{(\tp)}(\tth) - 
\alpha_1^{(\tp)}(\tth) > 0 \, , 
\quad 
\forall \tth > 0 \, . 
\end{equation}
The functions $\alpha_2^{(\tp)}(\tth)$, $\gamma_2^{(\tp)}(\tth)$ are analytic in  $\tth$. 
\end{teo}

The whole  section is devoted to the proof of Theorem \ref{lem:expansionL}. 

We first give simple properties of an 
 operator $ A $ as in \eqref{Am4}.  
  The adjoint of the $ \kappa $-band  operator $ A^{[\kappa]} $ is  
\begin{equation}\label{starij}
\big[A^{[\kappa]}\big]^* = (A^*)^{[-\kappa]} \, . 
\end{equation}
Formally each $ A$ is the sum of its $ \kappa $-bands
\begin{equation}\label{A.exp.band}
    A = \sum_{\kappa \in \bZ} A^{[\kappa]} \, \, , 
\end{equation} 
and the $ \kappa $-band of the composed operator is 
\begin{equation}\label{bandAB}
    (A\circ B)^{[\kappa]} = \sum_{\kappa_1 + \kappa_2 = \kappa} A^{[\kappa_1]} \circ B^{[\kappa_2]} \, . 
\end{equation}
The $ \kappa $-band operator 
 associated to 
 the multiplication operator 
  for a matrix of functions $ a(x) $ is the multiplication operator for the $\kappa $-harmonic  of $ a(x)$, namely 
\begin{equation}\label{karmonica}
a_{\kappa} e^{\im \kappa x } :\,  \  
h(x) \mapsto a_{\kappa} e^{\im \kappa x } h(x) 
\qquad \text{where} \qquad 
a_\kappa := \frac{1}{2 \pi}
\int_{\bT} a(x) e^{-\im \kappa x}
d x \, . 
\end{equation}
For any $ \kappa \in \bN $ the $ \kappa $-band operator 
$ [g(D)]^{[\kappa]} $ of a Fourier  multiplier $ g(D) $ is zero. 
  \end{ntt}
The following definition encodes an important  structure of the operators  $\cB(\mu,\e)$ in \eqref{WW}, $P(\mu,\e)$ in \eqref{Pproj} and $\mathfrak{B} (\mu,\e) $ in \eqref{Bgotico}. 
Roughly they are operators   whose jets $A_{i,j}$ have ``finite-range interactions" of order at most $j$ and with the same parity of $j$.

\begin{sia}[{\bf Spaces $\mathfrak{F}_n $ and  $\mathtt{F}_\ell$}]
\label{defFell} 
For any $ n \in \bN_0$ we define $\mathfrak{F}_n $ the space of $2\times 2$ matrices  of  operators $A$ fulfilling
\begin{equation}\label{TFj} 
 A^{[\kappa]} =  
0 \quad   \textup{ if }|\kappa| > n \quad 
 \textup{ or} \quad \kappa \nequiv n \mod{2}\, . 
\end{equation}
We denote by 
$\mathtt{F}_\ell$  the space of $2\times 2$  matrices  of operators  $A$
(cfr.  \eqref{notazione}) 
such that
$$ 
A =    A_\ell = 
\sum_{\substack{i+n=\ell \\ i,n\geq 0} } A_{i,n}  \quad
\text{such that} \quad 
A_{i,n} \in \mathfrak{F}_n \, . 
$$ 
In particular $ A_\ell^{[\kappa]}=  0 $ if  $ |\kappa|>\ell $. 
\end{sia}
The following properties of the spaces 
$ \mathtt{F}_\ell$ will be used repeatedly. 
\begin{lem}\label{TFjprop}
Let $ \ell, \ell' \in \bN_0 $
and $ A \in \mathtt{F}_{\ell}$, $ B \in \mathtt{F}_{\ell'}$.  Then 
\begin{itemize}
 \item[(i)] {\bf Highest interaction:} $A_\ell^{[\pm \ell]} = A_{0,\ell}^{[\pm \ell]}$;
    \item[(ii)] {\bf Composition:} $A \circ B \in \mathtt{F}_{\ell+\ell'}$ with 
    $    (A\circ B)^{[\pm(\ell+\ell')]}= A^{[\pm \ell]}_{0,\ell} \circ B^{[\pm \ell']}_{0,\ell'}$;
    \item[(iii)] {\bf Adjoint:} $A^* \in \mathtt{F}_{\ell}$; 
    \item[(iv)]  {\bf Finite range interaction:} For any $ v_1, v_2 \in \bC^2$,
    \begin{equation}\label{AFlele}
( A_{i,n} v_1 e^{\im j_1 x} , 
v_2 e^{\im j_2 x} ) = 0 \ \ \mbox{ if } |j_1 - j_2| > n \mbox{ or } j_1 - j_2 \not\equiv n \mod 2 \ . 
\end{equation}
\end{itemize}
\end{lem}

\begin{proof}
$(i)$ The only possibility to have harmonics $[\pm \ell]$ is that $ n = \ell$, namely $i = 0$.  
$(ii)$ By Definition 
\ref{defFell} of the space  $\mathtt{F}_\ell$ and using \eqref{A.exp.band}, we have
$$
\begin{aligned}
& A\circ B = \sum_{{i_1 + n_1 = \ell \, , \atop i_2 + n_2 = \ell' \, , } \atop \kappa_1, \kappa_2    \in \bZ } A_{i_1, n_1}^{[\kappa_1]}\circ B_{i_2, n_2}^{[\kappa_2]} 
= 
\sum_{i+n = \ell + \ell', 
\atop i,n \geq 0, \kappa \in \bZ}
C_{i,n}^{[\kappa]} \ , \\
& 
C_{i,n}^{[\kappa]}:=
\sum_{{i_1 + i_2 = i \atop n_1 + n_2 = n} \atop
\kappa_1 + \kappa_2 = \kappa} A_{i_1, n_1}^{[\kappa_1]}\circ B_{i_2, n_2}^{[\kappa_2]}  
\stackrel{\eqref{bandAB}}{=}
\sum_{{i_1 + i_2 = i \atop n_1 + n_2 = n} } \left(A_{i_1, n_1}\circ B_{i_2, n_2}\right)^{[\kappa]} 
\end{aligned}
$$
where, since $A_{i_1, n_1} \in \mathfrak{F}_{n_1}$ and  $B_{i_2, n_2} \in \mathfrak{F}_{n_2}$, 
the harmonic $\kappa_1$ have modulus not larger than $n_1$ and with the same parity as $n_1$, and $\kappa_2$ have modulus not larger than $n_2$ and 
with the same parity as $n_2$. 
Hence each $\kappa = \kappa_1 + \kappa_2$ has modulus not larger than $ n = n_1 + n_2$ and with the same parity of $ n_1 + n_2$.
Hence $A\circ B \in \mathtt{F}_{\ell + \ell'}$.
Moreover since $A$ has harmonics at most $\ell$ and $B$ at most $\ell'$, to reach harmonics  $\pm (\ell + \ell') $ the only possibility is via the terms $A_{i_1, n_1}^{[\pm\ell]} \circ B_{i_2, n_2}^{[\pm\ell']} $ which, again by the properties of the class $\mathtt{F}_\ell$,  are non-trivial only  
when $(i_1, n_1) = (0,\ell)$, $(i_2, n_2) = (0,\ell')$. 
$(iii)$ follows by  \eqref{starij}.
$(iv)$ follows since
$ ( A_{i,n} v_1 e^{\im j_1 x} , 
v_2 e^{\im j_2 x} ) = ( A_{i,n}^{[j_2 - j_1]} v_1 e^{\im j_1 x} , 
v_2 e^{\im j_2 x} )   $ 
which vanishes,
by the properties of the $\mathtt{F}_\ell$ matrices, if $|j_1 - j_2| > n $ or $ j_1 - j_2 \not\equiv n \mod{2}$. 
\end{proof}


\subsection{Taylor expansion of $\cB (\mu,\e) $ and $P(\mu,\e) $}

We first describe the Taylor expansion 
of the operator ${\cal B}(\umu+\delta,\e)$ defined in \eqref{WW}.
Such operator is the composition of multiplication operators for the analytic functions $a_\e(x)$, $p_\e(x)$ and  Fourier multipliers depending on $\mu, \e$.
We first prove that the jets $\cB_{\ell}$ of $\cB(\mu, \e)$ at $(\umu, 0)$ belong to the class $ \mathtt{F}_\ell $ for any $ \ell \in \bN_0$.
This is due to the fact that, 
by \eqref{leadingLin} and Definition \ref{def:even}, 
the Taylor coefficients $p_\ell (x)$, $a_\ell (x)$ of the functions 
$p_\e(x)$, $a_\e(x)$ have the form
 \begin{equation}\label{leadingpa}
 \begin{aligned}
 p_\ell (x) &=  \begin{cases} p_\ell^{[0]}+p_\ell^{[2]}\cos(2x)+\dots+p_\ell^{[\ell]}\cos(\ell x) & \text{if } \ell \text{ is even}\, , \\[2mm]
p_\ell^{[1]}\cos(x)+p_\ell^{[3]}\cos(3x)+\dots+p_\ell^{[\ell]}\cos(\ell x)  & \text{if }
\ell \text{ is odd} \, ,  \end{cases} \\ 
  a_\ell (x) &=  \begin{cases} a_\ell^{[0]}+a_\ell^{[2]}\cos(2x)+\dots+a_\ell^{[\ell]}\cos(\ell x) & \text{if } \ell \text{ is even} \, , \\[2mm]
a_\ell^{[1]}\cos(x)+a_\ell^{[3]}\cos(3x)+\dots+a_\ell^{[\ell]}\cos(\ell x)  & \text{if }
\ell \text{ is odd} \, ,  \end{cases}
 \end{aligned}
 \end{equation}
  with real  coefficients $a_\ell^{[i]}$, $p_\ell^{[i]}$, whereas the dependence in $\mu$ is only in Fourier multipliers.

 \begin{lem}[{\bf Expansion of $\cB(\mu,\e $}]
The jets of the Taylor expansion  at $(\mu,\epsilon) = (\umu,0) $ of the self-adjoint operator $ \cB (\mu,\e)  $ in \eqref{WW} at  $ (\umu,0) $, namely
\begin{equation}\label{def.Bj}
\cB (\umu+\delta,\e) = \sum_{\ell \geq 0}  \cB_\ell \, , 
\end{equation}
satisfy, 
for any $ \ell  \in \bN $, 
\begin{equation}\label{cBellell}
\cB_0\,,\; 
 \cB_{\ell} \in \mathtt{F}_\ell
\
 \text{ and } \quad 
\cB_{\ell}^{[\pm\ell]} = \cB_{0,\ell}^{[\pm\ell]} = \e^\ell \frac{e^{\pm \im \ell x}}2 \begin{bmatrix} a_\ell^{[\ell]} & -p_\ell^{[\ell]} (\pa_x +\im \umu) \\
 p_\ell^{[\ell]} (\pa_x +\im \umu \pm \im \ell )   & 0 
 \end{bmatrix} \, .
 \end{equation}
\end{lem}

\begin{proof}
Recall that the  operator $\cB_{i,n}^{[\kappa]}$
 is the  $\kappa$-band of the Taylor coefficient of order $\delta^i \e^n $ of the operator $\cB ( \umu + \delta, \e) $ in 
 \eqref{WW} at $\delta = \e = 0$, see \eqref{band.def} and the definition 
 \eqref{def.matrix.elements}
 of matrix elements.
The terms of order $\e^n$ come either from the terms of order $\e^n$ of $a_\e$ and $p_\e$ which, by \eqref{leadingpa} have only harmonics with the same parity of $n$ and not larger than $n$, or from the $\e^n$ Taylor coefficient of the Fourier multiplier $|D+\mu| \tanh\big((\tth + \ttf_\e) |D+\mu| \big)$,  which, in view of 
\eqref{leadingLin}, expand only in  even powers of $\e$. Then $\cB_{i,n}^{[\kappa]}$ satisfies  \eqref{TFj}, proving that 
  $\cB_{i,n} \in \mathfrak{F}_n$  for any $i,n \in \bN_0$, hence 
$\cB_\ell \in \mathtt{F}_\ell$. 
The  explicit  formula \eqref{cBellell} for  $\cB_{0,\ell}^{[\ell]} $ follows by \eqref{WW} and 
\eqref{karmonica}. 
\end{proof}
We now prove the important fact that the jets $P_\ell$ of the spectral 
 projector $P (\mu,\e)  $,  defined in \eqref{Pproj}, belong to the class $\mathtt{F}_\ell$. This is the property that allows us to propagate the structure $\mathtt{F}_\ell$ along the Kato's reduction scheme.
 \begin{lem}[{\bf 
 Expansion of $P (\mu,\e)$}]
 The Taylor expansion at $(\mu,\epsilon) = (\umu,0) $ of 
the projector $P (\mu,\e) $ in \eqref{Pproj}, is  
\begin{equation}\label{def.Pj}
P(\umu+\delta,\e) = \sum_{\ell \geq 0} P_\ell  \, ,\end{equation} 
where  
\begin{equation}\label{Psani1}
P_0 = P(\umu,0) \, , \qquad P_1 = \mathcal{P} \big[\textup{${\cal B}_1$}\big]\, , \qquad P_\ell = \sum_{q=1}^\ell \sum_{ \substack{ \ell_1+\dots+\ell_q=\ell \\
\ell_1, \ldots, \ell_q \in \bN}} \mathcal{P} [\cB_{\ell_q},\dots,\cB_{\ell_1}]\, , 
 \end{equation}
and,  for any    $ q\in \bN $,  
\begin{equation}\label{hP}
\mathcal{P} \big[A_1,\dots, A_q \big] := \frac{(-1)^{q+1} }{2\pi \im } \oint_\Gamma (\cL_{\umu,0}-\lambda)^{-1} \cJ A_1 (\cL_{\umu,0}-\lambda)^{-1} \dots \cJ A_q (\cL_{\umu,0}-\lambda)^{-1} \de\lambda \, .
\end{equation}
It results 
\begin{equation}\label{Psani}
 P_\ell \in \mathtt{F}_\ell \quad  \forall   \ell \in \bN_0 \,  .
 \end{equation}
Furthermore 
\begin{equation}\label{hPbis}
\mathcal{P} \big[A_1,\dots, A_q \big]_{i,n}^{[\kappa]} = \sum_{\substack{i_1+\dots+i_q =i\\ n_1+\dots+n_q =n \\ \kappa_1+\dots+\kappa_q =\kappa \\ \kappa_1,\dots, \kappa_q \in \bZ}}  \mathcal{P} \big[[A_1]_{i_1,n_1}^{[\kappa_1]},\dots,[A_q]_{i_q,n_q}^{[\kappa_q]} \big] \, .
\end{equation}
\end{lem}

\begin{proof}
For any $\lambda$ in the resolvent set of $\cL_{\umu,0}$, we Taylor expand the resolvent $(\cL_{\umu+\delta,\e} - \lambda)^{-1}$ at $\delta = \e = 0$ via Neumann series, getting by \eqref{def.Bj}
\begin{align*}
&(\cL_{\umu+\delta,\e} - \lambda)^{-1}  = 
\left( \uno + (\cL_{\umu,0} - \lambda)^{-1}  \sum_{\ell \geq 1}\cJ \cB_\ell  \right)^{-1} (\cL_{\umu,0} - \lambda)^{-1} \\
& =(\cL_{\umu,0}-\lambda)^{-1}+ \sum_{\ell \geq 1} \sum_{q = 1}^\ell (-1)^q \sum_{\ell_1 + \ldots + \ell_q = \ell \atop \ell_1, \ldots, \ell_q \geq 1}   (\cL_{\umu,0}-\lambda)^{-1} \cJ \cB_{\ell_1} (\cL_{\umu,0}-\lambda)^{-1} \dots \cJ \cB_{\ell_q} (\cL_{\umu,0}-\lambda)^{-1}  \, .
\end{align*}
Inserting such expansion in  \eqref{Pproj} yields \eqref{def.Pj}--\eqref{hP}.
Property \eqref{Psani} descends from the composition rule of  
Lemma  \ref{TFjprop} 
together with the fact that the operator $\cJ \cB_{\ell'} $ belongs to $ \mathtt{F}_{\ell'} $ 
(cfr. \eqref{cBellell}) and $\cL_{\umu,0} $ is  a Fourier multiplier in 
$ \mathtt{F}_{0} $.
\end{proof}

\subsection{Taylor expansion of $\mathfrak{B} (\mu,\e)$}

The next proposition
provides an efficient formula to 
compute 
the Taylor  expansion of the operator $\mathfrak{B} (\mu,\e) $ in \eqref{Bgotico} at any order in terms of the jets of $\cB(\mu,\e)$ and $P(\mu,\e)$. 

\begin{prop}[{\bf Expansion of $\mathfrak{B} ( \mu,\e) $}\null] \label{lem:expBgot} 
The Taylor expansion 
at $ (\mu, \epsilon ) =  (\umu,0) $ of the operator $\mathfrak{B} (\mu,\e) $ in \eqref{Bgotico}  is \begin{equation}\label{B.Tay.ex}
\mathfrak{B}(\umu+\delta,\e)  = \sum_{\ell \geq 0}\mathfrak{B}_\ell 
\end{equation}
where, for any $ \ell \in \bN_0 $,  
\begin{equation}\label{kBj} 
\mathfrak{B}_\ell 
= P_0^* \big( \cB_\ell + \cB_{\ell-1} P_1 + \dots+ \cB_1 P_{\ell-1}\big) P_0 + P_0^* {\cal N}_\ell P_0\, ,\quad \mathfrak{B}_\ell \in \mathtt{F}_\ell\, ,
\end{equation}
the operators ${\cal B}_1, \ldots, {\cal B}_\ell$ 
and $ P_1, \ldots,  P_{\ell-1} $ are the jets of the operators $\cB ( \mu,\e) $ in \eqref{WW}  and $P(\mu,\e) $ in \eqref{Pproj} at $ (\umu,0) $ respectively, and 
\begin{equation}\label{harmonic-gap}
{\cal N}_\ell \in \mathtt{F}_\ell \qquad \textup{satisfy}\qquad {\cal N}_0 = {\cal N}_1 = {\cal N}_2 = 0 \quad \textup{ and } \quad  {\cal N}_\ell^{[\pm \ell]} = 0\,, \quad \forall \ell =3,\dots, \tp \, .
\end{equation}
In particular 
\begin{equation}\label{Bijs}
\mathfrak{B}_{i,n} \in \mathfrak{F}_{n} \, 
, \quad \forall 
i, n \in \bN_0  \, . 
\end{equation}
\end{prop}

The rest of this subsection is devoted to the  proof of  Proposition \ref{lem:expBgot}.\\
We will use that the operators 
 $\cB(\mu,\e)$ and $P(\mu,\e) $ satisfy 
 \begin{align}\label{BP}
& (\cB(\mu,\e) P(\mu,\e))^* = 
P(\mu,\e)^* \cB (\mu,\e)  = \cB (\mu,\e)  P(\mu,\e)\, , \\
& Q (P(\mu,\e)-P_0)^2 = (P(\mu,\e)-P_0)^2 Q \quad \mbox{ with } \quad Q \in \{P(\mu,\e), P_0 \} \, ,
\label{BP1}
\end{align}
where $P_0=P(\umu,0) $.
Property
\eqref{BP} follows by \eqref{commuPL},
\eqref{propPU}, 
$ \cJ^2 = - \text{Id} $
and the fact that 
$\cB(\mu,\e)  $ is self-adjoint. 
Property \eqref{BP1}
is satisfied by any pair of projectors. 

As a first step to prove Proposition \ref{lem:expBgot} we prove the following lemma.
For any operator  $A$ 
we set $\textup{\textbf{Sym}}[A]:= \frac12 A+ \frac12 A^*$.

\begin{lem}[{\bf Splitting of $\mathfrak{B}(\mu,\e) $}]\label{expansionthm2}
The operator $\mathfrak{B}(\mu,\e) $ 
in \eqref{Bgotico} can be written as
\begin{equation}\label{kBfinalform}
\begin{aligned}
& \textup{$\mathfrak{B}$}(\mu,\e)  = P_0^* \textup{$\mathbf{Sym}$}\big[ \cB(\mu,\e)  + \mathring \cB(\mu,\e) \mathring P(\mu,\e) + {\cal R}(\mu,\e) \big] P_0  \\
& \text{where} \quad  
\mathring\cB (\mu,\e)  := \cB (\mu,\e) -\cB_0  \, , 
\  \mathring{P}(\mu,\e) := P(\mu,\e)  -P_0 \, , 
\end{aligned}
\end{equation}
 and 
\begin{equation}\label{Rgappa}
{\cal R}_\ell \in \mathtt{F}_\ell\qquad\textup{satisfy}\qquad {\cal R}_0 = {\cal R}_1 = {\cal R}_2= 0\, ,\quad 
{\cal R}_\ell^{[\pm \ell]} = 0\,, \quad \forall \ell =3,\dots , \tp \, .
\end{equation}
\end{lem}

\begin{proof}
By  \eqref{OperatorU} 
and denoting
$g(x):= (1-x)^{-\frac12}-1  $, 
we first write 
$$ 
U_{\mu,\e} P_0 = \big(\uno+g(\mathring P(\mu,\e)^2 )\big)P(\mu,\e) P_0 \, . 
$$ 
Then, by \eqref{Bgotico} and using also   \eqref{BP1} with $Q = P(\mu,\e) $, we get 
\begin{equation}\label{kB1}
\begin{aligned}
{\textup{$\mathfrak{B}$}}(\mu,\e) 
& = P_0^* U_{\mu,\e}^* \cB(\mu,\e)  U_{\mu,\e} P_0 \\
& =  P_0^*  \big(\uno+g^*(\mathring P(\mu,\e)^2)\big) P(\mu,\e)^* \cB(\mu,\e) P(\mu,\e) \big(\uno+g(\mathring P(\mu,\e)^2 )\big) P_0\, .
\end{aligned}
\end{equation}
Next using 
\eqref{BP} we obtain
\begin{align}
{\textup{$\mathfrak{B}$}}(\mu,\e)  &=  P_0^*  \big(\uno+g^*(\mathring P(\mu,\e)^2)\big)  \cB(\mu,\e) P(\mu,\e) \big(\uno+g(\mathring P(\mu,\e)^2 )\big) P_0   \label{kB2}  \\
&= P_0^* \textup{$\mathbf{Sym}$}\big[
\cB(\mu,\e) P(\mu,\e) + 2\cB (\mu,\e) P(\mu,\e) g(\mathring P(\mu,\e)^2) +  g^*(\mathring P(\mu,\e)^2) \cB(\mu,\e) P(\mu,\e) g(\mathring P(\mu,\e)^2) \big] P_0\, . \notag 
\end{align}
Finally, since $P(\mu,\e) P_0 = (\uno + \mathring P(\mu,\e)) P_0 $ we write 
$$
\cB(\mu,\e) P(\mu,\e) P_0 =\cB(\mu,\e) P_{0}+ \mathring \cB(\mu,\e) \mathring P(\mu,\e) P_0 + \cB_0 \mathring P(\mu,\e)  P_0
$$
and, 
using 
 that $ g (\mathring P(\mu,\e)^2 )$
 commutes with $P_0 $ by  
 \eqref{BP1} with $Q = P_0$, 
we conclude that  
\begin{equation} \label{kB3}
\textup{$\mathfrak{B}$}(\mu,\e) = P_0^* \textup{$\mathbf{Sym}$}\big[
\cB (\mu,\e) + \mathring \cB (\mu,\e)  \mathring P (\mu,\e)   + \cR (\mu,\e)  \big] P_0  
\end{equation}
where
\begin{equation}\label{siacR}
\begin{aligned}
\cR(\mu,\e) & 
:=  P_0^* \big[ 
{\cB_0 \mathring P(\mu,\e) + 2\cB_0  g(\mathring P(\mu,\e)^2)}+ 2\mathring\cB(\mu,\e)  g(\mathring P(\mu,\e)^2) \\
& \ \ + 2\cB ( \mu,\e) \mathring P(\mu,\e) g(\mathring P(\mu,\e)^2)+  g^*(\mathring P(\mu,\e)^2) \cB(\mu,\e) P(\mu,\e) g(\mathring P(\mu,\e)^2) \big]P_0 \, .
\end{aligned}
\end{equation}
We now show that $\cR(\mu,\e) $ satisfies the properties claimed in 
\eqref{Rgappa}.
\\[1mm]
\underline{ ${\cal R}_\ell \in \mathtt{F}_\ell$.} 
It is a consequence of  Lemma  \ref{TFjprop}, \eqref{cBellell}
and the fact that
$\big[g(\mathring P(\mu,\e)^2)\big]_\ell \in \mathtt{F}_\ell$, which,
writing 
$$ 
\big[g(\mathring P(\mu,\e)^2)\big]_\ell = \sum_{k\geq 1} c_k 
\sum_{\ell_1 + \ldots + \ell_{2k} = \ell \atop \ell_1, \ldots, \ell_{2k} \geq 1} P_{\ell_1} \circ \cdots \circ P_{\ell_{2k}} 
$$ 
(the $c_k$'s are the Taylor coefficients of $ g(x) $ at $ x = 0 $), 
follows  by 
\eqref{Psani} and the composition property of
Lemma  \ref{TFjprop}. 
\\[1mm]
\noindent\underline{$\cR_0 = \cR_1 = \cR_2 =  0$.} 
We first exhibit a cancellation in the  first two terms in \eqref{siacR}. Precisely we claim that 
\begin{equation}\label{claimkB} 
 P_0^* \cB_0 \mathring P(\mu,\e) P_0 + 2 P_0^* \cB_0  g(\mathring P(\mu,\e)^2)P_0 = 2 P_0^* {\cal B}_0 {\cal Q}_{\mu,\e} \mathring P(\mu,\e) ^4 P_0  \, ,\quad {\cal Q}_{\mu,\e}  :=  \widetilde g(\mathring P(\mu,\e)^2) \, ,
\end{equation}
where $\widetilde g(x)$ is the analytic function such that $g(x) = \frac12 x + \widetilde{g}(x)x^2$.
Since the operator 
${\cal Q}_{\mu,\e} $ in \eqref{claimkB} is $\cO_4$
and all the other terms in \eqref{siacR} are at least $\cO_3$, we deduce that
$\cR_0 = \cR_1 = \cR_2 =  0$.
Moreover as above one has ${\cal Q}_\ell \in \mathtt{F}_\ell$ (argue as for ${\cal R}_\ell$).
Let us prove \eqref{claimkB}.
By definition of $\widetilde g(x)$ we have
$
g(\mathring P(\mu,\e)^2) =  \tfrac12 \mathring P(\mu,\e)^2 + {\cal Q}_{\mu,\e} \mathring P(\mu,\e)^4 \, ,
$
and, using also \eqref{BP},
\begin{equation}\label{claimkBaux}
\begin{aligned}
&  P_0^* \cB_0 \mathring P(\mu,\e) P_0 + 2 P_0^* \cB_0  g(\mathring P(\mu,\e)^2)P_0 = \\
& \underbrace{P_0^* \cB_0 P_0 \mathring P(\mu,\e) P_0 +P_0^* \cB_0     \mathring P(\mu,\e)^2  P_0 }_{=0}+ 2 P_0^* \cB_0 {\cal Q}_{\mu,\e} \mathring P(\mu,\e)^4 P_0 
\end{aligned}
\end{equation}
where the underbraced term vanishes because of the identity $P_0 \mathring P(\mu,\e) P_0 = -  \mathring P(\mu,\e)^2 P_0 $, which comes from the projector identity $P_0+\mathring P(\mu,\e) = (P_0+\mathring P(\mu,\e))^2 $.
This proves \eqref{claimkB}. 
\\[1mm]
\noindent\underline{$
{\cal R}_\ell^{[\pm \ell]} = 0$ for any $ \ell =3,\dots , \tp $.} 
We first note that, by \eqref{siacR}-\eqref{claimkB} and \eqref{BP1} with $Q = P_0$,  the operator  $\cR(\mu,\e) $ has the form
\begin{equation}\label{thisR}
\textup{${\cal R}(\mu,\e) $} = P_0^* \textup{$\mathcal{E}$}(\mu,\e) P_0 \mathring P(\mu,\e)^2 P_0 
\end{equation}
with
\begin{equation}\label{emuep}
\begin{aligned}
\mathcal{E}(\mu,\e) 
& := 2 \cB_0 \mathcal{Q}_{\mu,\e} \mathring{P}(\mu,\e)^2 + 2\mathring{\cB}(\mu,\e)\, \breve{g}(\mathring{P}(\mu,\e)^2) + 2\cB(\mu,\e) \mathring{P}(\mu,\e) \,  \breve{g}(\mathring{P}(\mu,\e)^2) \\
& \quad + g^*(\mathring{P}(\mu,\e)^2) \cB(\mu,\e) P(\mu,\e) \, \breve{g}(\mathring{P}(\mu,\e)^2) \, ,\quad \mathcal{E}_0 = 0 \, ,
\end{aligned}
\end{equation}
where $\breve{g}(x) $ is the analytic function such that $g(x) = \breve{g}(x)x $.  By Lemma \ref{TFjprop},  \eqref{def.Bj}, \eqref{def.Pj} we deduce that 
\begin{equation}\label{ecale}
\mathcal{E}_\ell \in \mathtt{F}_\ell \, , 
\qquad \ell \in \bN \, . 
\end{equation}
We now show that ${\cal R}_\ell^{[\pm \ell]} =0$,
for any 
$\ell=3,\dots, \tp$, or, equivalently,
$$
\big( {\cal R}_\ell^{[\pm \ell]} \phi_1,\phi_2\big) 
\stackrel{\eqref{thisR}} = \big(P_0^* {\cal R}_\ell^{[\pm \ell]}P_0 \phi_1,\phi_2\big) = \big( {\cal R}_\ell^{[\pm \ell]}P_0 \phi_1, P_0 \phi_2\big)  = 0 \, ,\quad \forall \phi_1,\phi_2 \in L^2(\bT)\, ,
$$
which boils down to prove that, for any  $\ell=3,\dots, \tp$,
\begin{equation}\label{claimRvanish}
({\cal R}_\ell^{[\pm\ell]} f_0^-, f_0^- )=0\, ,\quad ({\cal R}_\ell^{[\pm\ell]} f_{\tp}^+, f_{\tp}^+ )=0\,, \quad ({\cal R}_\ell^{[\pm\ell]} f_{\tp}^+, f_{0}^- )=0\,,  \quad ({\cal R}_\ell^{[\pm\ell]} f_{0}^-, f_{\tp}^+ )=0  \, .
\end{equation}  
Since ${\cal R}_\ell \in \mathtt{F}_\ell$, by directly 
 comparing the 
 harmonics in the left and right vectors
 of the scalar product, 
 all these identities are trivial 
except for $({\cal R}_\tp^{[\tp]} f_{0}^-, f_{\tp}^+ )=0 $ and $({\cal R}_\tp^{[-\tp]} f_{\tp}^+, f_{0}^- )=0 $. 
In view of \eqref{thisR}
and Lemma \ref{TFjprop}  
we have 
\begin{equation}
\label{Rellell}
\big( {\cal R}_\tp^{[\tp]}f_0^-,f_\tp^+\big)  = \sum_{\ell_1+\ell_2+\ell_3=\tp \atop \ell_1, \ell_2, \ell_3 \geq 1} \big( \mathcal{E}_{\ell_1}^{[\ell_1]} 
\underbrace{P_0 P_{\ell_2}^{[\ell_2]} P_{\ell_3}^{[\ell_3]} f_0^-}_{=: \phi },f_\tp^+  \big)\, .   
\end{equation}
Since $\phi = \alpha f_0^- + \beta f_\tp^+ \in {\cal V}_{\umu,0} $ and, by  \eqref{ecale}, 
$\mathcal{E}_{\ell_1}^{[\ell_1]} \in \mathtt{F}_{\ell_1} $ 
with $
1 \leq \ell_1 \leq\tp-1 $, the scalar product in \eqref{Rellell} vanishes.
Similarly $({\cal R}_\tp^{[-\tp]} f_{\tp}^+, f_{0}^- )=0 $. Thus \eqref{claimRvanish} and consequently  the last identity in \eqref{Rgappa} hold. 
\end{proof}

In the next lemma we 
simplify the first terms in \eqref{kBfinalform}. 
\begin{lem} One has
\begin{equation}\label{tuttosym}
P_0^* \textup{$\mathbf{Sym}$}\big[ \cB(\mu,\e) + \mathring \cB(\mu,\e) \mathring P(\mu,\e) \big] P_0 = P_0^*\big( \cB(\mu,\e) + \mathring \cB(\mu,\e) 
\mathring P(\mu,\e) \big) P_0 \, .
\end{equation}
\end{lem}

\begin{proof}
We begin by proving the  identities
\begin{equation}\label{newidentities}
P_0^* \textup{$\cB_0$} = -\im \omega_*^{(\tp)} P_0^* \textup{$\cJ$}\,, \qquad P_0^*\textup{$\cB_0$} P(\mu,\e) P_0 = P_0^* P(\mu,\e)^* \textup{$\cB_0$} P_0\, .
\end{equation}
The first identity descends from the fact that,
for any $ f,g $ in $ L^2(\bT,\bC^2)$,  
$$
\big(P_0^*\textup{$\cB_0$} f, g\big) =\big( f, \textup{$\cB_0$} P_0 g\big) = \big(\cJ  f,  \cL_{\umu,0} P_0 g\big) = - \im \omega_*^{(\tp)}   \big( \cJ  f,  P_0 g\big) = -\im \omega_*^{(\tp)}  \big( P_0^* \cJ  f, g\big)\,  . 
$$
The second identity descends from the first one, since 
$$ 
\begin{aligned}
P_0^* \cB_0 P(\mu,\e) P_0 
= -\im \omega_*^{(\tp)}  P_0^* \cJ P(\mu,\e) P_0
& \stackrel{\eqref{propPU}}{=} P_0^* P(\mu,\e)^*( -\im \omega_*^{(\tp)}  P_0^* \cJ   ) \\
& = P_0^* P(\mu,\e)^* P_0^* \cB_0 \stackrel{\eqref{BP}}{=}  P_0^* P(\mu,\e)^*\cB_0 P_0\, . 
\end{aligned}
$$
We use \eqref{newidentities} to deduce that $P_0^*\big( \cB(\mu,\e) + \mathring \cB(\mu,\e) \mathring 
P(\mu,\e) \big) P_0 $ 
is selfadjoint. Indeed
$$
\begin{aligned}
P_0^*\big( \cB(\mu,\e) + \mathring \cB(\mu,\e) \mathring P(\mu,\e) \big) P_0 
& = P_0^* \big(\cB(\mu,\e) + \cB(\mu,\e)  P(\mu,\e) +  
\cB_0 P_0 - \cB_0 P(\mu,\e) - \cB (\mu,\e)  P_0 \big) P_0 \\
& = P_0^* \big( \overbrace{\cB (\mu,\e)  P(\mu,\e)}^{(\mathrm{I})}+  \overbrace{\cB_0 P_0}^{(\mathrm{II})} - \overbrace{\cB_0 P(\mu,\e)}^{(\mathrm{III})} \big) P_0\, . 
\end{aligned}
$$
The terms (\textrm{I}) and (\textrm{II}) are self-adjoint by \eqref{BP}.
Also the operator $P_0^* (\mathrm{III})P_0$ is self-adjoint by the second identity in \eqref{newidentities}. This proves \eqref{tuttosym}.
\end{proof}

\smallskip
\noindent{\it Proof of  Proposition \ref{lem:expBgot}}.  By \eqref{kBfinalform} and \eqref{tuttosym} we have 
$$
\mathfrak{B}(\mu,\e) =  P_0^* \big(\cB(\mu,\e) + 
(\cB(\mu,\e) - \cB(\umu,0))(P(\mu,\e) - P(\umu,0)) +  {\cal N}(\mu,\e) \big) P_0 
$$
where 
$ {\cal N}(\mu,\e) := \textup{$\mathbf{Sym}$}[{\cal R}(\mu,\e)] $. 
Consequently the jets 
$ \mathfrak{B}_\ell $ of $ \mathfrak{B}(\mu,\e) $ have the form in \eqref{kBj} and, by  Lemma \ref{TFjprop},  the fact that  $\cB_{\ell'},P_{\ell'} \in \mathtt{F}_{\ell'} $
for any $ 1
\leq \ell' \leq \ell $
(cfr. \eqref{cBellell}, \eqref{Psani}), 
and \eqref{Rgappa} we conclude that  $\mathfrak{B}_\ell\in  \mathtt{F}_{\ell}$. 
By \eqref{Rgappa} the operator ${\cal N}(\mu,\e) $ fulfills \eqref{harmonic-gap}, using also Lemma \ref{TFjprop} for the last identity. \qed

\subsection{
Proof of the weak conjecture
 \eqref{weakconj}} 
 \label{sec:up}

We now start proving  that the entries of the matrix $\tL^{(\tp)}(\mu,\e)$ given by \eqref{tocomputematrix} 
satisfy Assumption \ref{assH} 
as stated in Theorem \ref{lem:expansionL}.

\begin{lem}\label{lem:parityofentries}
The entries of the matrix  $\tB^{(\tp)}(\mu,\e)$ in \eqref{tocomputematrixB} satisfy the parity properties 
\eqref{paritaentri}. 
\end{lem}

\begin{proof}
On $ B_{\delta_0} ( \umu)
\times B_{\epsilon_0} (0) $
we have the Taylor expansions
\begin{equation}\label{apinep}
\alpha^{(\tp)}(\mu,\e) =
\big(\mathfrak{B}(\mu,\e) f_\tp^+, f_\tp^+ \big) 
= \sum_{n \geq 0}
a^{(\tp)}_n(\mu) \e^n
= \sum_{i,n \geq 0}
a_{i,n}^{(\tp)} \delta^i \e^n 
\end{equation}
where, by \eqref{Bijs} and 
recalling Definition \ref{defFell},  
$$
a_{i,n}^{(\tp)} = 
\big(\mathfrak{B}_{i,n} f_\tp^+, f_\tp^+ \big)  
= 0 
\qquad\text{if } n \text{ is odd} \, ,  \quad \forall i \in \bN_0 \, . 
$$
This proves that $ a^{(\tp)}_n (\mu) \equiv 0 $ for any $ n $ odd, namely 
$ \e \mapsto \alpha^{(\tp)}(\mu,\e)$ is even. Similarly
we deduce that 
$ \epsilon \mapsto \gamma^{(\tp)}(\mu,\e) =
\big(\mathfrak{B}(\mu,\e) f_0^+, f_0^+ \big) $
is even. 
Furthermore
\begin{equation}\label{betaesp}
\beta^{(\tp)}(\mu,\e) =
- \im
\big(\mathfrak{B}(\mu,\e) f_0^-, f_\tp^+ \big) 
= \sum_{n \geq 0}
b^{(\tp)}_n (\mu) \e^n 
= \sum_{i,n \geq 0}
b_{i,n}^{(\tp)} \delta^i \e^n 
\end{equation}
where, by \eqref{Bijs} and 
recalling Definition \ref{defFell} and \eqref{AFlele}, 
\begin{equation}\label{offdiagonalparity}
b_{i,n}^{(\tp)} = 
-\im \big(\mathfrak{B}_{i,n} f_0^-, f_\tp^+ \big)  
\stackrel{\eqref{Bijs}}= 0 
\qquad\text{if } n \not\equiv \tp \mod 2 \, ,
\quad \forall i \in \bN_0 \, . 
\end{equation}  
This proves that the function 
$ \e \mapsto \beta^{(\tp)}(\mu,\e)$ has the   parity of  $\tp$. The lemma is proved. 
\end{proof}

We now prove 
that $\beta^{(\tp)}(\mu,\e)$  vanish up to order $\tp-1$ in $ \epsilon $ 
as well as the order
$ \epsilon^{\tp+1}$.

\begin{lem} 
\label{lem:tBp0} 
Expansion \eqref{ciochevoglio} holds with $\beta_1^{(\tp)}(\tth)\e^\tp :=  -\im \big( \mathfrak{B}_{0,\tp} f_0^-,f_\tp^+ \big)  $.
\end{lem}

\begin{proof}
On $ B_{\delta_0} ( \umu)
\times B_{\epsilon_0} (0) $
we have the Taylor expansion
\eqref{betaesp}
\begin{equation}\label{espbetap}
\beta^{(\tp)}(\mu,\e) 
= \sum_{n =0}^{\tp-1}
b^{(\tp)}_n (\mu) \e^n 
+  
b^{(\tp)}_\tp (\mu) \e^\tp +
b^{(\tp)}_{\tp+1} (\mu) \e^{\tp+1} + r(\epsilon^{\tp+2}) \, ,
\end{equation}
where, for any $ n  $, 
\begin{equation}
\label{svilbetaij}
b_n^{(\tp)} (\mu) 
= \sum_{i \geq 0}
b_{i,n}^{(\tp)} \delta^i  \, , \qquad
b_{i,n}^{(\tp)} = 
- \im \big(\mathfrak{B}_{i,n} f_0^-, f_\tp^+ \big) \, . 
\end{equation}
Now, since $ f_0^- $, $ f_{\tp}^+ $ in \eqref{autovettorikernel0}   are 
Fourier supported respectively on the $0 $ and the $ \tp $ harmonic, 
using \eqref{Bijs},   
\eqref{AFlele} and \eqref{offdiagonalparity}
we have that 
the functions 
\begin{equation}\label{fanno0sottoep+1}
b^{(\tp)}_n (\mu) \equiv 0 \, , \quad \forall  
n < \tp \, , \qquad
b^{(\tp)}_{\tp+1} (\mu) 
\equiv 0 \, .
\end{equation}
By \eqref{espbetap}
and \eqref{fanno0sottoep+1}
we deduce that 
$$
\beta^{(\tp)}(\umu+\delta,\e) =
-\im \big(\mathfrak{B}(\mu,\e)  f_0^-,f_\tp^+ \big) = 
b^{(\tp)}_\tp(\umu+\delta ) \e^\tp   +
r(\e^{\tp+2}) = 
\beta_1^{(\tp)}(\tth)\e^\tp 
+ r( \delta \e^{\tp}, \e^{\tp+2}) \, ,
$$
where
$ \beta_1^{(\tp)}(\tth)\e^\tp 
:= b_{\tp}^{(\tp)} (\umu) \e^\tp = b_{0,\tp}^{(\tp)} \e^\tp = 
-\im \big( \mathfrak{B}_{0,\tp} f_0^-,f_\tp^+ \big) $ by 
\eqref{svilbetaij}.
 The lemma is proved.  
\end{proof}

We  now turn our attention to the second part of Theorem \ref{lem:expansionL}.

\begin{lem}
The expansions  \eqref{abc}, \eqref{omegappomega0m} and \eqref{a1g1} hold. 
\end{lem}

\begin{proof}
By   \eqref{unperturbed} and the Taylor expansion 
\eqref{apinep}  we have  
$$
\begin{aligned}
\alpha^{(\tp)}(\mu,\e) 
& = 
-\omega_\tp^+(\umu+\delta)
+ a^{(\tp)}_2
(\umu+\delta) \e^2 + r(\e^4) \\
\gamma^{(\tp)}(\mu,\e) 
& = 
\omega_0^-(\umu+\delta)
+ \gamma^{(\tp)}_2
(\umu+\delta) \e^2 + r(\e^4)
\end{aligned}
$$
proving \eqref{abc} after Taylor expanding
$ a^{(\tp)}_2
(\umu+\delta)$, $ \gamma^{(\tp)}_2
(\umu+\delta) $ at $\delta = 0$ and denoting 
 $ \alpha^{(\tp)}_2
(\tth) := a^{(\tp)}_2
(\umu) $.

We deduce 
\eqref{omegappomega0m} 
with 
$$
\begin{aligned}
& \alpha_1^{(\tp)}(\tth) :=  -( \pa_\mu \omega^+ )(\tp+\umu,\tth ) \stackrel{\eqref{omeghino}}{=} -\ch + (\pa_\varphi \Omega)(\tp+\umu,\tth) \, , \\
& \gamma_1^{(\tp)}(\tth) :=  ( \pa_\mu \omega^- )(\umu,\tth ) \stackrel{\eqref{omeghino}}{=} \ch + (\pa_\varphi \Omega)(\umu,\tth)\, .
\end{aligned}
$$
Note that, for any 
$ \tth > 0 $,  by \eqref{boomerang}
$$
\begin{aligned}
& \alpha_1^{(\tp)}(\tth)+\gamma_1^{(\tp)}(\tth)  = (\pa_\varphi \Omega)(\umu+\tp,\tth) +(\pa_\varphi \Omega)(\umu,\tth) >
 0 \\
& \gamma_1^{(\tp)}
(\tth) -\alpha_1^{(\tp)}
(\tth) 
=  
2 \ch + (\pa_\varphi \Omega)(\umu,\tth)-
(\pa_\varphi \Omega)(\umu+\tp,\tth)
>  2 \ch
\end{aligned}
$$ 
because 
 the function $x \mapsto (\pa_\varphi \Omega)(\umu +x,\tth)$ is strictly decreasing, 
 by \eqref{boomerang}. 
  Thus 
 also \eqref{a1g1} is proved.
 \end{proof}

We now have to compute the explicit expression \eqref{beta1exp} 
of the first leading term $ \beta_1^{(\tp)} (\tth) \e^\tp  $.  

\subsection{Formula for  $\beta_1^{(\tp)} (\tth)$}\label{Sec:entcoeff}

Let us consider depths $ \tth > 0 $ 
such that $ \umu = \uphi(\tp,\tth)$ in Lemma \ref{collemma} is {\it not} an integer and 
the eigenvectors $ \{ f_j^{\sigma} \}_{j \in \bZ, \sigma = \pm } $, 
$ f_j^{\sigma} = f_j^{\sigma} (\umu )$ in 
\eqref{def:fsigmaj}  form a 
complex symplectic {\it basis} of $L^2(\bT,\bC^2)$, cfr.
Remark \ref{rmk:depth}. In Lemma \ref{actionofL}  
we represent the action of the jets of the 
 Kato projectors 
 in this eigenfunction  basis, 
 and not just the exponential basis (of course these basis are deeply related). This is very convenient for
 the computation of $ \beta_1^{(\tp)} (\tth) $.
 
We also remark that this is with no loss of generality. 
Indeed, since the function $ \tth \mapsto \uphi(\tp,\tth)$ is non zero, real-analytic and bounded (see Lemma \ref{collemma}), 
the excluded depths are isolated and form a closed set, thus its complementary is dense in 
$ \tth > 0 $. In particular, the formula for $ \beta_1^{(\tp)}(\tth)$ in \eqref{beta1exp}, although proved for $\tth $ in a dense set of $(0,+\infty) $, holds for {\it all} $\tth>0$ by continuity
(the function $ \beta_1^{(\tp)}(\tth)$ is actually analytic in $ \tth > 0 $). 

The next lemma, proved in \cite[Lemma 4.1]{BMV4} (Item ($ii$) is formula   \cite[(4.19)]{BMV4}),  provides effective formulas to compute the action of the operators $ {\cal P}[{\cal B}_{\ell_1},\dots,{\cal B}_{\ell_q}]$  on the 
eigenvector basis.
\begin{lem}\label{actionofL}
Let $\ent{\ell}{\kappa}{j'}{j}{\sigma'}{\sigma}$ denote the entanglement coefficients in \eqref{entcoeff} and $f_j^\sigma:=f_j^\sigma(\umu)$ in \eqref{def:fsigmaj}, with $ \umu\in \bR\setminus \bZ$. Then:\\
{\it (i)} for any $q\in \bN$, $\ell_1,\dots, \ell_q \in \bN $ integers  $ j, j',  \kappa_1,\dots,\kappa_q\in \bZ $
and $ \sigma, \sigma' = \pm $,  one has 
\begin{align}\label{cBactspar}
&\big( \cB_{0,\ell_{q+1}}^{[\kappa_{q+1}]}\mathcal{P}[\cB_{0,\ell_q}^{[\kappa_q]},\dots,\cB_{0,\ell_1}^{[\kappa_1]}]  f_j^\sigma , f_{j'}^{\sigma'}\big)  \\ \notag 
&\qquad = \sum_{\sigma_1,\dots,\sigma_q = \pm } \sigma_1\dots\sigma_q \ent{\ell_1}{\kappa_1}{j_1}{j}{\sigma_1}{\sigma}  \ent{\ell_2}{\kappa_2}{j_2}{j_1}{\sigma_2}{\sigma_1} \dots \ent{\ell_q}{\kappa_q}{j_q}{j_{q-1}}{\sigma_q}{\sigma_{q-1}} \ent{\ell_{q+1}}{\kappa_{q+1}}{j'}{j_q}{\sigma'}{\sigma_q} \Res{j, &j_1, &\dots, &j_q}{\sigma, &\sigma_1, &\dots, &\sigma_q}  
\end{align}
where $j_1 := j+\kappa_1$, $j_2 := j_1+\kappa_2$, \dots, $j_q=j_{q-1}+\kappa_q $ and
\begin{equation}\label{generalresidue}
 \Res{j, &j_1, &\dots, &j_q}{\sigma, &\sigma_1, &\dots, &\sigma_q} := \dfrac{ (-\im\!)^q }{2\pi\im} \oint_\Gamma  \dfrac{ \de \lambda }{(\lambda-\im \omega_j^\sigma)(\lambda-\im \omega_{j_1}^{\sigma_1})\dots (\lambda-\im \omega_{j_q}^{\sigma_q}) } \, ,
 \end{equation}
where $\Gamma$ is a circuit winding once around $\im \omega_*^{(\tp)}$
counterclockwise 
(here $ \omega_*^{(\tp)} := \omega_*^{(\tp)}(\tth) $ is given in \eqref{Kreincollision})
and $\omega_{j}^{\sigma}$ are in \eqref{omegajsigma}.
\\[1mm]
(ii)
Let $j_1,\dots,j_q\in \bZ$ and $\sigma_1,\dots,\sigma_q=\pm$, with $(j_1,\sigma_1),\dots,(j_q,\sigma_q)\notin \{(0,-),(\tp,+)\}$. Then
\begin{equation}\label{residuevalue}
\Res{0, &j_1, &\dots, &j_q}{-, &\sigma_1, &\dots, &\sigma_q}= \Res{p, &j_1, &\dots, &j_q}{+, &\sigma_1, &\dots, &\sigma_q}= 
\dfrac{ 1}{ ( \omega_{j_1}^{\sigma_1}-\omega_*^{(\tp)})\ldots( \omega_{j_q}^{\sigma_q}-\omega_*^{(\tp)} )} 
\end{equation}
with $\omega_j^\sigma$ in \eqref{omegajsigma} and $ \omega_*^{(\tp)} = \omega_*^{(\tp)}(\tth) $ in \eqref{Kreincollision}.
\end{lem}

Recall that in Lemma \ref{lem:tBp0} we have proved that expansion \eqref{ciochevoglio} holds.
We now compute explicitly the coefficient $\im \beta_1^{(\tp)}(\tth) \e^\tp $.

\begin{lem}
\label{lem:tBp}  The coefficient $\beta_1^{(\tp)}
(\tth) $  in \eqref{ciochevoglio} is the function defined 
in \eqref{beta1exp}.
\end{lem}

\begin{proof} By Lemma \ref{lem:tBp0} we have $
( \textup{$\mathfrak{B}$}_{0,\tp} f_0^-, f_{\tp}^+ ) = ( \textup{$\mathfrak{B}$}_{0,\tp}^{[\tp]} f_0^-, f_{\tp}^+ )$. 
By \eqref{kBj}, \eqref{harmonic-gap}, and using that $\cB_\ell, P_\ell \in \mathtt{F}_\ell$, 
\begin{align}\notag
\im \beta_1^{(\tp)} \e^\tp &\,\;\,=\;\,\, ( \cB_{0,\tp}^{[\tp]} f_0^-, f_{\tp}^+ ) + ( [\cB_{0,\tp-1} P_{0,1}]^{[\tp]} f_0^-, f_{\tp}^+ ) + \dots + ([ \cB_{0,1} P_{0,\tp-1}]^{[\tp]} f_0^-, f_{\tp}^+ ) \\ \label{auxmax}
&\stackrel{\text{Lemma} \,  \ref{TFjprop} (i), (ii)}{=}
( \cB_{0,\tp}^{[\tp]} f_0^-, f_{\tp}^+ )
+ ( \cB_{0,\tp-1}^{[\tp-1]}P_{0,1}^{[1]} f_0^-, f_{\tp}^+ ) + \dots + ( \cB_{0,1}^{[1]}P_{0,\tp-1}^{[\tp-1]} f_0^-, f_{\tp}^+ ) \notag \, .
\end{align}
We now express it in terms of the entanglement coefficients in \eqref{entcoeff}. We have
\begin{align*}
\im \beta_1^{(\tp)}\e^\tp & =  \ent{\tp}{\tp}{\tp}{0}{+}{-} + \sum_{i=1}^{\tp-1}  (\cB_{0,\tp -i}^{[\tp-i]} P_{0,i}^{[i]}f_0^-,f_{\tp}^+) \\
 &\stackrel{\eqref{Psani1}}{=} \ent{\tp}{\tp}{\tp}{0}{+}{-} + \sum_{i=1}^{\tp-1} \sum_{q=1}^{i}\sum_{\ell_1+\dots+\ell_q=i} (\cB_{0,\tp-i}^{[\tp-i]}  \mathcal{P}[\cB_{\ell_q},\dots,\cB_{\ell_1}]_{0,i}^{[i]}f_0^-,f_{\tp}^+)  \\
& =  \ent{\tp}{\tp}{\tp}{0}{+}{-} + \sum_{i=1}^{\tp-1} \sum_{q=1}^{i}\sum_{\ell_1+\dots+\ell_q=i} (\cB_{0,\tp -i}^{[\tp-i]}\mathcal{P}[\cB_{0,\ell_q}^{[\ell_q]},\dots,\cB_{0,\ell_1}^{[\ell_1]}]f_0^-,f_{\tp}^+)  
\end{align*}
where in the last equality  we used formula \eqref{hPbis} in which    $\kappa_a \leq \ell_a$ (since  $\cB_{0,\ell_a} \in \mathfrak{F}_{\ell_a}$), and here the only possibility is that 
$ \kappa_a = \ell_a$ for $a = 1, \ldots, q$.

We now use Lemma \ref{actionofL} to compute explicitly the terms in the sum above.
\begin{align*}
\im \beta_1^{(\tp)}\e^\tp   &\stackrel{\eqref{cBactspar}}{=}\ent{\tp}{\tp}{\tp}{0}{+}{-} + \sum_{i=1}^{\tp-1} \sum_{q=1}^{i}\sum_{\substack{\ell_1+\dots+\ell_q=i\\ \sigma_1,\dots,\sigma_q = \pm}} \!\!\!\!\! \sigma_1\dots\sigma_q   \Res{0,\ell_1,\dots,\ell_q}{-,\sigma_1,\dots,\sigma_q} \, \ent{\ell_1}{\ell_1}{j_1}{0}{\sigma_1}{-}  \ent{\ell_2}{\ell_2}{j_2}{j_1}{\sigma_2}{\sigma_1} \dots \ent{\ell_q}{\ell_q}{j_q}{j_{q-1}}{\sigma_q}{\sigma_{q-1}} \ent{\tp-i}{\tp-i}{\tp}{j_q}{+}{\sigma_q} 
\end{align*}
with 
$$
j_1 := \ell_1 \, , \ 
j_2 := j_1+\ell_2 = \ell_1+\ell_2 \, , \ 
\dots \, , \ j_q:=j_{q-1}+\ell_q= \ell_1 +\dots +\ell_q =i \, .
$$
Thus 
$ 1 \leq j_1 < \ldots < j_q \leq\tp- 1 $ and
$(j_1,\sigma_1),\dots,(j_q,\sigma_q)\notin \{(0,-),(\tp,+)\}$ so that, 
using \eqref{residuevalue},  we have
 \begin{equation}
\im \beta_1^{(\tp)} \e^\tp= \ent{\tp}{\tp}{\tp}{0}{+}{-} + \sum_{i=1}^{\tp-1} \sum_{q=1}^{i}\sum_{\substack{\ell_1+\dots+\ell_q=i\\ \sigma_1,\dots,\sigma_q = \pm}} \!\!\!\!\!  \sigma_1\dots\sigma_q \dfrac{\ent{\ell_1}{\ell_1}{j_1}{0}{\sigma_1}{-}  \ent{\ell_2}{\ell_2}{j_2}{j_1}{\sigma_2}{\sigma_1} \dots \ent{\ell_q}{\ell_q}{j_q}{j_{q-1}}{\sigma_q}{\sigma_{q-1}} \ent{\tp-i}{\tp-i}{\tp}{j_q}{+}{\sigma_q} }{ ( \omega_{j_1}^{\sigma_1}-\omega_*^{(\tp)})\ldots(\omega_{j_q}^{\sigma_q}-\omega_*^{(\tp)}) }\, . \label{beta1step}
\end{equation}
 We sum  differently
 \eqref{beta1step} noting that the following map  between sets of natural indices  is a bijection 
 \begin{equation}
     \begin{aligned}
&\phi: \big\{(i,q;\ell_1,\dots,\ell_q)\, :\, 1\leq i\leq \tp -1\, ,\, 1\leq q \leq i\, ,\, \ell_1+\dots+\ell_q=i \big\} \rightarrow  \\
& \quad \ \big\{ (q;j_1,\dots,j_q)\, :\,
1 \leq q \leq \tp -1 \, , \ 0< j_1<\dots<j_q <\tp \big\}  \\
&\quad    \ (i,q;\ell_1,\dots,\ell_q) \mapsto (q; j_1:=\ell_1, j_2 := j_1+\ell_2,\dots, j_q:=j_{q-1}+\ell_q=i) \label{isomorph}
\end{aligned}
\end{equation}
with inverse 
$$
\phi^{-1}(q,j_1,\dots,j_q) = \big( j_q,q,j_1,j_2-j_1,\dots, j_q-j_{q-1} \big) \, .
$$ 
Thus, by \eqref{isomorph}, identity \eqref{beta1step} becomes 
$$
\im \beta_1^{(\tp)} \e^\tp
= \ent{\tp}{\tp}{\tp}{0}{+}{-} +  \sum_{q=1}^{\tp-1}\ \sum_{0<j_1<\dots<j_q <\tp} \  \sum_{\sigma_1,\dots,\sigma_q = \pm }   \sigma_1\dots\sigma_q \dfrac{\ent{\ell_1}{\ell_1}{j_1}{0}{\sigma_1}{-}  \ent{\ell_2}{\ell_2}{j_2}{j_1}{\sigma_2}{\sigma_1} \dots \ent{\ell_q}{\ell_q}{j_q}{j_{q-1}}{\sigma_q}{\sigma_{q-1}} \ent{\tp-j_q}{\tp-j_q}{\tp}{j_q}{+}{\sigma_q}}{ ( \omega_{j_1}^{\sigma_1}-\omega_*^{(\tp)})\ldots(\omega_{j_q}^{\sigma_q}-\omega_*^{(\tp)})  } 
$$
with $\ell_1 :=  j_1 $, $\ell_2:= j_2-j_1$, $\dots$, $\ell_{q} := j_q-j_{q-1}$. This proves formula \eqref{beta1exp}.\end{proof}

Lemmata
\ref{lem:tBp0} and \ref{lem:tBp} prove formula 
\eqref{beta1exp}.

We finally compute  the entanglement 
coefficients in \eqref{entanglementfdj}, in terms of the Fourier-Taylor coefficients of the Stokes wave. 
We use that 
$\Omega_j^{(\tp)}$ and $t_j^{(\tp)}$ in \eqref{tjp} satisfy 
\begin{equation}\label{idtjOmegaj}
\tanh\big(\tth(j+\uphi(\tp,\tth))\big) = \tfrac{\Omega_j^{(\tp)}}{t_j^{(\tp)}} \, , \quad   j+\uphi(\tp,\tth) = \Omega_j^{(\tp)}t_j^{(\tp)}\, ,\quad j\in\bN_0\,  ,
\end{equation}
since $j+\uphi(\tp,\tth) >0 $ for any $j\geq 0$.

\begin{lem}\label{lem:entexp}
For any $j \in \bZ$, $\sigma,\sigma'=\pm 1$, $\ell=1,\dots,\tp$ the  entanglement 
coefficient $\ent{\ell}{\ell}{j+\ell}{j}{\sigma'}{\sigma} $ in \eqref{entcoeff} has the expression  in \eqref{entanglementfdj}.
\end{lem}
\begin{proof}  
By \eqref{cBellell}, \eqref{def:fsigmaj} and \eqref{tjp} we have
$$
{\cal B}_{0,\ell}^{[\pm \ell]} f_j^\sigma = \frac{e^{\im (j\pm  \ell)x}}{2\sqrt{2\Omega_j^{(\tp)}}}\vet{-a_\ell^{[\ell]}\sqrt{\sigma}\Omega_j^{(\tp)} -\im\sqrt{-\sigma}p_\ell^{[\ell]}(j+\umu)}{-\im\sqrt{\sigma}p_\ell^{[\ell]}
(j\pm \ell +\umu) \Omega_j^{(\tp)} }\e^\ell \, ,
$$
whence, by \eqref{entcoeff}, \eqref{def:fsigmaj} we obtain
$$
\ent{\ell}{\pm \ell}{j\pm \ell}{j}{\sigma'}{\sigma}=\frac{\e^\ell}{4\sqrt{\Omega_j^{(\tp)} \Omega_{j\pm \ell}^{(\tp)}}}\big( \sqrt{\sigma}\bar{\sqrt{\sigma'}} a_\ell^{[\ell]} \Omega_j^{(\tp)} 
\Omega_{j\pm \ell}^{(\tp)}  
+ \im \sqrt{-\sigma}\bar{\sqrt{\sigma'}} p_\ell^{[\ell]} 
(j+\umu) \Omega_{j\pm \ell}^{(\tp)} 
 -\im \sqrt{\sigma}\bar{\sqrt{-\sigma'}} p_\ell^{[\ell]} 
(j\pm \ell +\umu) \Omega_j^{(\tp)}\big)\, .
$$
By exploiting the second identity in \eqref{idtjOmegaj}
 and using that, for $\sigma,\sigma'=\pm$,
$$
 \im\sqrt{-\sigma}\bar{\sqrt{\sigma'}} = -\sigma \sqrt{\sigma}\bar{\sqrt{\sigma'}},\quad -\im\sqrt{\sigma}\bar{\sqrt{-\sigma'}} = -\sigma' \sqrt{\sigma}\bar{\sqrt{\sigma'}}, \quad \sqrt{-\sigma}\bar{\sqrt{-\sigma'}} = \sigma\sigma'\sqrt{\sigma}\bar{\sqrt{\sigma'}}\, ,
$$
we obtain formula \eqref{entanglementfdj}.
\end{proof}

The function $ \beta_1^{(\tp)} (\tth) $ is analytic in $ \tth $
as the coefficients 
$ p_\ell^{[\ell]} (\tth) $, $ a_\ell^{[\ell]} (\tth) $
(cfr. Lemma \ref{pastructure})
and $\Omega_j^{(\tp)} (\tth)$,  $t_j^{(\tp)}(\tth) $
in  \eqref{tjp}.
The proof  of Theorem \ref{lem:expansionL} is complete.

\section{Non-degeneracy of  
$\beta_1^{(\tp)} (\tth)  $ 
}\label{combinatoricsincoming}

In this section we prove  that for {\it any}  
$\tp\geq 2 $ 
the analytic function   $\beta_1^{(\tp)} (\tth) $ in \eqref{beta1exp} 
satisfies \eqref{beta1plimit}.

\subsection{Proof of the strong conjecture \eqref{strongconj} }\label{sec61}

We now provide the 
asymptotic behavior of $\beta_1^{(\tp)}(\tth) $ as $\tth\to 0^+ $. 

\begin{teo}[{\bf Shallow-water limit of $\beta_1^{(\tp)}(\tth)$}]\label{asymptotics} For any $\tp\in \bN$, $\tp\geq 2$,
the function $\beta_1^{(\tp)} (\tth) $ in \eqref{beta1exp} 
has the asymptotic expansion  
\begin{equation}\label{limitbeta1ash0+}
\beta_1^{(\tp)}(\tth) = -\sqrt{\frac{\tp^2-1}{3}} \Big(\frac38\Big)^{\tp-1}\frac{\tp^2(\tp+1)^2}{24} \tth^{\frac{11}2-3\tp}+ O\big(\tth^{\frac{15}2-3 \tp }\big) 
\end{equation}
 as $\tth \to 0^+ $.
\end{teo}

The rest of the section is devoted to the proof of Theorem \ref{asymptotics}. \\
 The following asymptotic expansions 
 will be often used in the sequel.

\begin{lem}
For any $  j \in \bN_0  $
the quantities   $\Omega_j^{(\tp)} $ and $t_j^{(\tp)}$ in \eqref{tjp} 
satisfy, as $ \tth \to 0^+ $, 
\begin{equation}\label{expOmegajtj}
\begin{aligned}
 \Omega_j^{(\tp)} &=\tth^{\frac12} \Big( j  + 
 \frac{\tp^3-
 \tp -2j^3}{12}\tth^2+\frac{38 j^5-30 j^2\tp(\tp^2-1)-4 \tp^5-15 \tp^3+19 \tp}{720} \tth^4+ O( \tth^6) \Big)\, ,\\
 t_j^{(\tp)} &= \tth^{-\frac12}\Big(1+ \frac{\tth^2 j^2}{6}+O(\tth^4)\Big) \,   .
 \end{aligned}
\end{equation}
It results 
\begin{equation}\label{startingsqrt}
 \sqrt{\Omega_0^{(\tp)}\Omega_{\tp}^{(\tp)}} = \frac12 \tth^{\frac32}\tp \sqrt{\frac{\tp^2-1}{3}} \Big(1 -\frac{3\tp^2+8}{40} \tth^2 + O(\tth^4) \Big) \, .
\end{equation}
Furthermore 
$\omega_j^\sigma$ in \eqref{omegajsigma} satisfies, for any $ j\in \bN $, 
$ \sigma = \pm $, 
as $ \tth \to 0^+ $, 
\begin{align}\notag
\omega_j^+ - \omega_*^{(\tp)} &= \tth^{\frac52}  \Big( \frac{j^3-j+\tp-\tp^3}{6}  - \Big( \frac{19}{60}(j^2+1)  \frac{j^3-j+\tp-\tp^3}{6} - \frac{\tp (\tp^2-1)(\tp^2-j^2)}{90}\Big) \tth^2  +O(\tth^4)\Big) \, ,    \\ \label{diffomegas}
 \omega_j^- - \omega_*^{(\tp)} &=   2j \tth^{\frac12} \Big( 1 -\frac{j^2+1}{12}  \tth^2 +O(\tth^4) \Big) \, . 
\end{align} 
Furthermore for any $j\in \bZ$ and $\ell \in \bN$,  as $\tth\to 0^+$,
\begin{align}\label{exppjjajjbis}
&a_\ell^{[\ell]} - p_\ell^{[\ell]} (\sigma t_j^{(\tp)} + \sigma' t_{j+\ell}^{(\tp)})   \\ \notag 
&=\begin{cases} 3\ell \Big(\dfrac38 \Big)^{\ell-1} \tth^{2-3\ell} \big( 1 + \dfrac{19\ell^2+12j\ell-27\ell+12j^2+2}{54} \tth^2 + O(\tth^4) \big) \,,  &\sigma = \sigma'= + \, ,\\[2mm]
  - 5\ell \Big(\dfrac38 \Big)^{\ell-1} \tth^{2-3\ell} \big(1+O(\tth^2 ) \big)
    \,, &\sigma=\sigma'=-\, ,\\[2mm]
 - \ell \Big(\dfrac38 \Big)^{\ell-1}\tth^{2-3\ell}   \big(1  + \dfrac{25\ell^2-12j\ell-9\ell+2}{18} \tth^2 + O(\tth^4) \big) \,,   & \sigma\neq \sigma'=+\,,\\[2mm]
- \ell \Big(\dfrac38 \Big)^{\ell-1}\tth^{2-3\ell}  \big(1+O(\tth^2 ) \big)
  \,, & \sigma\neq \sigma'=-\,  .
\end{cases}
\end{align}
\end{lem}

\begin{proof}
 By \eqref{tjp} and \eqref{expuphinuova1} we deduce \eqref{expOmegajtj} and  \eqref{startingsqrt}. Similarly \eqref{omeghino}, \eqref{tjp} and \eqref{expOmegajtj} imply \eqref{diffomegas}.
The last expansion \eqref{exppjjajjbis}
relies on the previous ones together with
Lemma \ref{paasymashto0+}.
\end{proof}

We first show that many of the addends in \eqref{beta1exp}, namely those with 
at least 
two negative signs $ \sigma_{i}
= \sigma_j = -  $, $ i \neq j $, do {\it not}  contribute to the main terms of the asymptotic expansion in \eqref{beta1ash0+}.
This is an important computational simplification.

\begin{lem}\label{ordinivincenti}
For any integer $q=1,\dots,\tp-1$  and any signs $\sigma_1,\dots,\sigma_q \in \{ \pm \} $ we denote  $\vec{\sigma} :=(\sigma_1,\dots,\sigma_q) $  and its cardinality by 
$|\vec{\sigma}|:= |\{ \sigma_j=-\}|$.  Then, for any
$0 < j_1 < \dots < j_q < \tp$,
the term
$ \betone{q}{j_1,\dots,j_q}{\sigma_1,\dots,\sigma_q} $
in \eqref{beta1exp} satisfies, as $\tth \to 0^+$,
\begin{equation}\label{hierarchy}
\betone{q}{j_1,\dots,j_q}{\sigma_1,\dots,\sigma_q} = O\big(\tth^{\frac72-3\tp+2|\vec{\sigma}|} \big) \, .
\end{equation}
\end{lem}
\begin{proof}
By \eqref{entanglementfdj}, \eqref{expOmegajtj} and \eqref{exppjjajjbis}, for any  $j,\ell>0$ 
we have 
$ \e^{-\ell}\ent{\ell}{\ell}{j+\ell}{j}{\sigma'}{\sigma}  = O(\tth^{\frac52-3\ell}) $ whereas when $j=0$ using \eqref{startingsqrt} we get $ \e^{-\ell} \ent{\ell}{\ell}{\ell}{0}{\sigma'}{\sigma}  = O(\tth^{\frac72-3\ell}) $.
Therefore 
\begin{equation}\label{numeratorh0+}
 \e^{-\tp}\ent{j_1}{j_1}{j_1}{0}{\sigma_1}{-}\ent{j_2-j_1}{j_2-j_1}{j_2}{j_1}{\sigma_2}{\sigma_1}\dots \ent{\tp-j_q}{\tp-j_q}{\tp}{j_q}{+}{\sigma_q} = O(\tth^{\frac52 q+\frac72-3\tp})\, .
\end{equation}
On the other hand, by \eqref{diffomegas} and since  $j_i \geq 1 $ 
and $j_i^3-j_i+\tp-\tp^3 < 0  $  for any $i=1,\dots,q$, we deduce that
\begin{equation}\label{denominatorh0+}
 (\omega_{j_1}^{\sigma_1}-\omega_*^{(\tp)})\dots(\omega_{j_q}^{\sigma_q}-\omega_*^{(\tp)}) \sim C_{j_1,\dots,j_q}^{\sigma_1,\dots,\sigma_q} \tth^{\frac52(q-|\vec{\sigma}|)} \tth^{\frac12|\vec{\sigma}|}
 \qquad
 \text{for some} \quad 
 C_{j_1,\dots,j_q}^{\sigma_1,\dots,\sigma_q} \neq 0\, .
\end{equation}
The estimate \eqref{hierarchy} follows by \eqref{beta1exp}, \eqref{numeratorh0+} and \eqref{denominatorh0+}.
\end{proof}

We now  exhibit the explicit asymptotic expansion of the function  $\beta_1^{(\tp)}(\tth)$ in \eqref{beta1exp}. We  provide the expansion up to the order $ \tth^{\frac{11}{2} - 3\tp }$  since,
as we prove in 
Lemma \ref{lem:cozero}
below, the term of order $\tth^{\frac72-3\tp }$ vanishes.

\begin{prop}\label{summer2023}
The function  $\beta_1^{(\tp)}(\tth)$ in \eqref{beta1exp} admits the   asymptotic expansion, as $\tth \to 0^+$, 
\begin{equation}\label{beta1ash0+}
\beta_1^{(\tp)}(\tth)  =  -\sqrt{\frac{\tp^2-1}{3}} \frac{3^{\tp-1}\tp}{8^\tp}  \Big\{ A^{(\tp)} \tth^{\frac72-3\tp }   +\Big[ \Big( \frac{299}{1080} \tp^2 -\frac12\tp- \frac4{45} \Big) A^{(\tp)} + C^{(\tp)}  \Big] \tth^{\frac{11}2-3\tp} + O(\tth^{\frac{15}2-3\tp}) \Big\}
\end{equation}
where
\begin{align}\label{Ap}
    &A^{(\tp)} := \tp+\sum_{q=1}^{\tp-1} \sum_{0<j_1<\dots<j_q < \tp} \ac{q}(j_1,\dots,j_q)\, ,\\
    \label{Cp}
    &C^{(\tp)} := \frac{28}{27}\tp^3  +\sum_{q=1}^{\tp-1} \sum_{0<j_1<\dots<j_q < \tp} \ac{q}(j_1,\dots,j_q)  \mathtt{g}^{(\tp)}_q(j_1,\dots,j_q)\, ,
\end{align}
with
\begin{equation}
\label{cancelingasympcoeff}
\ac{q}(j_1,\dots,j_q) := (-12)^{q} j_1 \dots j_q \frac{j_1 (j_2-j_1) \dots (j_q-j_{q-1}) (\tp-j_q)}{(\tp^3-\tp + j_1 - j_1^3 )\dots (\tp^3-\tp + j_q - j_q^3 )}   
\end{equation}
and
\begin{align} \label{gq}
&\mathtt{g}^{(\tp)}_q(j_1,\dots,j_q) := \frac13 \Big[ \frac{49}{45} q + \frac{32}{9} j_1^2 - \frac{4}{9} - \frac{4}{9} \frac{\tp^3-\tp}{j_1} +\frac{5}{18}(\tp^3-\tp)\Big(\frac1{j_1}+\dots+\frac{1}{j_q} \Big) \\ \notag
&+\frac{38}{15} \big(j_1^2+\dots+j_q^2 \big) + \frac{\tp^3-\tp}{5}\Big(\frac{ \tp+j_1}{\tp^2+j_1^2+\tp j_1-1}+\dots+\frac{ \tp+j_q}{\tp^2+j_q^2+\tp j_q-1}\Big) - \frac{13}{9} \big(j_1j_2+j_2j_3 +\dots+j_q\tp \big) \Big]\, .
\end{align}

\end{prop}

\begin{proof}
We split the proof in different steps.
\\[1mm]
{\bf Step 1: reduction}.  Lemma \ref{ordinivincenti} implies that the asymptotic expansion of  $\beta_1^{(\tp)}(\tth)$ as $\tth \to 0^+$ is given only by terms with at most one minus sign appearing in \eqref{beta1exp}, i.e.
\begin{equation}
\begin{aligned}
\beta_1^{(\tp)}(\tth) & =  b_0^{(\tp)}(\tth) + \sum_{q=1}^{\tp-1} \sum_{0<j_1<\dots <j_q < \tp}   \betone{q}{j_1,\dots,j_q}{+,\dots,+}(\tth)\\  & \ - \sum_{q=1}^{\tp-1} \sum_{i=1}^q \sum_{0<j_1<\dots<j_q < \tp}   \betone{q}{j_1,\dots, j_{i-1},j_i,j_{i+1},\dots,j_q}{+,\dots,\ +\ ,\ -,\ +,\dots,+}(\tth) + O\big( \tth^{\frac{15}2-3\tp} \big) \, . 
\end{aligned}
\end{equation}
We compute the expansion at $\tth \to 0^+$ of these  coefficients.
\\[1mm]
{\bf Step $2$: computation of $ b_0^{(\tp)}(\tth)$.}
By \eqref{beta1exp}, \eqref{entanglementfdj}, \eqref{startingsqrt} and \eqref{exppjjajjbis} we obtain 
\begin{align}\label{bash0+1}
 b_0^{(\tp)}(\tth) =-\sqrt{\frac{\tp^2-1}{3}} \frac{3^{\tp-1}\tp}{8^\tp}  \Big[\tp \tth^{\frac72-3\tp} +  \bc{0} \tth^{\frac{11}2-3\tp}  + O\big( \tth^{\frac{15}2-3\tp} \big) \Big] 
\end{align}
with 
\begin{equation}\label{bc0ecco}
\bc{0} :=   \Big(\frac{473}{360}\tp^2 - \frac12\tp- \frac{4}{45} \Big)  \tp\, .
\end{equation}
\\[1mm]
{\bf Step $3$: computation of $\betone{q}{j_1,\dots,j_q}{+,\dots,+}(\tth)$.}
For any $q=1,\dots,\tp -1$ and any  
integer $0<j_1<\dots<j_q<\tp$, using \eqref{beta1exp} with $\sigma_1=\dots=\sigma_q = +$ together with \eqref{entanglementfdj} and \eqref{expOmegajtj}-\eqref{exppjjajjbis} we get that  
\begin{equation}
\label{N+D+}
\betone{q}{j_1,\dots,j_q}{+,\dots,+}(\tth) =  \frac{N^+}{D^+}
\end{equation}
where, neglecting terms $ O(\tth^4) $, 
\begin{align}\label{bash0+2aux}
 N^+ & :=\Big( \frac14 \Big)^{q+1} \sqrt{\Omega_0^{(\tp)}\Omega_\tp^{(\tp)}} \Omega_{j_1}^{(\tp)}\dots \Omega_{j_q}^{(\tp)}  \big( a_{j_1}^{[j_1]} - p_{j_1}^{[j_1]} 
 (- t_0^{(\tp)} +  t_{j_1}^{(\tp)})\big) \\  \notag
 &\ \ \times  \big( a_{j_2-j_1}^{[j_2-j_1]} - p_{j_2-j_1}^{[j_2-j_1]} ( t_{j_1}^{(\tp)} +  t_{j_2}^{(\tp)})\big) \dots \big( a_{\tp-j_q}^{[\tp-j_q]} - p_{\tp-j_q}^{[\tp-j_q]} ( t_{\tp}^{(\tp)} +  t_{j_q}^{(\tp)})\big)  \\ \notag
 &= \tth^{\frac72+\frac52 q- 3\tp} \frac{3^{\tp-1}}{8^\tp} \sqrt{\frac{\tp^2-1}{3}}\tp2^q   \Big(1 -\tfrac{3\tp^2+8}{40} \tth^2 \Big)  j_1 \dots j_q\prod_{i=1}^q \Big(1   +
 \tfrac{\tp^3-\tp-2j_i^3}{12 j_i}\tth^2 \Big) 
  \\[-2mm] \notag
 &\ \ \times(- j_1) (j_2-j_1) \dots (\tp-j_q) \big(1  + \tfrac{25j_1^2-9j_1+2}{18} \tth^2 \big) \Big[\prod_{i=1}^{q-1}\big( 1 + \tfrac{19(j_{i+1}-j_i)^2+12j_i(j_{i+1}-j_i)-27(j_{i+1}-j_i)+12j_i^2+2}{54} \tth^2\big) \Big] \\ \notag
 &\ \ \times \big( 1 + \tfrac{19(\tp-j_q)^2+12j_q(\tp-j_q)-27(\tp-j_q)+12j_q^2+2}{54} \tth^2\big)  \\ \notag
 \end{align}
 and 
\begin{equation}\label{formaD}
D^+:= 
\tth^{\frac52q} (-6)^{-q}\underbrace{ (\tp^3-\tp-j_1^3+j_1)\dots (\tp^3-\tp-j_q^3+j_q) }_{=:D_1^+}
\prod_{i=1}^q \Big[  1 -\Big( \tfrac{19}{60} (j_i^2 +1) + \tfrac{1}{15} \tfrac{\tp(\tp^2-1) (j_i+\tp)}{j_i^2+j_i \tp+\tp^2-1} \Big) \tth^2\Big]
\, . 
\end{equation}
We write  the terms $ N^+, D^+ $ 
in \eqref{bash0+2aux} and
\eqref{formaD} as 
\begin{equation}
\label{ndtot}
\begin{aligned}
& N^+ = -\tth^{\frac72+\frac52 q- 3\tp} \frac{3^{\tp-1}}{8^\tp} \sqrt{\frac{\tp^2-1}{3}}\tp2^q  \, N_1^+  \big( 1 + \tth^2  N_2^+ +  O(\tth^4)\big) 
\, , \\
& D^+ =\tth^{\frac52q} (-6)^{-q}  D_1^+ \big( 1 + \tth^2  D_2^+ + O(\tth^4)  \big) \, , 
\end{aligned}
\end{equation}
where 
\begin{equation}\label{N1}
    N_1^+ := j_1 \ldots j_q \, j_1 (j_2-j_1) \ldots(\tp-j_q)
\end{equation}
whereas, setting $j_{q+1}:= \tp$, 
\begin{align}
\notag
N_2^+ & := - \tfrac{3\tp^2+8}{40} + \sum_{i=1}^q \tfrac{\tp^3-\tp-2j_i^3}{12 j_i} + \tfrac{25 j_1^2 - 9 j_1 +2}{18} + \sum_{i=1}^q \tfrac{19(j_{i+1}-j_i)^2 +12 j_i (j_{i+1}-j_i) -27 (j_{i+1}-j_i) + 12j_i^2 + 2 )}{54} \\
\notag
&=  - \tfrac{3\tp^2+8}{40} + \tfrac{\tp^3-\tp}{12} \Big(\tfrac{1}{j_1}+\dots + \tfrac{1}{j_q} \Big) - \tfrac16 (j_1^2+\dots+j_q^2) + \tfrac{25 j_1^2 - 9 j_1 +2}{18} + \sum_{i=1}^q \tfrac{19 j_{i+1}^2 +j_i^2 -26 j_i j_{i+1} -27 (j_{i+1}-j_i) + 2 }{54} \\
\notag
&=  - \tfrac{3\tp^2+8}{40} + \tfrac{\tp^3-\tp}{12} \Big(\tfrac{1}{j_1}+\dots + \tfrac{1}{j_q} \Big) - \tfrac16 (j_1^2+\dots+j_q^2) - \tfrac12\tp + \tfrac1{27} q + \tfrac19 \\ 
\notag
&\quad  + \big( \tfrac{25}{18} - \tfrac{19}{54}\big) j_1^2 + \tfrac{19}{27} \big( j_1^2 + \dots + j_q^2 \big)   + \tfrac{19}{54} \tp^2 -\tfrac{13}{27} \big( j_1 j_{2} + j_2 j_3 + \dots + j_q\tp\big) \\
\label{N2}
&=\tfrac{299}{1080} \tp^2 - \tfrac12\tp- \tfrac4{45} + \tfrac1{27} q +  \tfrac{\tp^3-\tp}{12} \Big(\tfrac{1}{j_1}+\dots + \tfrac{1}{j_q} \Big) \\
\notag &\quad + \tfrac{29}{54} (j_1^2+\dots+j_q^2) + \tfrac{28}{27} j_1^2  -\tfrac{13}{27} \big( j_1 j_{2} + j_2 j_3 + \dots + j_q\tp\big) 
\end{align}
and
\begin{align} \label{D2p}
D_2^+ :&= - \sum_{i=1}^q  \tfrac{19}{60} (j_i^2 +1) + \tfrac{1}{15} \tfrac{\tp(\tp^2-1) (j_i+\tp)}{j_i^2+j_i \tp+\tp^2-1}  \\ \notag  
&= -\tfrac{19}{60} (j_1^2 + \dots + j_q^2) - \tfrac{19}{60} q - \tfrac{\tp(\tp^2-1)}{15} \big( \tfrac{j_1+\tp}{\tp^2+j_1\tp+j_1^2-1} + \dots+ \tfrac{j_q+\tp}{\tp^2+j_q\tp+j_q^2-1} \big)  \, .
\end{align}
In conclusion,  in view of \eqref{N+D+} and \eqref{ndtot}, by the explicit expressions of $N_1, D_1, N_2, D_2$ given  respectively in 
\eqref{N1}, \eqref{formaD}, \eqref{N2}, \eqref{D2p} we obtain
\begin{align}\label{quellacoipiu}
    \betone{q}{j_1,\dots,j_q}{+,\dots,+}(\tth) & = 
\frac{N^+}{D^+} = 
-\tth^{\frac72- 3\tp} \frac{3^{\tp-1}}{8^\tp} \sqrt{\frac{\tp^2-1}{3}} \,\tp\,  (-12)^q \, 
\frac{N_1^+}{
D_1^+}
(1+ (N_2^+-D_2^+) \tth^2 +
O(\tth^4)) \\ \notag
& =-\sqrt{\frac{\tp^2-1}{3}} \frac{3^{\tp-1}\tp}{8^\tp}  \Big[\ac{q}(j_1,\dots,j_q) \tth^{\frac72-3\tp} + \bc{q}(j_1,\dots,j_q) \tth^{\frac{11}2-3\tp} + O\big( \tth^{\frac{15}2-3\tp} \big) \Big]   
\end{align}
 with 
 $ \ac{q}(j_1,\dots,j_q):= (-12)^q \frac{N_1^+}{D_1^+}$
 given explicitly  in \eqref{cancelingasympcoeff} and 
 $\bc{q}(j_1,\dots,j_q):= (-12)^q \frac{N_1^+}{D_1^+} (N_2^+-D_2^+)$ given explicitly by 
\begin{align}\label{bcqecco}
& \bc{q}(j_1,\dots,j_q) =  \Big( \frac{299}{1080} \tp^2 -\frac12\tp- \frac4{45} + \frac{191}{540}q +\frac{28}{27} j_1^2 + \frac{461}{540}\sum_{i=1}^q j_i^2 \\ \notag
& \qquad\qquad  \qquad + \frac{\tp^3-\tp}{12} \sum_{i=1}^q \frac{1}{j_i} + \frac{\tp^3-\tp}{15} \sum_{i=1}^q \frac{\tp+j_i}{\tp^2+j_i^2+\tp j_i-1}  - \frac{13}{27} \sum_{i=1}^q j_i j_{i+1} \Big) \ac{q}(j_1,\dots,j_q)
 \end{align} 
where we set $ j_{q+1} :=\tp$.

\noindent{\bf Step $4$: computation of $\betone{q}{j_1,\dots, j_{i-1},j_i,j_{i+1},\dots,j_q}{+,\dots,\ +\ ,\ -,\ +,\dots,+}(\tth)$.} For any $q=1,\dots,\tp-1$ and any  
integer $0<j_1<\dots<j_q<\tp$, using \eqref{beta1exp} with $\sigma_1=-$, $\sigma_2=\dots=\sigma_q = +$ together with \eqref{entanglementfdj} and \eqref{expOmegajtj}-\eqref{exppjjajjbis} we get that  
\begin{equation}\label{NiDi}
\betone{q}{j_1,\dots, j_{i-1},j_i,j_{i+1},\dots,j_q}{+,\dots,\ +\ ,\ -,\ +,\dots,+}(\tth) = \frac{N_i}{ D_i} \, ,
\quad \forall 
i=1,\dots, q \, , 
\end{equation}
where, neglecting terms $ O(\tth^2) $,
\begin{equation}\label{Di}
 D_i:= 2 (-6)^{1-q} \tth^{\frac52q-2} j_i \prod_{k=1\atop k\neq i}^q  (\tp^3-\tp-j_k^3+j_k)\, ,\vspace{-10mm}
\end{equation}
and
\begin{align}\label{N1-}
N_1 & :=\Big( \frac14 \Big)^{q+1} \sqrt{\Omega_0^{(\tp)}\Omega_\tp^{(\tp)}} \Omega_{j_1}^{(\tp)}\dots \Omega_{j_q}^{(\tp)}  \big( a_{j_1}^{[j_1]} - p_{j_1}^{[j_1]} 
 (- t_0^{(\tp)} -  t_{j_1}^{(\tp)})\big)  \big( a_{j_2-j_1}^{[j_2-j_1]} - p_{j_2-j_1}^{[j_2-j_1]} ( - t_{j_1}^{(\tp)} +  t_{j_2}^{(\tp)})\big) \\  \notag
 &\ \ \times  \big( a_{j_3-j_2}^{[j_3-j_2]} - p_{j_3-j_2}^{[j_3-j_2]} ( t_{j_2}^{(\tp)} +  t_{j_3}^{(\tp)})\big) \dots \big( a_{\tp-j_q}^{[\tp-j_q]} - p_{\tp-j_q}^{[\tp-j_q]} ( t_{\tp}^{(\tp)} +  t_{j_q}^{(\tp)})\big)  \\ \notag
 &= \tth^{\frac72+\frac52 q- 3\tp} \frac{3^{\tp-1}}{8^\tp} \sqrt{\frac{\tp^2-1}{3}}\frac{\tp2^q }{3}  j_1 \dots j_q 
(-5 j_1) \big(-(j_2-j_1)\big) (j_3-j_2) \dots (\tp-j_q) \, ,
 \end{align}
whereas for any  $i = 2, \ldots, q $, setting $j_{q+1}:=\tp$,
\begin{align}\label{Ni-}
N_i :&= \Big( \frac14 \Big)^{q+1} \sqrt{\Omega_0^{(\tp)}\Omega_\tp^{(\tp)}} \Omega_{j_1}^{(\tp)}\dots \Omega_{j_q}^{(\tp)}  \big( a_{j_1}^{[j_1]} - p_{j_1}^{[j_1]} 
 (- t_0^{(\tp)} +  t_{j_1}^{(\tp)})\big)  \big( a_{j_{i}-j_{i-1}}^{[{j_{i}-j_{i-1}}]} - p_{j_{i}-j_{i-1}}^{[{j_{i}-j_{i-1}}]} (  t_{j_{i-1}}^{(\tp)} -  t_{j_i}^{(\tp)})\big) \\  \notag
 &\ \ \times  \big( a_{j_{i+1}-j_{i}}^{[{j_{i+1}-j_{i}}]} - p_{j_{i+1}-j_{i}}^{[{j_{i+1}-j_{i}}]} (-  t_{j_{i}}^{(\tp)} +  t_{j_{i+1}}^{(\tp)})\big) \prod_{k=1 \atop k\neq i-1,i}^{q}\big( a_{j_{k+1}-j_k}^{[j_{k+1}-j_k]} - p_{j_{k+1}-j_k}^{[j_{k+1}-j_k]} ( t_{j_k}^{(\tp)} +  t_{j_{k+1}}^{(\tp)})\big)   \\[-2mm] \notag
 &= \tth^{\frac72+\frac52 q- 3\tp} \frac{3^{\tp-1}}{8^\tp} \sqrt{\frac{\tp^2-1}{3}}   \frac{\tp 2^q}9 j_1 \dots j_q 
(-j_1) (j_2-j_1) \dots (\tp-j_q) =  -\frac{N_1}{15} \, .
 \end{align}
By \eqref{NiDi},  \eqref{Di}, \eqref{N1-} and \eqref{Ni-} we obtain
\begin{align}
\label{bash0+3}
 &\betone{q}{j_1,\dots, j_{i-1},j_i,j_{i+1},\dots,j_q}{+,\dots,\ +\ ,\ -,\ +,\dots,+}(\tth) = -\sqrt{\frac{\tp^2-1}{3}} \frac{3^{\tp-1}\tp}{8^\tp}  \Big[\hbc{q}(i;j_1,\dots,j_q) \tth^{\frac{11}2-3\tp} + O\big( \tth^{\frac{15}2-3\tp} \big)\Big]  
 \end{align} 
with
\begin{align}\label{bpIIasymp}
&\hbc{q}(i;j_1,\dots,j_q):=  \begin{cases} \dfrac{5 (\tp^3-\tp+j_1-j_1^3)}{36 j_1} \ac{q}(j_1,j_2,\dots,j_q)  & \textup{if } i=1\, ,  \\[3mm] - \dfrac{\tp^3-\tp+j_i-j_i^3}{108 j_i} \ac{q}(j_1,j_2,\dots j_q) & \textup{if } i = 2, \ldots, q  \, . \end{cases}
\end{align}
{\bf Step 5: conclusion}. 
The sum of the coefficients of order  $\tth^{\frac{7}{2}-3\tp}$ in \eqref{bash0+1} and \eqref{quellacoipiu} gives  $A^{(\tp)}$ in \eqref{Ap}. By summing the coefficients of order  $\tth^{\frac{11}{2}-3\tp}$ in \eqref{bash0+1}-\eqref{bc0ecco}, \eqref{quellacoipiu}-\eqref{bcqecco} and \eqref{bash0+3}-\eqref{bpIIasymp} we obtain the coefficient of $\tth^{\frac{11}{2}-3\tp}$ in \eqref{beta1ash0+} as, setting $\ac{q}:= \ac{q}(j_1,\dots,j_q)$ for brevity,
\begin{align}
& \bc{0}+\sum_{q=1}^{\tp-1} \sum_{0<j_1<\dots<j_q < \tp} \bc{q}(j_1,\dots,j_q) -  \sum_{q=1}^{\tp-1} \sum_{0<j_1<\dots <j_q < \tp} \sum_{i=1}^{q}  \hbc{q}(i;j_1,\dots,j_q)  
\\ \notag
& = \Big[\Big( \tfrac{299}{1080} + \tfrac{28}{27} \Big) \tp^2 - \tfrac12\tp- \tfrac{4}{45} \Big] \tp+\sum_{q=1}^{\tp-1} \sum_{0<j_1<\dots<j_q < \tp} \Big( \tfrac{299}{1080} \tp^2 -\tfrac12\tp- \tfrac4{45} \Big)  \ac{q}
\\ \notag
&+ \sum_{q=1}^{\tp-1} \sum_{0<j_1<\dots<j_q < \tp} \Big[ \tfrac{191}{540} q + \tfrac{28}{27} j_1^2 + \tfrac{461}{540} \sum_{i=1}^q j_i^2 + \tfrac{\tp^3-\tp}{12} \sum_{i=1}^{q} \tfrac{1}{j_i} + \tfrac{\tp^3-\tp}{15} \sum_{i=1}^{q}  \tfrac{\tp+j_i}{\tp^2+j_i^2+\tp j_i-1} -\tfrac{13}{27} \sum_{i=1}^{q} j_i j_{i+1}\Big]\ac{q} \\ \notag
&- \sum_{q=1}^{\tp-1} \sum_{0<j_1<\dots<j_q < \tp}  \Big[ \tfrac{5(\tp^3-\tp+j_1-j_1^3)}{36j_1} + \tfrac{\tp^3-\tp+j_1-j_1^3}{108 j_1} \Big] \ac{q} + \sum_{q=1}^{\tp-1} \sum_{0<j_1<\dots<j_q < \tp} \sum_{i=1}^q \Big[ \tfrac{\tp^3-\tp}{108 j_i} +\tfrac{1}{108} -\tfrac{j_i^2}{108} \Big] \ac{q} \\ \notag 
& = \Big( \tfrac{299}{1080} \tp^2 -\tfrac12\tp- \tfrac4{45} \Big) \Big(\overbrace{ \tp+\sum_{q=1}^{\tp-1} \sum_{0<j_1<\dots<j_q < \tp}\!\!\!\!\!\!\!\!\! \ac{q}(j_1,\dots,j_q)}^{= A^{(\tp)}\text{ by \eqref{Ap}}} \Big) + \tfrac{28}{27} \tp^3
\\ \notag
&+ \sum_{q=1}^{\tp-1} \sum_{0<j_1<\dots<j_q < \tp} \Big[ \big(\tfrac{191}{540}+\tfrac{1}{108}\big) q + \tfrac{28}{27} j_1^2 + \big(\tfrac{461}{540} -\tfrac{1}{108} \big) \sum_{i=1}^q j_i^2 + (\tp^3-\tp) \big(\tfrac{1}{12} + \tfrac{1}{108}\big)
 \sum_{i=1}^{q} \tfrac{1}{j_i} \\\notag
 &+ \tfrac{\tp^3-\tp}{15} \sum_{i=1}^{q}  \tfrac{\tp+j_i}{\tp^2+j_i^2+\tp j_i-1} -\tfrac{13}{27} \sum_{i=1}^{q} j_i j_{i+1}\Big]\ac{q} - \sum_{q=1}^{\tp-1} \sum_{0<j_1<\dots<j_q < \tp}  \tfrac{4}{27}\Big[ \tfrac{\tp^3-\tp}{j_1} + 1 - j_1^2 \Big] \ac{q} 
\end{align}
which, by elementary sums, is equal to 
$
 \Big( \frac{299}{1080} \tp^2 -\frac12\tp- \frac4{45} \Big) A^{(\tp)} + C^{(\tp)} 
$ 
appearing in \eqref{beta1ash0+}.
\end{proof}

As 
anticipated, the problem is  {\it degenerate}   as $ \tth \to 0^+ $, namely 
the term  of order $\tth^{\frac72-3\tp} $ 
of the asymptotic expansion \eqref{beta1ash0+}   is actually zero. 

\begin{lem}[Degeneracy as $ \tth \to 0^+ $]\label{lem:cozero}
For any  $\tp \in \bN$, $\tp \geq 2$, the coefficient 
$ A^{(\tp)} $  in \eqref{Ap} is  
\begin{equation}\label{estate2023}
A^{(\tp)}  = 0\,  \, .  
\end{equation}
\end{lem}

\begin{proof}
In  Appendix \ref{sec:canc}. 
\end{proof}


{The subsequent term, of order $\tth^{\frac{11}{2}-3\tp} $ 
is shown to be non-zero for any integer $ \tp  \geq 2 $ by a result we conjectured and whose proof is due to Koutschan, Van Hoeij, and Zeilberger in \cite{vHZ}.}

\begin{clem}[Koutschan - Van Hoeij - Zeilberger \cite{vHZ}]\label{primasommma}
For any  $\tp \in \bN$, $\tp \geq 2$,  the coefficient 
$C^{(\tp)}$ in \eqref{Cp} is
\begin{equation}\label{primavera2024}
C^{(\tp)}  = \frac{\tp(\tp+1)^2}{3}\,  .
\end{equation} 
\end{clem}
\begin{proof}
{The result is in \cite[Formula (4)]{vHZ}}. 
For the reader's convenience we report a proof in \cite{BCMVadd}. We remark that the function $g_q^{(\tp)}(j_1,\dots,j_q) $ in \cite[Formula (2)]{vHZ} is related to the function $\tg_q^{(\tp)}(j_1,\dots,j_q) $ in \eqref{gq} by the identity
$$
g_q^{(\tp)}(j_1,\dots,j_q) = 3\tg_q^{(\tp)}(j_1,\dots,j_q) - \frac{28}{9}\tp^2 \, .
$$
Using \cite[Formula (3)]{vHZ} one readily   verifies that \cite[Formula (4)]{vHZ} is equivalent to \eqref{primavera2024}. 
\end{proof}
\noindent{\it Proof of  Theorem \ref{asymptotics}}. 
Insert the value $A^{(\tp)} = 0 $ 
and $C^{(\tp)}$ in \eqref{primavera2024} in 
formula \eqref{beta1ash0+}.
\qed

\subsection{Limit of $\beta_1^{(\tp)}(\tth)$ as $\tth\to +\infty$}\label{infinitelim}

We now prove that the term $\beta_1^{(\tp)}(\tth)$ in \eqref{beta1exp} vanishes in the infinite-depth limit.

\begin{prop}\label{deeplimit}
For any $\tp\geq 2 $,
it results $ \lim\limits_{\tth \to +\infty} \beta_1^{(\tp)}(\tth) = 0 $. 
\end{prop}

The rest of this section is devoted to the proof of Proposition 
\ref{deeplimit}.

\begin{lem}
For any $j\geq 0$, 
the terms in \eqref{tjp} and \eqref{omegajsigma} satisfy 
\begin{equation}\label{tjpinfinito} 
\begin{aligned}
& \lim_{\tth \to + \infty}\Omega_j^{(\tp)} = 
\lim_{\tth \to + \infty} t_j^{(\tp)} =  \sqrt{j+\frac{(\tp-1)^2}{4}}=: \nu_j^{(\tp)}
  \, , \\ 
&   \lim_{\tth \to + \infty} \omega_j^\sigma - \omega_*^{(\tp)} =  \frac{2j-\tp+1}{2} - \sigma \nu_j^{(\tp)}\, .
\end{aligned}
 \end{equation}
 The terms in \eqref{beta1exp}-\eqref{entanglementfdj} satisfy,
 for any $1\leq q \leq \tp-1$,  $0<j_1<\dots<j_q<\tp$ and $\sigma_1,\dots,\sigma_q=\pm$,
\begin{equation}\label{cancellazionimagiche1}
\lim_{\tth \to +\infty} b_0^{(\tp)}  = 0\, , \qquad 
\lim_{\tth \to \infty} \betonep{q}{j_1,\dots,j_q}{\sigma_1,\dots,\sigma_q} = A_{j_1,\dots,j_q} B_{j_1,\dots,j_q}^{\sigma_1,\dots,\sigma_q}\, ,
\end{equation}
where, recalling  $\varpi_{j}$ in \eqref{paugualiinf} and setting
 $(j_{q+1},\sigma_{q+1}):=(\tp,+)$,
\begin{align}\label{Bjssigmas}
&A_{j_1,\dots,j_q} :=(-1)^q \frac{\sqrt{\tp^2-1}\nu_{j_1}^{(\tp)}\dots  \nu_{j_q}^{(\tp)}\varpi_{j_1} \varpi_{j_2-j_1} \dots \varpi_{j_q-j_{q-1}} \varpi_{\tp-j_q}}{8^{q+1}j_1 (\tp-j_1)\dots j_q(\tp-j_q) } \\ 
\label{Bjssigmas1} 
&
B_{j_1,\dots,j_q}^{\sigma_1,\dots,\sigma_q}:= \big( 1+ \nu_0^{(\tp)}-\sigma_{1} \nu_{j_1}^{(\tp)} \big) \prod_{i=1}^{q} \big( 1-\sigma_i \nu_{j_i}^{(\tp)}-\sigma_{i+1} \nu_{j_{i+1}}^{(\tp)} \big) \big( 2j_i-\tp+1+2\sigma_i \nu_{j_i}^{(\tp)} \big)\, .
\end{align}
\end{lem}

\begin{proof}
The expansion  
\eqref{tjpinfinito} 
follows by \eqref{expaphip}.
By \eqref{beta1exp}-\eqref{entanglementfdj},  
\eqref{tjpinfinito}  and 
\eqref{paugualiinf} we obtain
\begin{equation*}
\lim_{\tth \to \infty} b_0^{(\tp)} =  \frac{\sqrt{\tp^2-1}}{8} \varpi_\tp \Big(1- \frac{\tp+1}{2}+\frac{\tp-1}{2} \Big)  = 0 
\end{equation*}
which gives the first identity \eqref{cancellazionimagiche1}. 
Next, by \eqref{beta1exp}-\eqref{entanglementfdj}, \eqref{tjpinfinito}  and 
\eqref{paugualiinf} we get   
\begin{align*}
&\lim_{\tth \to \infty} \betonep{q}{j_1,\dots,j_q}{\sigma_1,\dots,\sigma_q} =  \frac{\sqrt{\tp^2-1}\nu_{j_1}^{(\tp)}\dots  \nu_{j_q}^{(\tp)}\varpi_{j_1} \varpi_{j_2-j_1} \dots \varpi_{j_q-j_{q-1}} \varpi_{\tp-j_q}}{2^{q+3}(2j_1-\tp+1 - 2\sigma_1 \nu_{j_1}^{(\tp)})\dots (2j_q-\tp+1 - 2\sigma_q \nu_{j_q}^{(\tp)} )} \cdot \\
&\quad \cdot (1+\nu_0^{(\tp)}-\sigma_1 \nu_{j_1}^{(\tp)}) \Bigg( \prod_{i=1}^{q-1} (1-\sigma_i \nu_{j_i}^{(\tp)}-\sigma_{i+1} \nu_{j_{i+1}}^{(\tp)}) \Bigg) (1-\sigma_q \nu_{j_q}^{(\tp)}- \nu_{\tp}^{(\tp)}) \, ,
\end{align*}
which, since 
$(2j-\tp+1 - 2 \nu_{j}^{(\tp)})(2j-\tp+1 + 2 \nu_{j}^{(\tp)}) = -4j(\tp-j) $,  proves the second identity in \eqref{cancellazionimagiche1}
with $ A_{j_1,\dots,j_q} $  and $ B_{j_1,\dots,j_q}^{\sigma_1,\dots,\sigma_q} $ 
defined in 
\eqref{Bjssigmas1}.  
\end{proof}

The key result in the proof of Proposition \ref{deeplimit}
is the following lemma. 
\begin{lem}\label{Sjsvanishing}
For any $q \in \bN$ and integers $0<j_1<\dots<j_q < \tp$ the terms in \eqref{Bjssigmas1} satisfy
\begin{equation}\label{sommaallinfinito}
S_{j_1,\dots,j_q}:=\sum_{\sigma_1,\dots,\sigma_q=\pm} \sigma_1 \dots \sigma_q B_{j_1,\dots,j_q}^{\sigma_1,\dots,\sigma_q} = 0\, .
\end{equation}

\begin{proof} 
Identity
\eqref{sommaallinfinito}
is a direct consequence of 
the following facts 
\begin{align}\label{induzq}
& S_{j_1,\dots,j_q}  
= -4\nu_{j_q}^{(\tp)} j_q\, S_{j_1,\dots,j_{q-1}} \, , \  \forall q \geq 2 \, , \\
& 
S_{j_1} = 0 \, .   \label{induz1}
\end{align}
{\it 
Proof of \eqref{induzq}.} 
For any $ q \geq 2 $ we have
\begin{equation}\label{primopasso}
S_{j_1,\dots,j_q} =\!\!\!\!\! \sum_{\sigma_1,\dots,\sigma_{q-1}=\pm}\!\!\!\!\! \sigma_1 \dots \sigma_{q-1} \Big( B_{j_1,\dots,j_{q-1},j_q}^{\sigma_1,\dots,\sigma_{q-1},+} - B_{j_1,\dots,j_{q-1},j_q}^{\sigma_1,\dots,\sigma_{q-1},-} \Big)  
\end{equation}
where by \eqref{Bjssigmas1} 
\begin{equation}
\label{bdiff} B_{j_1,\dots,j_{q-1},j_q}^{\sigma_1,\dots,\sigma_{q-1},+} - B_{j_1,\dots,j_{q-1},j_q}^{\sigma_1,\dots,\sigma_{q-1},-}  = 
\frac{ B_{j_1,\dots,j_{q-1}}^{\sigma_1,\dots,\sigma_{q-1}} }{1-\sigma_{q-1} \nu_{j_{q-1}}^{(\tp)} - \nu_\tp^{(\tp)}} \Delta_q
\end{equation}
with
\begin{equation}\label{Deltoneinfinito} 
\begin{aligned}\Delta_q\, :=  (1-\sigma_{q-1} \nu_{j_{q-1}}^{(\tp)}- \nu_{j_{q}}^{(\tp)}) (1-\nu_{j_q}^{(\tp)}-\nu_{\tp}^{(\tp)}) (2j_q-\tp+1+2\nu_{j_q}^{(\tp)} ) \\ - (1-\sigma_{q-1} \nu_{j_{q-1}}^{(\tp)}+ \nu_{j_{q}}^{(\tp)}) (1+ \nu_{j_q}^{(\tp)}-\nu_{\tp}^{(\tp)}) (2j_q-\tp+1-2\nu_{j_q}^{(\tp)} ) \,.
\end{aligned}
 \end{equation}
 Setting 
\begin{equation}\label{setabcv}
a:=1-\sigma_{q-1} \nu_{j_{q-1}}^{(\tp)} \,, \quad  b:= 1-\nu_{\tp}^{(\tp)} \stackrel{\eqref{tjpinfinito}}{=} - \frac{\tp-1}{2}\,,\quad c:= 2j_q -\tp+1\,, \quad \nu:= \nu_{j_q}^{(\tp)}\, ,
\end{equation}
 the term $\Delta_q$ in \eqref{Deltoneinfinito} is given by
 \begin{align} \notag
 \Delta_q &= (a-\nu)(b-\nu)(c+2\nu) - (a+\nu)(b+\nu)(c-2\nu) = \big( 4ab -2(a+b) c + 4\nu^2 \big) \nu \\ \label{Deltaexpression}
 & \stackrel{\eqref{setabcv}} = -4\nu_{j_q}^{(\tp)} j_q \Big(\frac{1-\tp}{2}-\sigma_{q-1} \nu_{j_{q-1}}^{(\tp)}\Big)
 = -4\nu_{j_q}^{(\tp)} j_q \Big(1-\sigma_{q-1} \nu_{j_{q-1}}^{(\tp)} - \nu_\tp^{(\tp)} \Big) 
 \end{align}
since $\frac{\tp+1}{2}=\nu_\tp^{(\tp)} $
by \eqref{tjpinfinito}. 
 By \eqref{primopasso}, \eqref{bdiff}, \eqref{Deltaexpression}  we conclude that 
$$ 
 S_{j_1,\dots,j_q} = -4\nu_{j_q}^{(\tp)} j_q\, S_{j_1,\dots,j_{q-1}} = (-4)^{q-1} \nu_{j_2}^{(\tp)} j_2 \dots \nu_{j_q}^{(\tp)} j_q \, S_{j_1}  \, .
$$ 
{\it
Proof of \eqref{induz1}.} 
By \eqref{Bjssigmas1}, $ S_{j_1}= B_{j_1}^+ - B_{j_1}^- =: \Delta_1$ which is still given by formula \eqref{Deltoneinfinito}
setting $(j_0,\sigma_0):=(0,-) $. By \eqref{Deltaexpression} we have
$$
    S_{j_1} = \Delta_1 = - 4 \nu_{j_1}^{(\tp)}j_1 \Big(1+ \nu_0^{(\tp)}-\frac{\tp+1}{2} \Big) = 0 
$$   
 by    \eqref{tjpinfinito}. 
 The proof of the lemma is complete.
\end{proof}
\end{lem}
\noindent{\it Proof of Proposition \ref{deeplimit}}.
In view of \eqref{beta1exp},  \eqref{cancellazionimagiche1} and Lemma \ref{Sjsvanishing} we conclude that
$$ 
\lim_{\tth \to +\infty } \beta_1^{(\tp)}(\tth) =  \lim_{\tth \to +\infty } b_0^{(\tp)} + \sum_{q=1}^{\tp-1} \sum_{0<j_1<\dots <j_q < \tp} A_{j_1,\dots,j_q} \sum_{\sigma_1,\dots,\sigma_q = \pm }\sigma_1 \dots \sigma_q  B_{j_1,\dots,j_q}^{\sigma_1,\dots,\sigma_q}
=0 
$$
as claimed. \qed

\smallskip

\noindent{\it Proof of Theorem \ref{thm:finalmat}}.
Theorem \ref{thm:finalmat} follows by Theorems \ref{lem:expansionL}, \ref{asymptotics} and Proposition \ref{deeplimit}. \qed

\begin{rmk}\label{beta1rmk}
If $\tp=2,3,4$, inserting in \eqref{beta1exp}-\eqref{entanglementfdj} the explicit values of the Taylor-Fourier coefficients $p_\ell^{[\ell]}$,  $a_\ell^{[\ell]}$, $\ell=1,\dots,4$ computed in \cite[Proposition A.7]{BMV_ed} we have  explicit
 formulas for $\beta_1^{(\tp)}(\tth )$. 
 They are plotted in Figures \ref{plotb1p2},  
obtained by the Mathematica sheet \texttt{beta1.nb} 
encoded in the Mathematica library \texttt{CoeffsLin.m}, collected in 
 \url{https://git-scm.sissa.it/amaspero/isolas}. 
 The zeros of $\beta_1^{(\tp)} (\tth) $ 
 coincide, for $\tp=2,3$, with the numerical 
 values given in  \cite{CDT,HY}.
\end{rmk}

\appendix

\section{Spectral collisions}
\label{spectralcollisions}

In this Appendix we prove Lemmata \ref{collemma} and \ref{thm:unpert.coll}.
\\[1mm]
\noindent{\bf Proof of Lemma \ref{collemma}}. 
We shall use that, since  $ \Omega (\varphi , \tth) = \sqrt{\varphi \tanh (\tth \varphi )} $ 
 is  strictly concave  for $ \varphi > 0 $
(cfr. \eqref{boomerang})
and $ \Omega (0, \tth) = 0 $, then  it is strictly sub-additive
\begin{equation}\label{subadd}
\Omega (\varphi + z , \tth) < \Omega (\varphi, \tth) + \Omega (z, \tth) \, , 
\qquad \forall \varphi , z > 0   \, , 
\end{equation}
and 
\begin{equation}\label{valzero}
\Omega (z, \tth) <  z \Omega (1, \tth)  \,  , \quad \forall z > 1  \, .
\end{equation}
{\sc Proof of Item $i$)}. 
We first prove that the equations
$\omega^\sigma(\varphi,\tth) = \omega^{\sigma}(\varphi+\tp,\tth)$,
for any $ \sigma = \pm $,   
have no solution,  see Figure \ref{fig.nocollision}. 
Recalling \eqref{Oomegino}, 
\begin{equation}\label{concavity}
\omega^+(\varphi,\tth) = \omega^+(\varphi+\tp,\tth) \qquad \Leftrightarrow \qquad
\ch\tp- \Omega ( \varphi + \tp, \tth )  = - \Omega ( \varphi , \tth ) \, . 
\end{equation} 
By \eqref{valzero} 
the function  $ d (\varphi, \tth ) := \ch\tp-  \Omega(\varphi +\tp,\tth) $
satisfies  $ d (0, \tth ) > 0  $ (since $\tp\geq 2 $) and 
$ d (\varphi, \tth )  > -  \Omega(\varphi +\tp,\tth) \geq -  \Omega(\varphi,\tth) $
for any $ \varphi \leq -\tp $ as well as 
$ d (\varphi, \tth )  > 0 > -  \Omega(\varphi,\tth) $
for any $ \varphi \in (-\tp,0)  $, see the right hand side of
Figure \ref{fig.nocollision}. For any $ \varphi > 0 $
we deduce by \eqref{subadd} and \eqref{valzero} that 
$  \Omega(\varphi+\tp,\tth) <  \Omega(\varphi,\tth) + \ch\tp $ and therefore  
the equation \eqref{concavity} has no solutions. 

Since $ \omega^- (\varphi, \tth) = - \omega^+ (-\varphi, \tth) $ 
the equation $ \omega^-(\varphi,\tth) = \omega^-(\varphi+\tp,\tth)$
has no solutions as well. 
\\[1mm]
{\sc Proof of Item $ii$)}.
We now prove that 
 there exists a unique positive solution
 $ \uphi  (\tp, \tth) > 0 $ of 
\begin{equation}\label{intersection}
\omega^-(\varphi,\tth) = \omega^+(\varphi+\tp,\tth) \qquad
\Leftrightarrow \qquad 
\ch\tp-  \Omega(\varphi +\tp,\tth)   = \Omega(\varphi,\tth) \, . 
\end{equation}
Indeed 
$ d (\varphi, \tth ) = \ch\tp-  \Omega(\varphi +\tp,\tth) $
satisfies $d(0,\tth) >0 $ and it  
is strictly decreasing to $-\infty$, 
whereas $\Omega(0,\tth) = 0 $ and 
$\Omega(\cdot,\tth)$ is strictly increasing to $+\infty $,
see Figure \ref{fig.Kreincollision}. This proves  \eqref{Kreincollision}. 
By  simmetry 
$\varphi=-\uphi(\tp,\tth)-\tp $ is the unique 
solution of \eqref{intersection} 
in $ (-\infty, - \tp) $  proving  \eqref{negativesolution}. 
We claim that for $-\tp \leq \varphi\leq 0 $  there are no further intersection points
between the graphs of $ d(\cdot, \tth)$ and $\Omega(\cdot,\tth) $ 
if  $\tp\geq 3$ and one intersection point $\varphi=-\frac{\tp}2 =-1$ if $\tp=2$,
see Figure \ref{fig.Kreincollision}.  
Indeed  $d(\varphi,\tth) > \Omega(\varphi,\tth) $, for $\varphi \in [-\tp,0] $,  if and only if 
$$
\tp \ch >  \Omega
\Big(\frac{\tp}2-y,\tth\Big) +  \Omega\Big(\frac{\tp}2+y,\tth\Big) =: g(y)\,,\quad \forall y \in
\Big[-\frac{\tp}2,\frac{\tp}2 \Big]\, .
$$
The function $g$ is even and, by \eqref{boomerang}, strictly concave. Hence 
it attains at $ y = 0 $  its maximum value
$$
g(0) =2 \Omega\Big(\frac{\tp}2, \tth\Big)  
\stackrel{\eqref{valzero}} < 2 \frac{\tp}{2} \Omega \big(1, \tth\big) =\tp\ch  \,,
\quad \forall \tp \geq 3 \, , 
$$
with equality holding  for $\tp= 2 $.
\\[1mm]
{\sc Proof of Item $iii$)}.
It  follows because $ \omega^- (\varphi, \tth) = - \omega^+ (-\varphi, \tth) $
amounts to
$  \Omega(\varphi,\tth) + \Omega(\varphi +\tp,\tth)  = - \ch\tp$ and the function $ \Omega (\varphi, \tth) $ is non-negative.  
\medskip

We now prove the properties of the positive function $\uphi(\tp,\tth) > 0 $ 
stated in Lemma \ref{collemma}.
\\[1mm]
{\sc Proof of  Item \ref{it1}.}
For any $ \tth > 0 $ 
 we have 
\begin{equation}\label{implicitproblem}
\mathcal{F}_{\tp} (\uphi(\tp,\tth),\tth) = 0 \qquad \text{where} \qquad 
\mathcal{F}_\tp (\varphi,\tth) := \Omega(\varphi,\tth)+\Omega(\varphi+\tp,\tth ) - \ch\tp\, . 
\end{equation}
Since $ \mathcal{F}_\tp (\varphi,\tth)  $ is  analytic and 
 $ \pa_\varphi \mathcal{F}_\tp (\varphi,\tth) > 0 $ by \eqref{boomerang}, 
 we deduce by the analytic implicit function theorem 
that $\tth\mapsto\uphi(\tp,\tth) $  is analytic. 
\\[1mm]
{\sc Proof of item \ref{it2}.}
By \eqref{implicitproblem}  and recalling that $ \ch = \Omega (1, \tth) $
we have 
\begin{align}
 \mathcal{F}_{\tp+1} (\uphi(\tp,\tth),\tth) 
 & =  \Omega (\uphi(\tp,\tth), \tth) + \Omega( \uphi(\tp,\tth) +\tp+ 1, \tth)
 - \ch (\tp+1) \notag \\
&  =  \Omega( \uphi(\tp,\tth) +\tp+ 1, \tth ) - \Omega( \uphi(\tp,\tth) + \tp, \tth ) - 
\Omega (1, \tth) < 0 \label{onega}
 \end{align}
by the sub-additivity property in
 \eqref{subadd} with $ \varphi =   \uphi(\tp,\tth) +\tp$ and $ z = 1 $.  
Then \eqref{onega} and the fact that
$ \partial_\varphi \mathcal{F}_{\tp+1} (\varphi,\tth) > 0 $ 
 imply that $ \uphi(\tp,\tth) < \uphi(\tp+1,\tth) $, for any $\tp\geq 2 $, proving
 the first part of \eqref{uallinfi}. 
 For any $ \tth > 0 $ 
 the monotone sequence
 $\tp\mapsto \uphi(\tp,\tth) $ admits a limit as $\tp\to + \infty $. 
Let us prove that $ \lim_{\tp \to + \infty} \uphi(\tp,\tth) = + \infty $. 
By contradiction suppose that  $  \sup_{\tp \geq 2}\uphi(\tp,\tth) < + \infty $ is finite. 
Then by \eqref{implicitproblem} we  deduce the contradiction
$
\ch  = \lim_{\tp \to \infty} \tp^{-1} 
\Omega (\uphi(\tp,\tth))  
+ \tp^{-1} \Omega (\uphi(\tp,\tth) +\tp)  = 0  $. 
\\[1mm]
{\sc Proof of  item \ref{it3}.}
Recalling \eqref{implicitproblem}, $ \ch = \Omega (1, \tth) $,  \eqref{valzero} and 
that 
$ \Omega (\varphi, \tth) $ is strictly increasing for $ \varphi > 0 $,  
we have that 
$$
\mathcal{F}_2 (0,\tth) = 
\Omega(2,\tth) -  2 \Omega(1,\tth) < 0  \, , \quad 
\mathcal{F}_2(1,\tth) =  \Omega(3,\tth) - \Omega(1,\tth )  > 0   
 \, , \quad \forall \tth > 0 \, , 
$$
and therefore the unique zero $ \uphi(2,\tth)  >  0 $ of 
$\mathcal{F}_2 (\varphi,\tth) = 0 $
satisfies $ \uphi(2,\tth) \in (0,1)$ for any $ \tth > 0 $.

For $\tp\geq 3 $ we already know by Item \ref{it2} 
that 
$ \mathcal{F}_\tp(  \uphi(\tp-1,\tth) ,\tth)  <0  $
and, since $\tfrac14 (\tp-1)^2 \geq 1 $ for any $\tp\geq 3$, we have 
\begin{equation}\label{zerofp}
\mathcal{F}_\tp\big(\tfrac14 (\tp-1)^2,\tth\big) = \frac{\tp-1}{2}\underbrace{\sqrt{\tanh\big(\tfrac{\tth}{4} (\tp-1)^2 \big)}}_{\geq \sqrt{\tanh(\tth)}} + \frac{\tp+1}{2}\underbrace{\sqrt{\tanh\big(\tfrac{\tth}{4} (\tp+1)^2 \big)}}_{> \sqrt{\tanh(\tth)}} -\tp\sqrt{\tanh(\tth)} > 0\, .
\end{equation}
Therefore, since $ \Omega (\varphi, \tth) $ is strictly increasing for $ \varphi > 0 $,  
we deduce that 
$ \uphi(\tp,\tth) \in \big(0,\tfrac14 (\tp-1)^2\big)$ for any $\tp\geq 3 $. 
Furthermore,  
by the  asymptotic expansion 
\begin{equation}\label{expexp}
\sqrt{\tanh (z)} = \Big( 1 - 2 \frac{e^{-2 z}}{1+e^{-2 z}} \Big)^{\frac12}
= 1 -  e^{-2 z} + O (e^{- 4  z}) \qquad \text{as} \qquad z \to + \infty \, , 
\end{equation}  
we deduce that 
$$
\mathcal{F}_2 \Big( \frac14,\tth \Big) 
= \frac12 \big(1-e^{-\frac{\tth}2}+O(e^{-\tth}) \big) + \frac32 \big(1+O(e^{-\tth}) \big) - 2   \big(1+O(e^{-\tth}) \big) < 0\, ,\quad \text{for } \tth \geq \tth_* \textup{ large enough}\, ,
$$
which implies that $ \uphi(2,\tth) > 1/ 4 $ for any $ \tth \geq \tth_* $.  
 \\[1mm]
{\sc Proof of item \ref{it4}.}
For any $\tp\geq 2 $, by the previous item the function $ \tth \mapsto \uphi(\tp,\tth) $ is bounded and 
 $ \tth \uphi(\tp,\tth) \to 0^+ $ as $ \tth  \to 0^+  $. By \eqref{implicitproblem} 
and the Taylor expansion
$ \sqrt{\tanh (z)} = \sqrt{z} \big( 1 - \frac{z^2}{6} + 
\frac{19}{360} z^4 + O(z^6 ) \big)  $ as $ z \to 0^+  $
we first deduce that $\uphi :=\uphi(\tp,\tth) $
\begin{equation}\label{staruphi}
\begin{aligned}
&\uphi \Big(1-\tfrac16 \tth^2 \uphi^2  + \tfrac{19 }{360} \tth^4 \uphi^4 + O(\tth^6 ) \Big) +
\big(\uphi+ \tp\big) \Big(1- \tfrac16 \tth^2 \big(\uphi+ \tp\big)^2  
+ \tfrac{19}{360}  \tth^4 \big(\uphi + \tp\big)^4 + O(\tth^6 )  \Big)  \\
&\qquad =  \tp \Big(1-\tfrac16 \tth^2 + \tfrac{19}{360}\tth^4 + O(\tth^6)\Big) \, ,
\end{aligned}
\end{equation}
and then 
$$
\uphi(\tp,\tth) \big(1+ \cO(\tth^2) \big) 
= \frac{\tp}{12} \tth^2 
\big[ (\uphi(\tp,\tth) + \tp)^2  - 1]  + O(\tth^4)  \, .
$$
This proves that $ \uphi(\tp,\tth) = O( \tth^2) $ as $ \tth \to 0^+ $
and then reinserting the result in \eqref{staruphi} we obtain the sharper Taylor expansion
 in \eqref{expuphinuova1}. 

In order to prove that 
$ \lim_{\tth \to +\infty} \uphi(\tp,\tth) = (\tp-1)^2 / 4 $
it is sufficient to show that 
for any sequence $ \tth_n \to \ + \infty $ there is a subsequence 
$ \tth_{n_k}$ such that  $ \uphi(\tp,\tth_{n_k}) \to (\tp-1)^2 / 4 $.
By item \ref{it3} we have that 
$ 1/ 4  \leq \uphi(\tp,\tth_n) \leq (\tp-1)^2 / 4 < +\infty$ for $ n \geq n_0 $ large enough. 
By compactness, there exists a subsequence $\tth_{n_k}\to +\infty$ 
such that $\uphi(\tp,\tth_{n_k})\to \ell \in [1/4, (\tp-1)^2 / 4 ] $. Thus
by \eqref{implicitproblem} we deduce that 
\begin{equation}\label{infinito}
 \underset{k \to +\infty}{\lim} \mathcal{F}_\tp(\uphi(\tp,\tth_{n_k}),\tth_{n_k}) = \sqrt{\ell} + \sqrt{\ell+\tp} -\tp = 0 \,,
\end{equation}
namely $\ell = \tfrac14 (\tp-1)^2 $, being the  unique solution of \eqref{infinito}.
This proves that  $\uphi(\tp,\tth) \to \tfrac14 (\tp-1)^2 $ as $\tth \to +\infty $.

We now prove the sharper asymptotic expansion in \eqref{expaphip}.
By \eqref{implicitproblem} and \eqref{expexp} we have, as $ \tth \to + \infty $,
\begin{equation}\label{anyp}
0 =\mathcal{F}_\tp (\uphi(\tp,\tth),\tth)
= \underbrace{ \sqrt{\uphi (\tp,\tth)}  + 
\sqrt{\uphi (\tp,\tth) +\tp}
-\tp }_{ = \frac{2\tp}{\tp^2-1} y_\tp (\tth) + O(y_\tp^2 (\tth))} -
\sqrt{\uphi (\tp,\tth)}  e^{-2 \tth \uphi (\tp,\tth)} +
\tp  e^{-2 \tth}  +  O( e ^{- 4 \tth \uphi (\tp,\tth)} )  
+ O(e^{-4 \tth})  
\end{equation}
writing $ \uphi (\tp,\tth) = \frac{(\tp-1)^2}{4} + y_\tp (\tth)$ where $ y_\tp (\tth) \to 0 $
as $ \tth \to + \infty $. 
We first consider the case  $\tp= 2 $.
By Items \ref{it2} and \ref{it3}  
we know  that 
$y_2(\tth) \to 0^+$ as $\tth\to + \infty $. 
Then, by  \eqref{anyp}  
and since $ e^{ - 2 \tth  y_2 (\tth) }  \leq 1$, we deduce that 
\begin{equation}\label{p=2as}
\frac43 y_2 (\tth) = 
\sqrt{  \frac14 + y_2 (\tth) } \, 
e^{ -   \frac{\tth}{2}  }  e^{ - 2 \tth  y_2 (\tth) } 
 + O(y_2^2 (\tth))  + O( e^{-  \tth}) \, .
\end{equation}
Therefore $ | y_2 (\tth) | \leq C e^{- \frac{\tth}{2} } $ and then, 
reinserting this bound in  \eqref{p=2as} we deduce that
$ y_2 (\tth) = \tfrac{3}{8} e^{ -   \frac{\tth}{2}  } + o ( e^{ -   \frac{\tth}{2}  }) $ which is 
the asymptotic 
expansion \eqref{expaphip} for $\tp= 2 $.
The cases $\tp= 3 $ and $\tp\geq 4 $ follow similarly.
By \eqref{anyp}, 
writing $ \uphi (3,\tth) = 1 + y_3 (\tth) $ where $ y_3 (\tth) \to 0^- $, we have 
$$ 
\frac34 y_3 (\tth) = - 3  e^{ - 2 \tth } + \sqrt{1 + y_3 (\tth) } 
e^{- 2 \tth} e^{- 2 \tth y_3 (\tth)}
+ O(y_3^2 (\tth) )   + O(e^{- 4 \tth })
$$
and therefore \eqref{expaphip} for $\tp= 3 $ follows.
Finally for any $\tp\geq 4 $ we deduce by \eqref{anyp} 
and since  $ \frac{(\tp-1)^2}{4} > 1 $ that 
$ \frac{2\tp}{\tp^2-1} y_\tp (\tth) = -\tp e^{ - 2 \tth } + O(y_\tp(\tth)^2) + O(e^{- 4 \tth })$
and therefore \eqref{expaphip} holds for any $\tp\geq 4 $. 
\qed \smallskip
\\[1mm]
\noindent{\bf Proof of Lemma \ref{thm:unpert.coll}}. 
In view of \eqref{omeghino} the operator $ \cL_{\mu,0} $  in \eqref{cLmu} has a non-simple eigenvalue if and only if  the equation 
$\lambda_j^\sigma(\mu,\tth) = \lambda_{j'}^{\sigma'}(\mu,\tth)  $
has a solution $(j,\sigma) \neq (j',\sigma') $ with $j,j'\in \bZ $ and $\sigma,\sigma'=\pm$. This amounts to solve 
$$ 
\omega^\sigma(j+\mu,\tth) = \omega^{\sigma'}(j+\mu+\tp,\tth) 
\quad \text{for some} \quad \sigma, \sigma' \in \{ \pm \} \, , \ 
 j \in \bZ \, ,  \ \tp\in \bN_0 \, ,
 $$ 
namely 
\begin{equation}\label{simdou}
\omega^\sigma(\varphi,\tth) = \omega^{\sigma'}(\varphi+\tp,\tth)
\quad \text{for some} \quad \sigma, \sigma' \in \{ \pm \} \, , \ 
\tp \in \bN_0 \, , \ \varphi-\mu  \in \bZ \, . 
\end{equation}
In view of Lemma \ref{collemma} and Remark \ref{rem:BF}
the equation \eqref{simdou} has a solution 
with 
$ \omega^\sigma(\varphi,\tth) = \omega^{\sigma'}(\varphi+\tp,\tth) \neq 0 $
if and only 
if $ \sigma = - $, $ \sigma' = +  $, $\tp\geq 2 $ and $\uphi(\tp,\tth)-\mu \in \bZ $. 
Then Item \ref{iteml1} follows. For Item \ref{iteml2}, if $\uphi(\tp,\tth)-\mu =: k \in \bZ $ then,  
 in view of \eqref{Kreincollision}, \eqref{omeghino}, \eqref{def:fsigmaj} and since $ \uphi(\tp,\tth) \neq 0 $ then 
  $\im \omega_*^{(\tp)}(\tth)   $ in \eqref{intcollision} 
is a double eigenvalue of $ \cL_{\mu,0}$
with associated eigenvectors  \eqref{autovettorikernel}.
Furthermore 
\eqref{intcollisioncre} follows by \eqref{Kreincollision}
and \eqref{uallinfi} because $ 
\varphi \mapsto \ch \varphi + \Omega (\varphi, \tth ) $ is increasing for 
$ \varphi > 0 $.
Finally by \eqref{negativesolution}, 
the opposite double eigenvalue $-\im \omega_*^{(\tp)} (\tth)  $ 
is obtained for $\varphi=-\uphi(\tp,\tth)-\tp $ and belongs to the spectrum of $\cL_{-\mu,0} $. 
\qed

\section{Proof of the cancellation \eqref{estate2023}}\label{sec:canc}

In this Appendix we  prove   \eqref{estate2023}. 
We first show 
in Proposition \ref{lemB1} 
that $A^{(\tp)}$ is proportional to the   determinant of  a $ \tp\times \tp$ matrix $\mathrm{I\!I\!I} $. Then 
in Lemma \ref{lemB2} we show that 
$\det \mathrm{I\!I\!I} = 0 $ and so $ A^{(\tp)} =0 $. 

\begin{prop}\label{lemB1}
For any integer  $\tp \geq 2$, the term $A^{(\tp)}$ in \eqref{Ap} is 
\begin{equation}\label{Apisadeterminant}
    A^{(\tp)} = \frac1{\prod\limits_{j=1}^{\tp-1} (\tp^3-\tp + j - j^3 )} \det \, \mathrm{I\!I\!I} \, ,
\end{equation}
where 
$\mathrm{I\!I\!I}$ is the $\tp \times \tp$   tridiagonal matrix with integer  entries
\begin{equation}\label{entriesIII}
\mathrm{I\!I\!I}_k^j := \begin{cases} 
2\tp^3-2\tp-2j^3-10 j & \mbox{if } k=j\leq \tp -1\, ,\\
j^3-j-\tp^3+\tp & \mbox{if } k=j+1 \textup{ or }k=j-1 \, ,\\
1 & \mbox{if } j=k=\tp\, ,\\
0 & \mbox{otherwise } ,
\end{cases}
\end{equation} 
with $k$  and $j$ being respectively the row and column indices.
\end{prop}
Before proving Proposition \ref{lemB1}, we show that $ \ker \mathrm{I\!I\!I} \neq \{ 0 \} $. 
A direct calculus shows that:

\begin{lem}\label{lemB2} For any integer $\tp \geq 2$,
the vector 
$ \mathbf{v} :=  \begin{pmatrix} 1,& 2, & \dots, & \tp-1, & 3\tp(\tp-1)^2 \end{pmatrix}^{\top} \in \bR^\tp  
$  
belongs to the kernel of the matrix $\mathrm{I\!I\!I} $ defined in \eqref{entriesIII}. Thus
$\det \, \mathrm{I\!I\!I} = 0$.
\end{lem}



The rest of the Appendix is devoted to the proof of Proposition \ref{lemB1}. 

For any $ 1 \leq q \leq \tp - 1 $, we regard 
any set of $ q $ indices $ \{ j_1, \ldots, j_q  \} $ satisfying 
$ 1 \leq j_1  < \ldots < j_q \leq \tp - 1 $
as 
$ \phi^{-1}(1) = \{ j_1, \ldots, j_q \} $ 
where the function 
$ \phi: \{1,\dots,\tp-1\} \to \{0,1\}$,
$ j \mapsto \phi(j) $, 
attains value $ \phi(j_1) = \ldots = \phi(j_q) = 1  $,  
and $\phi(j)=0$ if and only if  
$j \neq j_1,\dots,j_q $.
In this way  
we  rewrite \eqref{Ap} as
\begin{equation}\label{newsummation}
A^{(\tp)} = \tp+\sum_{q=1}^{\tp-1} \sum_{1 \leq j_1<\dots<j_q \leq \tp-1}\!\! \! \!\! \! \ac{q}(j_1,\dots,j_q) = \!\! \! \!\! \! \sum_{\phi: \{1,\dots,\tp-1\} \to \{0,1\} } 
\!\! \! \!\! \!\!\! \! \!\! \! \ac{| \phi^{-1}(1)|}\big(\phi^{-1}(1)\big) \, ,\quad \textup{with }\ac{0}(\emptyset):=\tp\, ,
\end{equation}
summing on all the 
functions 
$\phi: \{1,\dots,\tp-1\} \to \{0,1\}$,  
arranging  $\phi^{-1}(1) = \{ j_1 < \ldots < j_q \} $ in ascending order, and 
denoting $ q := |\phi^{-1}(1)| $  its  cardinality. If $ \phi \equiv 0 $
we have $ \phi^{-1} (1) = \emptyset $. 

In the sequel we denote 
$\delta_{j,k}:= 1 $ if $ j=k $
and  $ \delta_{j,k}:= 0 $ otherwise, i.e.\
the Kronecker delta.

\begin{lem}
Let $\phi: \{1,\dots,\tp-1\} \to \{0,1\}$.
Then the terms in \eqref{cancelingasympcoeff} are 
\begin{equation}\label{banale}
\ac{|\phi^{-1}(1)|}\big(\phi^{-1}(1)\big) =\frac1{\prod\limits_{j=1}^{\tp-1} (\tp^3-\tp + j - j^3 )}  \det \widetilde{M}_\phi 
\end{equation}
where $\widetilde{M}_\phi$ is the $\tp\times \tp$  matrix   with  entries
\begin{equation}\label{matrixtildeMphi}
[\widetilde M_\phi]_k^j := \begin{cases}
-12 j \min\{j,k\}  &\textup{if } \phi(j)=1 \,, 1 \leq j \leq \tp - 1 \, ,  \ \forall \,  
k = 1, \ldots  \tp  \, ,  
\\
(\tp^3-\tp+j-j^3) \delta_{j,k}  &\textup{if } \phi(j)=0\,, 1 \leq j \leq \tp - 1 \, ,   \ \forall \,  
k = 1, \ldots  \tp  \, ,   \\
k &\textup{if }j=\tp\, , \ \forall \,  
k = 1, \ldots  \tp  \,  .
\end{cases}
\end{equation}
\end{lem}

\begin{proof} 
If $\phi \equiv 0$ the matrix
$ \widetilde M_\phi $ in \eqref{matrixtildeMphi} is triangular with diagonal  entries
$ [\widetilde M_\phi]_j^j = \tp^3-\tp+j-j^3  $ for any $ 1 \leq j \leq \tp - 1 $ and 
$ [M_\phi]_\tp^\tp = \tp $. Thus  \eqref{banale} holds, since 
$ \ac{0}(\emptyset):=\tp $, cfr. \eqref{newsummation}.
Thus in the sequel $ \phi \not \equiv 0 $. 
\\[1mm]
{\bf Step 1}. 
{\it The determinant of 
the $\tp\times \tp$ matrix 
$M_\phi$  with  entries
\begin{equation}\label{matrixMphi}
[M_\phi]_k^j  := \begin{cases}
\min\{j,k\} &\textup{if } \phi(j)=1 \, , \, j \in \{ 1, \ldots, \tp -1 \} \, ,  \textup{ or }j=\tp\,,\\
\delta_{j,k}  &\textup{if } \phi(j)=0 \, , \, j \in \{ 1, \ldots, \tp -1 \} \, , 
\end{cases}
\end{equation}
is equal to} 
\begin{equation}\label{proddifffin}
\det M_\phi = 
j_1(j_2-j_1)\dots(j_q-j_{q-1})(\tp -j_q) \, , 
\end{equation}  
{\it where} $  \{ 1 \leq j_1 < \ldots < j_q 
\leq \tp - 1 \} = \phi^{-1}(1) $.  
 We write  
\begin{equation}\label{toenlarge}
j_1(j_2-j_1)\dots(j_q-j_{q-1})(\tp -j_q) = \det \begingroup 
\setlength\arraycolsep{2pt} 
\underbrace{{\footnotesize \begin{pmatrix}  j_1  & 0 & 0  & \dots & 0 \\
j_1 & j_2-j_1 & 0 & \dots & 0 \\
\vdots & \vdots & \ddots & \ddots & \vdots \\
j_1 & j_2-j_1 & \dots & j_q - j_{q-1} & 0  \\
j_1 & j_2-j_1 & \dots & j_q - j_{q-1} & \tp-j_q \end{pmatrix}}}_{=: Q_\phi } \endgroup = \det \underbrace{\begingroup 
\setlength\arraycolsep{2pt}{\footnotesize\begin{pmatrix}  j_1  & j_1 & j_1  & \dots & j_1 \\
j_1 & j_2 & j_2 & \dots & j_2 \\
\vdots & \vdots & \ddots & \ddots & \vdots \\
j_1 & j_2 & \dots & j_q  & j_q  \\
j_1 & j_2 & \dots & j_q  & \tp \end{pmatrix}}\endgroup }_{=: N_\phi} 
\end{equation}
where the  $j$-th column of the 
matrix $N_\phi$ 
is the sum of the first $j$ columns of the  matrix $ Q_\phi$.  
It results 
\begin{equation}\label{detuguali}
\det N_\phi = \det \widehat M
\end{equation}
where 
$\widehat M$ is the $\tp\times \tp $ matrix  
defined as follows: 
for any  $j=1,\dots,\tp$ 
\begin{itemize}
    \item for any $ 1 \leq j \leq \tp - 1 $ if $\phi(j)=0$, 
    i.e.\ $j \neq j_1,\dots,j_q $, 
    then the $j$-th column of $\widehat M$ is the vector $e_j$ of the standard basis of $\bR^\tp$;
    \item for any $ 1 \leq j \leq \tp - 1 $ if  $\phi(j)=1$, i.e.\ $j = j_\iota$ with $\iota=1,\dots,q$, then the $j$-th column of $\widehat M$ has its $j_1, \ldots, j_q$-th rows equal to the  $1, \ldots, q$-th rows of the $i$-th column of $N_\phi$, i.e.\ $ [\widehat M]^{j_\iota}_{j_\kappa} = [N_\phi]^\iota_{\kappa} $ for any $ \iota,\kappa \in \{1, \ldots, q\} $;
    \item the remaining elements $\widehat M_k^j$ are  equal to $\min\{j,k\}$, in particular the $ \tp$-th column is
    $ \widehat M_k^\tp = k $.
\end{itemize}
The identity \eqref{detuguali}
follows expanding the determinant by the above bullets. 
We actually claim that 
\begin{equation}\label{equalse}
\widehat M = M_\phi \, . 
\end{equation} 
Indeed the $\tp$-th column of $\widehat M $ and $  M_\phi$ are identical. If 
for $ 1 \leq j \leq \tp - 1 $, 
$ \phi (j) = 0 $, then 
by the first bullet, the $ j $-th column of 
$ \widehat M $ and $ M_\phi $ coincide. 
We now check that also the 
$j$-th columns of $\widehat M $ and $  M_\phi$ coincide, if $ j = j_\iota $, $ \iota \in \{1, \ldots, q \}  $.
If $ k \notin \{ j_1, \ldots, j_q \} $
by the third bullet 
$ \widehat M_k^j  = \min\{j,k\} = [M_\phi]_k^j $ as in \eqref{matrixMphi}.
If 
$ k = j_{\kappa} $ for some $ \kappa \in \{1,\dots,q\}$, then,  
by the second bullet 
(and since $j_1,\dots,j_q$ are arranged in ascending order)
$ \widehat M_k^j  =
\widehat M_{j_\kappa}^{j_\iota}  = [N_\phi]_{\kappa}^\iota \stackrel{\eqref{toenlarge}}{=} j_{\min\{\kappa,\iota\}} = \min\{j_{\iota},j_{\kappa}\} = \min\{j,k\} = [M_\phi]_k^j $. 
 This proves \eqref{equalse}. Finally \eqref{proddifffin} follows from  
 \eqref{equalse}, \eqref{detuguali} and \eqref{toenlarge}. 
\\[1mm]
{\bf Step 2}. 
The matrix $\widetilde{M}_\phi$ 
in \eqref{matrixtildeMphi} is obtained   
by multiplying the $j$-th column, $ 1 \leq j \leq \tp - 1 $,  
of the matrix $M_\phi$ in
\eqref{matrixMphi} by the factor  $-12 j$, if $\phi(j)=1$, or  $(\tp^3-\tp+j-j^3) $ if $\phi(j)=0$. Thus, by computing the determinant  of  $\widetilde{M}_\phi$, using 
\eqref{proddifffin} and recalling \eqref{cancelingasympcoeff}, we deduce \eqref{banale}.  
\end{proof}

We now write $A^{(\tp)}$ as the determinant of a matrix. 

\begin{lem} For any integer $ \tp\geq 2$, the term $A^{(\tp)}$ in \eqref{Ap} is 
\begin{equation}\label{partialresult}
    A^{(\tp)} =  \frac1{\prod\limits_{j=1}^{\tp-1} (\tp^3-\tp + j - j^3 )}  \det \mathtt{M} 
\end{equation}
where 
$\mathtt{M}$ is the $\tp \times \tp$     matrix with integer  entries
\begin{equation}\label{matrixttM}
\mathtt{M}_k^j := \begin{cases}
-12 j \min\{j,k\}  &\textup{if } k\neq j \leq \tp -1\,,\\
\tp^3-\tp-j^3-12 j^2 +j   &\textup{if } k = j \leq \tp -1 \,, \\
k &\textup{if }j=\tp \, , \, \forall k =1, \ldots, \tp  \, .
\end{cases}
\end{equation}
\end{lem}

\begin{proof}
By \eqref{newsummation} and 
\eqref{banale} we  have
\begin{equation}\label{granfinale?}
A^{(\tp)} =   \tfrac1{\prod\limits_{j=1}^{\tp-1} (\tp^3-\tp + j - j^3 )}  \sum_{\phi: \{1,\dots,\tp-1\} \to \{0,1\} }  \det \widetilde{M}_\phi \, ,
\end{equation}
and the matrix $ \widetilde{M}_\phi $ in \eqref{matrixtildeMphi} can be written as 
\begin{equation}\label{detmezzo}
\widetilde{M}_\phi = \begin{pmatrix} col_{1;\phi(1)} \ \big| & col_{2;\phi(2)} 
\ \big| & \dots \ \big| & col_{\tp-1;\phi(\tp-1)}
\ \big| & \widetilde{col}_1 \end{pmatrix} 
\end{equation}
with  column vectors
\begin{equation}\label{colmve}
col_{j;1} := \big( -12 j \min\{j,k\} \big)_{k=1,\dots,\tp}\,,\quad col_{j;0} := \big( (\tp^3-\tp + j - j^3 ) \delta_{j,k} \big)_{k=1,\dots,\tp}\,, \quad \widetilde{col}_{1} := \big( 1, \dots, \tp \big)^\top\, .
\end{equation}
By  \eqref{granfinale?}, \eqref{detmezzo} and 
the multilinearity of the determinant  
we deduce
$$
    A^{(\tp)} =  \tfrac1{\prod\limits_{j=1}^{\tp-1} (\tp^3-\tp + j - j^3 )}  \det \underbrace{\begin{pmatrix} col_{1;0} + col_{1;1}  \ \big| & col_{2;0} + col_{2;1} \ \big| & \dots \ \big| & col_{\tp-1;0} + col_{\tp-1;1} \ \big|  & \widetilde{col}_1 \end{pmatrix}}_{= \mathtt{M}} 
$$
where, by  inspection, cfr.\eqref{colmve}, 
the entries of the matrix $\mathtt{M} $ are those given in \eqref{matrixttM}.
\end{proof}

\noindent{\it Proof of Proposition \ref{lemB1}.} 
The determinant of $\mathtt{M} $ is equal to 
\begin{equation}\label{determinatalid}
    \det \mathtt{M} = \det \mathtt{H} 
\end{equation}
where  the matrix 
$\mathtt{H}$ is   obtained by $\mathtt{M} $ in \eqref{matrixttM}  
substituting each row,  except the first one, with the difference between itself and the previous row, 
in formulas 
 $  \mathtt{H}_1^j := \mathtt{M}_1^j $,  $ \mathtt{H}_k^j  := \mathtt{M}_k^j  - \mathtt{M}_{k-1}^j $, $  \forall $, $ k=2,\dots,\tp $, 
 $  \forall \, j=1,\dots,\tp $. 
In view of \eqref{matrixttM} 
its entries are  
\begin{equation}\label{matrixttH}
\mathtt{H}_k^j  :=
    \begin{cases}
-12 j   &\textup{if } k < j \leq \tp -1\,,\\
\tp^3-\tp-j^3-11 j   &\textup{if } k = j \leq \tp -1 \,, \\
- (\tp^3-\tp-j^3 +j)  &\textup{if } k = j+1\,,\\
0  &\textup{if } j+2 \leq k \leq \tp \,,\\
1 &\textup{if }j=\tp \, , \, \forall k =1, \ldots, \tp  \, .
\end{cases}
\end{equation}
Furthermore
\begin{equation}\label{determinatalid1}
    \det \mathtt{H} = 
     \det \mathrm{I\!I\!I}
\end{equation}
where  the matrix 
$ \mathrm{I\!I\!I} $  is  obtained from $\mathtt{H}$ by substituting each row,  except the last one, with the difference between itself and the next row, in formulas 
$  \mathrm{I\!I\!I}_\tp^j := \mathtt{H}_\tp^j $, $ \mathrm{I\!I\!I}_k^j := \mathtt{H}_k^j  - \mathtt{H}_{k+1}^j $, $ \forall \,  k=1,\dots,\tp-1 $, $  \forall j=1,\dots,\tp $.  
In view of \eqref{matrixttH},  the matrix $\mathrm{I\!I\!I} $ is the  tridiagonal one in \eqref{entriesIII}. 

In conclusion  \eqref{partialresult},  \eqref{determinatalid} and
\eqref{determinatalid1} imply  \eqref{Apisadeterminant}. \qed

 \begin{footnotesize} 
 	
\end{footnotesize}

\noindent 
\footnotesize{Work  supported by PRIN 2020 (2020XB3EFL001) “Hamiltonian and dispersive PDEs”, PRIN 2022 (2022HSSYPN)   “Turbulent Effects vs Stability in Equations from Oceanography". A. Maspero is supported by the European Union  ERC CONSOLIDATOR GRANT 2023 GUnDHam, Project Number: 101124921.
P.\ Ventura is supported by the ERC STARTING GRANT 2021 ``Hamiltonian Dynamics, Normal Forms and Water Waves" (HamDyWWa), Project Number: 101039762.
 Views and opinions expressed are however those of the authors only and do not necessarily reflect those of the European Union or the European Research Council. Neither the European Union nor the granting authority can be held responsible for them.
}
\normalsize

\medskip

\noindent 
{\it Massimiliano Berti}, SISSA, Via Bonomea 265, 34136, Trieste, Italy, \texttt{berti@sissa.it},  
\\[1mm] 
{\it Livia Corsi},  Univ. di Roma 3, 
Dip. di Matematica e Fisica, 
Roma, 00146, Italy
\texttt{livia.corsi@uniroma3.it},
\\[1mm]
{   \it Alberto Maspero}, SISSA, Via Bonomea 265, 34136, Trieste, Italy,  \texttt{amaspero@sissa.it}, 
 \\[1mm] 
{   \it Paolo Ventura}, Università degli Studi di Milano, Via Saldini 50, 
20133 Milano, \texttt{paolo.ventura@unimi.it}.
\end{document}